\documentclass[11pt,twoside]{preprint}
\usepackage[utf8]{inputenc}
\usepackage[full]{textcomp}
\usepackage[T1]{fontenc}
\usepackage{lmodern}

\usepackage[usenames,dvipsnames]{color}
\usepackage[colorinlistoftodos,prependcaption,color=yellow,textsize=tiny]{todonotes}

\usepackage{amsthm}
\usepackage{amssymb}
\usepackage{amsfonts}
\usepackage{mydefenv}
\usepackage[english]{babel}
\usepackage{hyperref}
\usepackage{url}
\usepackage{breakurl}
\usepackage{ulem}

\usepackage{tikz}
\usepackage{pgfplots}

\usepackage{microtype}
\usepackage{mathtools}
\usepackage{comment}
\usepackage[a4paper,innermargin=1.4in, 
outermargin=1.4in,
bottom=1.5in,marginparwidth=1in,marginparsep=3mm]{geometry}
\usepackage{graphicx}
\usepackage{enumerate}
\usepackage{enumitem}
\usepackage{kantlipsum}
\usepackage{subfig}
\usepackage{mathrsfs}
\usepackage{latexsym}
\usepackage{verbatim}
\usepackage{xcolor}
\usetikzlibrary{calc,intersections,through,backgrounds}
\usetikzlibrary{decorations.pathreplacing}
\usetikzlibrary{shapes}

\usepackage{dsfont}

\usetikzlibrary{decorations.markings}
\usetikzlibrary{external}

\tikzset{->-/.style={decoration={
  markings,
  mark=at position #1 with {\arrow{>}}},
  postaction={decorate}}}


\colorlet{darkblue}{blue!90!black}
\colorlet{darkred}{red!90!black}
\colorlet{darkgreen}{green!70!black}

\def\eps{\varepsilon}

\newcommand{\hmin}{\tilde{h}}
\newcommand{\taumin}{\tilde{\tau}}

\DeclareMathOperator*{\esssup}{ess\,sup}

\newcommand{\cS}{\mathcal{S}}

\newcommand{\gridlim}{Z}

\newcommand{\h}{h}

\newcommand{\Ipcx}{I_h^{\mathrm{pcx}}\,}
\newcommand{\Ipcy}{I_h^{\mathrm{pcy}}\,}
\newcommand{\Iplx}{I_h^{\mathrm{plx}}\,}

\newcommand{\Iply}{I_h^{\mathrm{ply}}\,}
\newcommand{\floor}[1]{\left\lfloor #1 \right\rfloor}

\newcommand{\pc}[1]{{#1_h^{\text{pcx}}}}
\newcommand{\pcx}[1]{{#1^{\text{pcx}}}}
\newcommand{\plx}[1]{{#1^{\text{plx}}}}
\newcommand{\plt}[1]{{#1^{\text{plt}}}}
\newcommand{\pctl}[1]{{#1^{\text{pct-}}}}
\newcommand{\pctr}[1]{{#1^{\text{pct+}}}}

\newcommand{\lc}[1]{{#1^{\text{plt,pcx}}}}
\newcommand{\clc}[1]{{#1^{\text{pct-pcx}}}}
\newcommand{\crc}[1]{{#1^{\text{pct+pcx}}}}

\newcommand{\sprod}[2]{\left\langle#1,#2\right\rangle}

\newcommand{\Tr}{\mathrm{Tr}}

\newcommand{\cAOmT}{\cA}

\renewcommand{\bullet}{\cdot}

\renewcommand{\d}{\mathrm{d}}

\theoremstyle{definition}

\numberwithin{equation}{section}

\title{Stochastic partial differential equations arising in self-organized criticality}

\author{\v{L}ubom\'ir Ba\v{n}as\thanks{Faculty of Mathematics, Bielefeld University,
    Universit\"atsstr.~25, 33615 Bielefeld, Germany (banas@math.uni-bielefeld.de)}, Benjamin Gess\thanks{Max Planck Institute for Mathematics in the
    Sciences, Inselstr.~22, 04103 Leipzig, Germany and Faculty of
    Mathematics, Bielefeld University, Universit\"atsstr.~25, 33615
    Bielefeld, Germany (b.gess@mis.mpg.de)}, Marius Neuß\thanks{Max Planck Institute for Mathematics in the
    Sciences, Inselstr.~22, 04103 Leipzig, Germany (neuss@mis.mpg.de)}}

\begin{document}

\maketitle

\begin{abstract}
We study scaling limits of the weakly driven Zhang and the Bak-Tang-Wiesenfeld
(BTW) model for self-organized criticality. We show that
the weakly driven Zhang model converges to a stochastic partial differential equation (PDE) with
singular-degenerate diffusion. In addition, the deterministic BTW model is
shown to converge to a singular-degenerate PDE. Alternatively, the proof
of the scaling limit can be understood as a convergence proof of a
finite-difference discretization for singular-degenerate stochastic PDEs. 
This extends recent work on finite difference approximation of
(deterministic) quasilinear diffusion equations to discontinuous diffusion
coefficients and stochastic PDEs. 
In addition, we perform numerical simulations illustrating key features of the considered models 
and the convergence to stochastic PDEs in spatial dimension $d=1,2$.


\end{abstract}

\section{Introduction}
\label{2sec:introduction}

The concept of self-organized criticality (SOC) was introduced by Bak, Tang
and Wiesenfeld in the seminal articles \cite{BTW,BTW88} by means of the
paradigmatic ``sandpile model''. This particle model serves as a guiding
example of a non-equilibrium system reaching criticality without apparent
external tuning of model parameters, thus exemplifying the concept of
SOC. The field of SOC has since received strong interest in physics and
alternative particle models like the Zhang model have been introduced, see
for example \cite{Watkins-Pruessner,Arenas,
  Dickman2000,OFC,Zhang,ricepile,Manna_1991,Bak,Carlson-modelling,Dhar-Ramaswamy}.
Despite their apparent simplicity, the analysis of the dynamics of these
sandpile models is challenging and the microscopic models contain several
arbitrary degrees of freedom, such as the structure of the underlying grid
and its size. Consequently, already shortly after the introduction of the
concept of SOC, continuum models mimicking the discrete systems have been
introduced in the physics literature, see \eg
\cite{Diaz-G,Diaz-G-PhysRevA,Cafiero1995,Arenas,Andrade-Pires,Grinstein-Lee-Sachdev,Dickman-Vespignani-Zapperi,Carlson-Annals-Prob,Carlson-Montakhab,Hergarten-Neugebauer,BCRE,Bantay-Janosi,Hwa-Kardar,Prigozhin94,Garrido-Lebowitz-Maes-Spohn}.
Informally, these are thought to correspond to continuum limits
of the discrete systems \eqref{1eq:147-new}, \eqref{1eq:2-new}, leading to
singular-degenerate stochastic partial differential equations (SPDEs) of the type
\begin{align}\label{1eq:3new}
  \begin{split}
    \d X(t) &= \Delta\tilde\phi(X(t))\d t + B(X(t))\d W(t)  \quad \text{on }(0,T]\times(0,1).\\
  \end{split}
\end{align}
Motivated by this development, a large number of mathematical contributions have been devoted to the analysis of such (stochastic) PDEs, see, for example,
\cite{Barbu17, Barbu-Cellular,BRR-Rd,BBDPrR-SOC-via-SPDE,Barbu-SOC-convergence,BDPrR-SPMEbook,Barbu-Bogachev,BDPrR-SPMEandSOC,BR-SPMEandSOCdims,BDPrR-existence-strong,BDPrR-existence-nonneg,BDPrR-FTE,BM,DPrRRW,DPrR-weaksolns,RRW,RWZ-reflection,G-R,Gess-FTE}. 
However, despite the ongoing interest in continuum models for SOC, so far their rigorous justification in terms of scaling limits of the discrete systems remained an open problem. As a consequence, also the question in which scaling regimes these continuum models apply remained unanswered.

The goal of the present work is to address the aforementioned open problems: We identify scaling regimes for the model parameters $D$ and $\sigma^2$ in \eqref{1eq:2-new} for which the rescaled weakly driven Zhang model converges to the solution of a singular stochastic PDE of the type \eqref{1eq:3new}. The proof of this fact is challenging due to the interplay of the irregularity of the random forcing, which converges to space-time white noise, with the degeneracy of the diffusion. For this reason, a proof based on distribution spaces taking the place of more classic function spaces has to be developed. In addition, we prove that the pure diffusion part of the BTW model converges to a deterministic, singular-degenerate PDE. In this case, an additional difficulty arises due to the non-reflexivity of the corresponding energy space. This is overcome in the present work by devising an SVI approach for this setting. 

The ``sandpile model'', also called BTW model in the following, specifies
the evolution of the height $(X^{n,j})_{n=0,\dots,N; j\in \Lambda}\subset
\R$ of a number of particles on a $d$-dimensional grid $j\in\Lambda = \{0,
\dots, Z\}^d$ of length $Z\in\N_{\geq 2}$ in discrete time $n\in\{0, 1,\dots, N\}$,
$N\in\N$. The dynamics of the height are induced by a competition of energy
input and dissipation: As long as a critical height $K$ is not exceeded,
that is $\|X^{n,\cdot}\|_\infty \leq K$, energy is added to the system. If
the critical height $K$ is exceeded, the dissipation mechanism, a so-called
toppling event, is invoked redistributing energy throughout the
system. Without loss of generality, $K$ can be set to $1$ by a scaling argument. This leads to the definition of the dynamics via
\begin{equation}\label{1eq:147-new}
  X^{n+1,j} = X^{n,j} + D \sum_{j' \sim j} \left(\tilde\phi(X^{n,j'}) -
    \tilde\phi(X^{n,j})\right)+ \Ind{\{\|X^{n,\cdot}\|_\infty \leq 1\}} \xi^{n,j},
\end{equation}
for $j\in \Lambda' = \{1,\dots,Z-1\}^d$, where the sum reaches over direct
neighboring cells $j'$ of $j$, $D\in \big(0,\frac1{2d}\big]$, $\tilde\phi = \tilde\phi_1$ with
\begin{equation}\label{1eq:148-new}
  \tilde\phi_1(x) = \Ind{(1,\infty)}(x) - \Ind{(-\infty, -1)}(x),
\end{equation}
and $\xi^{n,j} = \tilde\mu \delta_{j, s_n}$, $\tilde\mu > 0$, are random variables with $s_n \sim
\mathrm{Uni}(\Lambda')$ independently identically distributed. By fixing
the height to zero on $\Lambda\setminus \Lambda'$, energy can be dissipated
via toppling.

Subsequently, several modifications of the original BTW model have been introduced, for example, the Zhang model for which the diffusion $\tilde\phi = \tilde\phi_2$ takes the form
\begin{equation}\label{1eq:149-new}
  \tilde\phi_2(x) =  |x| \left(\Ind{(1,\infty)}(x) -
    \Ind{(-\infty, -1)}(x)\right), 
\end{equation}
among many more, see, for example \cite{OFC,Feder-Feder-Stickslip}. Both
models share the properties to be slowly driven by energy input with rapid
relaxation, that have been identified in \cite[Section
III.1]{Dickman2000} as characteristics of systems displaying SOC.

Following \cite[Section III.2]{Dickman2000}, we introduce the \textit{weakly} driven BTW/Zhang models
\begin{equation}\label{1eq:2-new}
  X^{n+1,j} = X^{n,j} + D \sum_{j' \sim j} \left(\tilde\phi(X^{n,j'}) -
    \tilde\phi(X^{n,j})\right) + \xi^{n,j},
\end{equation}
where now $(\xi^{n,j})_{n=0,\dots,N; j\in\Lambda'}$ are independent random
variables identically distributed with positive mean $\E\xi^{n,j} = \tilde\mu > 0$ and finite variance $\mathrm{Var}\left(\xi^{n,j}\right) = \sigma^2<\infty$.
We call these models weakly driven,
since the totally asymmetric noise in \eqref{1eq:147-new}, which only takes non-negative values, is replaced by a weakly asymmetric noise which still
has positive mean, but may also take negative values.

Starting from the zero initial condition, energy builds up in the system
until toppling events appear. These can induce chain reactions, that is, cause subsequent topplings. A series of
$m$ toppling events is called an avalanche of size $m$.
A key observation in SOC is that the systems described
above reach criticality, in the sense that the statistics of the sizes of
avalanches show a power law behavior (see e.\,g.\,\cite{BTW88}). The numerical simulations presented in
Section~\ref{sec_num_power} below demonstrate
that this behavior is still present in the weakly asymmetric case \eqref{1eq:2-new}.

More precisely, with the  space-time rescaling 
$h=\frac1{Z}, \tau = \frac{T}{N}$, model \eqref{1eq:2-new} becomes
\begin{equation}\label{1eq:150-new}
  \begin{split}
  X_{h,\tau}^{n+1,\cdot} &= X_{h,\tau}^{n,\cdot} + \tau \Delta_{h}\tilde\phi\left(X_{h,\tau}^{n,\cdot}\right) +
  \tilde\xi_{h,\tau}^{n,\cdot},\quad\text{for }n=0,\dots, N-1,\\
  X_{h,\tau}^0 &= x_{h}^0,
\end{split}
\end{equation}
with zero boundary condition, where the random variables $(\tilde\xi_{h,\tau}^{n,j})_{j\in \Lambda'}$ are
assumed to be $\R$-valued, independent identically distributed with mean
$\mu\tau$, variance $\frac{\tau}{h^d}$ and finite sixth moments. This scaling means that the constants $D$,
$\tilde\mu$ and $\sigma^2$ in \eqref{1eq:2-new} are replaced by
$\frac{\tau}{h^2} = \frac{T Z^2}{N}$, 
$\tau\mu = \frac{\mu T}{N}$ and $\frac{\tau}{h^d}=\frac{T Z^d}{N}$,
respectively.
Furthermore, $\Delta_h$ denotes the discrete
Dirichlet Laplacian. Recall that on the level of the discrete model, $D$
encodes the share of the quantity on the critical site being redistributed
during one toppling event, and $\mu$ encodes the average quantity added in
each time step.

In the following, we identify grid functions on the lattice $h\Lambda'$
with their prolongations chosen to be piecewise affine in time and piecewise constant in
space.

\begin{thm}[See Theorem \ref{2main-thm} below] \label{1thm-Zhang} Let $d=1$. Consider the rescaled weakly
  driven Zhang model \eqref{1eq:150-new} with diffusion nonlinearity $\phi_2$
  and initial condition $x^0_{h} \in \R^{Z-1}$. Assume that
  that the (strong) CFL condition
  \begin{align}
    \frac{\tau}{h^2} \to 0\quad \text{for }\tau, h\to 0\label{1eq:CFL}
  \end{align}
   is satisfied and that $x_{h}^0 \to
  x_0\in L^2([0,1])$ for $h\to 0$. Then,
  \begin{displaymath}
    X_{h,\tau} \to X \quad \text{in } L^2([0,T]; L^2([0,1]) \text{
      and } L^\infty([0,T]; \Hd)
  \end{displaymath}
 for $\tau, h\to 0$ in distribution, where $X$ is the unique weak solution to 
\begin{align}\label{1eq:Zhang}
  \begin{split}
    \d X(t) &= \Delta\tilde\phi_2(X(t))\,\d t + \mu\d t + \d W(t)  \quad \text{on }(0,T]\times(0,1),\\
    X(0) &= x_0,
  \end{split}
\end{align}
with initial state $x_0\in L^2([0,1])$ and $W$ is a cylindrical
$\mathrm{Id}$-Wiener process on $L^2([0,1])$. 
\end{thm}

While the Zhang model and the BTW model share the difficulty of a discontinuous and locally degenerate diffusion, the diffusivity of the BTW model also degenerates for large values. This complicates the mathematical analysis. In this case, the convergence of the purely diffusive part can still be obtained. 
\begin{thm}[See Theorem \ref{2BTW-thm} below] 
  Consider the BTW model without external forcing and let $d=1$. Assume that \eqref{1eq:CFL} is
  satisfied and that $x_h^0\to x_0$ in $L^2([0,1])$ for $h\to 0$. Then,
  \begin{displaymath}
    X_{h,\tau} \to X \quad \text{weakly* in } L^\infty([0,T]; \Hd)
  \end{displaymath}
  for $\tau,h\to 0$, where $X$ is the unique EVI solution to 
\begin{align}\label{1eq:BTWeq}
  \begin{split}
    \d X(t) &= \Delta\tilde\phi_1(X(t))\,\d t  \quad \text{on }(0,T]\times(0,1),\\
    X(0) &= x_0,
  \end{split}
\end{align}
with initial state $x_0\in L^2([0,1])$.
\end{thm}

The above two main results can be reinterpreted in terms of the convergence of
time explicit finite difference schemes for singular-degenerate (stochastic) PDE. From
this viewpoint, the results presented in this work partially extend those
of the recent contribution \cite{delTesoI}. More precisely, the results of
\cite{delTesoI} in particular imply the convergence of explicit finite difference schemes of singular-degenerate PDE of the type \eqref{1eq:3new} with $B\equiv 0$
and Lipschitz continuous nonlinearities $\tilde\phi$. Due to
the discontinuous nature of the diffusion coefficients  \eqref{1eq:148-new}, \eqref{1eq:149-new}
the results of \cite{delTesoI} are not applicable in the present setting. In addition, the arguments developed
in \cite{Evje-Karlsen-explicit,delTesoI} rely on compactness arguments in
$L^1$, and, therefore, require at least $L^1$ regularity for the
forcing. In the context of stochastic PDE considered here, these methods
are not applicable, since (1+1)-dimensional space-time white noise has spatial regularity
only $\cC^{-\frac1{2}-}$. Therefore, arguments applicable in spaces of
distributions have to be developed, which is done in the present work by
establishing an $H^{-1}$-based approach rather than working in $L^1$. 

In the language of numerical analysis, the first main result obtained in this work and,
analogously, the second one can be re-formulated as follows.
\begin{Cor}
  Let $d=1$. Consider a rectangular grid covering $[0,T]\times[0,1]$ with
  grid size $h$ and time step size $\tau$ and assume that \eqref{1eq:CFL}
  is satisfied. Then, the solutions of the explicit finite difference
  discretization of the stochastic PDE \eqref{1eq:Zhang}, i.e., \eqref{1eq:150-new} with $\tilde\phi = \tilde\phi_2$,
  converge for $\tau,h\to 0$ to the unique weak solution of \eqref{1eq:Zhang} in the
  same sense as in Theorem \ref{1thm-Zhang}.
\end{Cor}

\begin{Rem}
  The proof in \cite{delTesoI} relies on a
  comparison principle on the level of the discrete scheme, which is closely
connected to the CFL-type condition
\begin{displaymath}
  \frac{\tau}{h^2} \leq \frac1{2d\,\mathrm{Lip}_{\tilde\phi}},
\end{displaymath}
where $\mathrm{Lip}_{\tilde\phi}$ is the Lipschitz constant of the nonlinearity
${\tilde\phi}$ (see \cite[p. 2272]{delTesoI}). Due to the discontinuous nature of the nonlinearities in \eqref{1eq:148-new} and \eqref{1eq:149-new} such a condition can only be satisfied in a
limiting sense, thus motivating \eqref{1eq:CFL}. 
\end{Rem}

In Section~\ref{sec_num}, numerical simulations are included going beyond the above setup for which rigorous results are obtained, by investigating rates of convergence, convergence in stronger topologies, and relaxed assumptions. For example, numerical simulations for $d=2$ are included indicating the existence of a
non-trivial limit for the Zhang and BTW model \eqref{1eq:150-new}. Notably, in higher spatial dimension $d\geq 2$, the regularity of space-time white noise and, thus, of solutions to \eqref{1eq:Zhang} becomes worse and renormalization may become necessary. 

We conclude the introduction with some brief comments on the method of
proof: The proof of convergence is based on a compactness argument and thus fundamentally relies on establishing uniform energy estimates on the discrete solutions. In order to establish these estimates, we introduce discrete analogs of the continuous $H^{-1}$ norm, where the continuous Laplacian $\Delta$ is replaced by its discrete counterpart $\Delta_h$. This allows to reproduce the known continuous energy estimates on the 
discrete level, but leads to discretization dependent norms and spaces. 
The discrete stability estimates are then carefully transferred to spatially continuous interpolants 
of the discrete solutions.
The derived estimates provide the basis for the subsequent compactness arguments based on the ``weak convergence approach'' inspired by \cite{Flandoli-Gatarek}. The identification of the resulting limit as a probabilistically weak solution to the stochastic PDE driven by space-time white noise is challenging, due to its multivalued nature, but can be resolved by monotonicity techniques, as long as the corresponding energy space is reflexive. This finishes the proof for the Zhang model. 
The monotonicity approach fails in the case of the BTW nonlinearity, since weak convergence can only be obtained in
weaker topologies due to the non-reflexivity of the energy space. At this point the solution concept of
stochastic variational inequalities (SVI solutions) proves crucial as it does not rely on reflexivity. This approach allows to conclude the convergence of the discrete BTW model in the deterministic setting.

The structure of this paper is as follows: We first give an overview on
the mathematical and physical literature in Section \ref{2sec:literature} and
introduce some general notational conventions in Section
\ref{2sec:notation}. The continuum limits of the weakly driven Zhang and BTW model
are then treated in Section \ref{2sec:setting-main-result} and Section
\ref{2sec:det-analysis}, respectively.
Numerical experiments are included in Section~\ref{sec_num}.

\subsection{Mathematical literature}\label{2sec:literature}
We first mention previous attempts to approach SOC in a continuous
setting. Related to the scaling limit approach, one strategy consists in
considering cellular automata resulting from a reformulation and
modification of the original sandpile models, as proceeded in
\cite{Carlson-modelling}, in order to obtain a problem which is more
accessible for analysis. For one of these models, a hydrodynamic limit PDE
has been rigorously obtained in \cite{Carlson-Annals-Prob}. Scaling limits
of the SOC models introduced above have been asserted in
\cite{Diaz-G,Diaz-G-PhysRevA,Arenas,Barbu-Cellular,Barbu17,Bantay-Janosi}. For the
existence of a scaling limit for deterministic sandpiles started from
specific initial configurations, we refer to \cite{Pegden-Smart}. In
\cite{Prigozhin,BCRE,Hergarten-Neugebauer}, systems of PDEs are analyzed as
ad-hoc models for natural processes displaying power-law statistics.

The analysis of the
scaling behavior of SOC particle models leads to questions also arising in
the analysis of explicit finite difference discretizations of (generalized)
porous media equations. The latter has been subject to a lively research activity
in numerical mathematics, advancing from the classical power functions (\eg
\cite{Hoff-DiBenedetto}) via differentiable nonlinearities
(\cite{Evje-Karlsen-explicit}) to merely Lipschitz nonlinearities
(\cite{delTesoI}). As related results, we mention convergence results for
implicit finite difference schemes of degenerate porous media equations
(\cite{Evje-Karlsen-implicit,delTesoI}) and a finite-difference
discretization of a fractional porous medium equation
(\cite{DelTeso-fractional}).

Regarding the numerical approximation of probabilistically weak solutions of nonlinear SPDEs
we refer to the overview paper \cite{opw} and the references therein.
For discretizations of stochastic porous media equations, we refer to
\cite{Grillmeier-Gruen}, where an $L^2$ based finite element
approach is applied in order to construct and analyze solutions for sufficiently smooth noise
as well as to the recent work \cite{spme_bgv} that employs an $H^{-1}$ based finite element approach
which also covers the case of space-time white noise for $d=1$.
In \cite{Gerencser-Gyongy-FD,Gyongy-FD}, linear SPDEs with multiplicative noise
are discretized using finite difference approximations in space, while
\cite{McDonald} considers space-time finite difference approximations of
linear parabolic SPDEs with additive noise. To the best of the authors'
knowledge, the present work is the first time that finite difference
approximations of stochastic porous media equations are rigorously
analyzed.

Concerning the underlying techniques the main arguments of this article
rely on, we mention the following sources of theory and inspiration. For
Yamada-Watanabe type results, we refer to \cite{Yamada-Watanabe} for the
foundational work and to \cite{KurtzII,Roeckner} for applications to SPDEs.
The meanwhile classical weak convergence approach has been used previously, \eg,
by \cite{Fischer-Gruen,BFH-incompressible,Gess-Gnann,Gess-Gnann-Dareiotis-Gruen},
relying on a Skorohod-type result by Jakubowski \cite{Jakubowski}. For the
identification of the limit of the discrete approximations as a solution,
we use the theory of maximal monotone operators given \eg in \cite{Barbu} in a
similar way as \cite{Liu-Stephan}, and a generalized Donsker-type
invariance principle given in \cite{Erickson}.

\subsection{Notation}\label{2sec:notation}
Let $\cO\subset \R^d$, $d\in\N$, be an open and bounded set. For
$k\geq 0$, let $\cC^k(\cO)$ ($\cC_c^k(\cO)$) be the space of $k$ times
continuously differentiable real-valued functions (with compact support),
where the index $k=0$ will be omitted. Similarly, for a Banach space $V$,
$\cC^k([0,T]; V)$ denotes be the space of $k$ times continuously
differentiable curves in $V$ parametrized by $t\in[0,T]$. Let
$L^2:=L^2(\cO)$ be the Lebesgue space of square integrable functions,
endowed with the norm $\norm{\cdot}_{L^2}$. Let $\Hsob:= \Hsob(\cO)$ be the
Sobolev space of weakly differentiable functions with zero trace, endowed
with the norm $\norm{u}_\Hsob = \norm{\nabla u}_{L^2}$, and let $\Hd$ be
its topological dual space.
For two separable Hilbert spaces $H_1$, $H_2$, we write $L_2(H_1, H_2)$ for
the space of all Hilbert-Schmidt operators from $H_1$ to $H_2$.

We now turn to the finite-dimensional structures which we will use to formulate
numerical convergence results. From now on, we fix
\begin{displaymath}
  \cO = [0,1] \subset \R.
\end{displaymath}
Consider an equidistant grid on the unit
interval with grid points $(x_i)_{i=0}^\gridlim$ with
$h = \frac1{\gridlim}, \gridlim \in \N$ and $x_i = ih$. For
$i=0,\dots,\gridlim-1$ let $y_i = \left(i+\frac1{2}\right)h$. Consider the
sets of intervals $(K_i)_{i=0,\dots,\gridlim}$ and
$(J_i)_{i=0,\dots,\gridlim-1}$ given by
\begin{align}\label{2eq:partition}
  \begin{split}
    &K_0 = [x_0,y_0),\ K_\gridlim =
    \left[y_{\gridlim-1},x_\gridlim\right],
    \ K_i = [y_{i-1},y_i) \text{ for }i=1,\dots,\gridlim-1,\\
    &J_i = [x_i, x_{i+1}) \text{ for }i=0,\dots,\gridlim-1.
  \end{split}
\end{align}
We consider the space of grid functions on $(x_i)_{i=0}^Z$ with zero
boundary conditions, which is isomorphic to $\R^{\gridlim-1}$. We write
$\mathbf{1} = (1, \dots, 1) \in \R^{\gridlim-1}$. We define the following
prolongations.

\begin{Def}\label{2Def:prolong}
  Let $u_h = (u_{h,1}, \dots, u_{h,Z-1})\in \R^{Z -1}$. We then define
  the piecewise linear prolongation with respect to the grid
  $(x_i)_{i=0,\dots,\gridlim}$ with zero-boundary conditions by
  \begin{displaymath}
    \Iplx: \R^{\gridlim-1} \hookrightarrow \Hsob, \ u_h \mapsto
    \plx{u_h} := \sum_{i=0}^{\gridlim-1} \left[u_{h,i} + \frac{u_{h,i+1} - u_{h,i}}{h} (\cdot -
      x_i)\right]\Ind{J_i},
  \end{displaymath}
  using the convention $u_{h,0} = u_{h,Z} = 0$, and the piecewise constant
  prolongation by
  \begin{displaymath}
    \Ipcx: \R^{\gridlim-1} \hookrightarrow L^2, \
    u_h\mapsto \pcx{u} := \sum_{i=1}^{\gridlim-1} u_{h,i}\Ind{K_i}.
  \end{displaymath}
  The image of $\Ipcx$, \ie the space of piecewise constant functions on
  the partition $(K_i)_{i=0}^Z$ with zero Dirichlet boundary conditions,
  will be denoted by $\pc{S}$. The $L^2$-orthogonal projection to this
  space will be denoted by $\pc\Pi$. Note that
  $\Ipcx: \R^{Z-1} \to \pc{S}$ is
  bijective.
\end{Def}

\newcommand\bigzero[2]{\makebox(#1,#2){\text{\huge0}}}
 Let
$\sp{\cdot}{\cdot} := \sp{\cdot}{\cdot}_{l^2}$ denote the inner product arising from the Euclidean
norm $\norm{\cdot}:=\norm{\cdot}_{l^2}$ on $\R^{\gridlim-1}$. For a matrix $A\in \R^{(Z-1)\times(Z-1)}$,
$\norm{A}$ denotes the matrix norm induced by $\norm{\cdot}$.
Let $\Delta_h\in \R^{(Z-1)\times(Z-1)}$ be the matrix corresponding to the
finite difference Laplacian on grid functions on $(x_i)_{i=0}^Z$ with zero
Dirichlet boundary conditions,
\ie
\begin{equation}\label{2eq:Deltah}
\Delta_h = - \frac1{h^2}
  \begin{pmatrix}
    2&-1&&&&\\
    -1&2&-1&&\bigzero{0}{10}&\\
    &-1&\ddots&\ddots&&\\
    &&\ddots&\ddots&-1&\\
    &\bigzero{0}{0}&&-1&2&-1\\
    &&&&-1&2
  \end{pmatrix}.
\end{equation}
Recall that $-\Delta_h$ is symmetric and positive definite (for a formal
argument, see Lemma \ref{2lem-spectralnorm} below). Hence, the following
definition is admissible.
\begin{Def}\label{2def-discr-norms}
  On $\R^{Z-1}$, we
  define the inner products $\sp{\cdot}{\cdot}_0$, $\sp{\cdot}{\cdot}_1$ and
  $\sp{\cdot}{\cdot}_{-1}$ by
  \begin{align*}
    \sp{u}{v}_0 &= h \sp{u}{v},\\
    \sp{u}{v}_{1} &= \sp{-\Delta_h u}{v}_0,\\
    \sp{u}{v}_{-1} &= \sp{(-\Delta_h)^{-1}u}{v}_0
  \end{align*}
  for $u,v\in \R^{Z-1}$. The induced norms are denoted by $\norm{\cdot}_0, \norm{\cdot}_1$ and $\norm{\cdot}_{-1}$.
\end{Def}

\begin{Rem}\label{2Rem:isometry}
  The norm $\norm{\cdot}_0$ in Definition
  \ref{2def-discr-norms} corresponds to the $L^2$ norm on $\cO$ by the fact
  that
  \begin{displaymath}
    \Ipcx: \left(\R^{Z-1}, \norm{\cdot}_0\right) \to \left(\pcx{S_h},
      \norm{\cdot}_{L^2}\right)
  \end{displaymath}
  is an isometry, \ie
  \begin{displaymath}
    \sp{u}{v}_0 = \sp{\Ipcx u}{\Ipcx v}_{L^2}\quad \text{for }u,v\in\R^{Z-1}.
  \end{displaymath}
  Furthermore, Definition \ref{2def-discr-norms} suggests
  to view $\norm{\cdot}_1$ and $\norm{\cdot}_{-1}$ as discrete analogs of the
  $\Hsob$ and $\Hd$ norms on $\cO$, respectively. These connections are more
  subtle and will be
  made more precise in Lemma \ref{2prop-gridfns}, Lemma \ref{2est-H-1} and
  Proposition \ref{2conv-pc} below. 
\end{Rem}

Next, we consider a lattice for the time interval $[0,T]$, $T>0$. For
$\tau>0$ such that $T= N\tau$, $N\in\N$, consider the equidistant grid $(0,
\tau, 2\tau, \dots, N\tau)$. We then define the following prolongations of
grid functions.
\begin{Def}\label{2Def-prolong-time} \label{2prolong-spacetime}
  Let $v = (v_k)_{k=0}^N \subseteq \R$ be a grid function on the
  previously described grid of length $\tau$ and let $t_\tau := \tau \floor{\frac{t}{\tau}}$. Then we define
  the piecewise linear prolongation $\plt{v}: [0,T]\to\R$, the left-sided piecewise constant
  prolongation $\pctl{v}: [0,T]\to\R$ and the right-sided piecewise constant
  prolongation $\pctr{v}: [0,T]\to\R$ by
  \begin{align*}
    \plt{v}(t) = \frac{t-t_\tau}{\tau} v_{\floor{t/\tau}+1} + \frac{t_\tau +
                \tau - t}{\tau} v_{\floor{t/\tau}}, \quad
    \pctl{v}(t) = v_{\floor{t/\tau}}, \quad
    \pctr{v}(t) = v_{\floor{t/\tau}+1}.
  \end{align*}
  For space-time grid functions
  $\left(u_{k,l}\right)_{k=0,\dots,N; l=1,\dots,Z-1} \subset \R$, we define
  the space-time prolongations $\lc{u}, \clc{u}, \crc{u}$ by extending both
  in space and time. Note that it does not matter whether one first carries
  out the space or the time prolongation.
\end{Def}

\section{Continuum limit for the weakly driven Zhang model}\label{2sec:setting-main-result}

For the first main result, let $\phi_2: \R\to 2^\R$ be the maximal monotone
extension of
\begin{equation}\label{2eq:96}
  \tilde\phi_2: \R\ni x\mapsto x\,\Ind{\abs{x}> 1}(x),
\end{equation}
which is a special case of the Zhang nonlinearity in
\eqref{1eq:149-new}. We consider the singular-degenerate stochastic partial differential
inclusion
\begin{align}\label{2eq:43}
  \begin{split}
    \d X(t) &\in \Delta(\phi_2(X(t))) + \mu\d t + \d W(t),\\
    X(0) &= x_0,
  \end{split}
\end{align}
on the interval $(0,1)\subset \R$ with zero Dirichlet boundary
conditions, where $\mu\geq 0$ and $x_0\in L^2 := L^2((0,1))$. In this
setting, $W$ is a cylindrical $\mathrm{Id}$-Wiener process in $L^2$. We
set the stage for the following analysis by defining a notion of solution
to \eqref{2eq:43} in a probabilistically weak sense.
\begin{Def}\label{2Def-weak-soln}
  A triple
  $\left( (\tilde \Om, \tilde\cF, (\tilde\cF_t)_{t\in[0,T]}, \tilde\P),
    \tilde X, \tilde W\right)$, where
  $(\tilde\Om,\tilde\cF, (\tilde\cF_t)_{t\in[0,T]}, \tilde\P)$ is a complete
  probability space endowed with a normal filtration, 
  \begin{displaymath}
    \tilde X\in L^2\left(\tilde\Om\times[0,T]; L^2\right)\cap L^2\left(\tilde\Om;
    L^\infty([0,T]; \Hd)\right)
  \end{displaymath}
  is an $
  (\tilde\cF_t)_{t\in[0,T]}$-progressively measurable process and $\tilde
  W$ is a cylindrical $\mathrm{Id}$-Wiener process with respect to
  $(\tilde\cF_t)_{t\in[0,T]}$ in
  $L^2$, is a weak solution to \eqref{2eq:43}, if there exists an $
  (\tilde\cF_t)_{t\in[0,T]}$-progressively measurable process $\tilde Y\in
  L^2(\tilde\Om\times[0,T]; L^2)$ such that
  \begin{equation}\label{2eq:111}
    \tilde X(t) = x_0 + \mu t + \int_0^t \Delta \tilde Y(r)\,\d r +  \tilde W(t)
  \end{equation}
  is satisfied in $L^2(\tilde\Om\times[0,T]; (L^2)')$, and
  \begin{equation}\label{2eq:112}
    \tilde Y(t) \in \phi_2(\tilde X(t))\quad   (\d t\otimes \d
    x)\text{-almost everywhere }\tilde\P\text{-almost surely}.
  \end{equation}
\end{Def}

It will be shown in Section \ref{2sec:uniqueness} that the processes $(\tilde
X, \tilde W)$ of every weak solution to \eqref{2eq:43} have the same law
with respect to the Borel $\sigma$-algebra of $L^2([0,T];
L^2)\times\cC([0,T]; \Hd)$.

\newcommand{\sfth}{\sqrt{\frac{\tau}{h}}}

We make the following central
assumption for the rest of this article.
\begin{Ass}\label{2ass-tau-h}
  Let $T>0$. We choose sequences $(Z_m)_{m\in\N}, (N_m)_{m\in\N} \subset
  \N$  such that, for $h_m = \frac1{Z_m}$, $\tau_m = \frac{T}{N_m}$
  ($m\in\N$),
  \begin{displaymath}
    h_m\to 0\text{ for }m\to\infty
  \end{displaymath}
  and
  \begin{equation}
    \tag{CFL}
    \frac{\tau_m}{h_m^2} \to
    0\quad\text{for }m\to\infty
  \end{equation}
  is satisfied, which presents a strengthened Courant-Friedrichs-Lewy-type condition.
\end{Ass}

Motivated by the discrete Zhang model,
we construct a family of time-discrete evolution processes on $\R^{Z_m-1}$ as
follows. For each $m\in\N$, we define
$(X_{h_m}^n)_{n\in\{0,1,\dots, N_m+1\}} \subset \R^{Z_m-1}$ iteratively by
\begin{align}\label{2eq:88}
  \begin{split}
  X_{h_m}^{n+1} &= X_{h_m}^n + \tau_m \Delta_{h_m}\tilde\phi_2(X_{h_m}^n) +
  \mu\tau_m + 
  \sqrt{\frac{\tau_m}{h_m}} \xi_{h_m}^n,\quad\text{for }n=0,\dots, N_m,\\
  X_{h_m}^0 &= x_{h_m}^0,
\end{split}
\end{align}
where $(x_{h_m}^0)_{m\in\N}\subset \R^{Z_m-1}$ such that
$\pcx{(x_{h_m}^0)} \to x_0$ in $L^2$, and
$(\xi_{h_m}^{n,l})_{n=0,\dots,N_m; l=1,\dots,Z_m-1}$ are centered independent random
variables identically distributed on a probability
triple $(\Om, \cF, \P)$. We assume that $\E(\xi_{h_1}^{0,1})^2 = 1$
and that
$\E(\xi_{h_1}^{0,1})^6$ is finite.

Recall the extensions of grid functions to functions defined in Definition
\ref{2prolong-spacetime}. We then
have the
following main result, which will be proved at the end of
Section \ref{2sec:constr-weak-solution}.

\begin{thm}\label{2main-thm}  
  Recall the notation from Section \ref{2sec:notation}, let Assumption
  \ref{2ass-tau-h} be satisfied and, for $m\in\N$, consider the process
  $(X_{h_m}^n)_{n=0}^{N_m}$ given by \eqref{2eq:88}. Then, for
  $m\to \infty$, $\lc{X_{h_m}}$ converges in distribution to the unique
  weak solution to \eqref{2eq:43} with respect to the weak topology on
  $L^2([0,T]; L^2)$ and the weak* topology on $L^\infty([0,T]; \Hd)$.
\end{thm}

\subsection{A priori estimates and transfer to extensions}\label{2sec:constr-weak-solution}

For the rest of this section, we write $\phi = \phi_2$ and $\tilde\phi =
\tilde\phi_2$, and we drop the index $m$ of the discretization
sequences
\begin{displaymath}
(h_m)_{m\in\N}, (Z_m)_{m\in\N}, (\tau_m)_{m\in\N}, (N_m)_{m\in\N},
\end{displaymath}
writing instead $(h)_{h>0}$ etc. Moreover, the convergence of
sequences and usually nonrelabeled subsequences indexed by $m$ for
$m\to\infty$ will be denoted by $h\to 0$. Expressions like ``for $h>0$''
have to be understood in the sense ``for all elements of $(h_m)_{m\in\N}$''
or ``for all elements of the subsequence at hand''. For $h>0$, we set $(\cF_{h}^n)_{n=0}^{N}$,
$\cF_{h}^n\subseteq \cF$, to be the filtration generated by
$(\xi_{h}^k)_{k=0}^{N}$, \ie
\begin{equation}\label{2eq:149}
  \cF_{h}^n = \sigma\left(\xi_{h}^k: k\in\{0,\dots, n-1\}\right)\quad \text{for
  }n\in\{0,\dots, N\}.
\end{equation}
We have the following bounds on the discrete process $X_h$ defined in
\eqref{2eq:88}.
\begin{lem}\label{2discr-energy-est}
  Let $\tau, h> 0, Z, N\in\N$ as in Assumption \ref{2ass-tau-h}, where we
  choose $h$ small enough for $\tau<1$ and $\frac{\tau}{h^2} \leq \frac1{12}$
  to be satisfied. Then, the discrete process in \eqref{2eq:88} satisfies
  \begin{align}\label{2eq:107}
    \begin{split}
      \norm{X_h^n}_{-1}^2 + \mathcal{S}_{n,h}
      &\leq \norm{x_h^0}_{-1}^2 +
      \sum_{k=0}^{n-1}\bigg(2\sfth\sp{X_h^k}{\xi_h^k}_{-1} + C \tau\\
        &\quad + 2
        \tau^{\frac{3}{2}} h^{-\frac1{2}}
        \sp{\Delta_h\tilde{\phi}(X_h^k) + \mu\mathbf{1}}{\xi_h^k}_{-1}
        + \frac{\tau}{h}
        \norm{\xi_h^k}_{-1}^2\bigg)
    \end{split}
  \end{align}
  and
  \begin{equation}\label{2eq:109}
    \E\norm{X_{h}^n}_{-1}^2 + \E\,\mathcal{S}_{n,h}
    \leq \E\norm{x_h^0}_{-1}^2 + n\tau\,\Tr(-\Delta_h^{-1}) + n\tau C
  \end{equation}
  for all $n\in\{1,\dots,N+1\}$, where
  \begin{displaymath}
    \mathcal{S}_{n,h} \in \left\{\tau \sum_{k=0}^{n-1}
      \sp{X_{h}^k}{\tilde{\phi}(X_{h}^k)}_0, \tau
      \sum_{k=0}^{n-1} \norm{\tilde{\phi}(X_{h}^k)}_0^2, \tau
      \sum_{k=0}^{n-1} \norm{X_{h}^k}_0^2 - n\tau\right\}.
  \end{displaymath}
  Moreover, we have for $n\in\{1,\dots,N\}$
  \begin{equation}\label{2eq:91}
    \frac1{h}\E\norm{\xi_h^n}_{-1}^2 =
    \Tr(-\Delta_h^{-1})
  \end{equation}
  and
    \begin{align}
      \label{2eq:faktor-korrigiert}
      \begin{split}
        &\E\norm{X_h^n}_{-1}^2 + 2 \tau\E
        \sum_{k=0}^{n-1}\sp{X_h^k}{\tilde{\phi}(X_h^k)}_0\\
        &\leq
        \E\norm{x_h^0}_{-1}^2 + n\tau\, \Tr(-\Delta_h^{-1}) +
        2 \mu\tau \E\sum_{k=0}^{n-1}\sp{X_h^k}{\mathbf{1}}_{-1}\\
        &\quad +
        \frac{5\tau}{h^2}
        \left(\E\norm{x_h^0}_{-1}^2 + T\,\Tr(-\Delta_h^{-1}) + C\right).
      \end{split}
    \end{align}
\end{lem}
\begin{proof}
  For $n\in\{0,\dots,N\}$, we compute
  \begin{align}\label{2eq:108}
    \begin{split}
      &\norm{X_h^{n+1}}_{-1}^2 = \norm{X_h^n + \tau \Delta_h\tilde{\phi}(X_h^n) + \sfth \xi_h^n + \mu\tau\mathbf{1}}_{-1}^2\\
      &= \norm{X_h^n}_{-1}^2 +
      2\tau\sp{X_h^n}{\Delta_h\tilde{\phi}(X_h^n)}_{-1}\\
      &\quad + 2\mu\tau\sp{X_h^n}{\mathbf{1}}_{-1} +
      \tau^2\norm{\Delta_h\tilde{\phi}(X_h^n)}_{-1}^2 +
      2\mu\tau^2\sp{\Delta_h \tilde\phi(X_h^n)}{\mathbf{1}}_{-1}\\
      &\quad+ 2\sfth\sp{X_h^n}{\xi_h^n}_{-1} +
      2 \tau^{\frac{3}{2}} h^{-\frac1{2}}
      \sp{\Delta_h\tilde{\phi}(X_h^n) + \mu\mathbf{1}}{\xi_h^n}_{-1}\\
      &\quad+ \frac{\tau}{h}\norm{\xi_h^n}_{-1}^2 + \mu^2\tau^2 \norm{\mathbf{1}}_{-1}^2.
    \end{split}
  \end{align}
  Towards the first claim, we aim to absorb the terms of the third line of
  \eqref{2eq:108} up to stably-scaled constants into the term
  \begin{displaymath}
    \tau\sp{X_h^n}{\Delta_h\tilde\phi(X_h^n)}_{-1} =
    -\tau\sp{X_h^n}{\tilde\phi(X_h^n)}_{0} = -\tau\norm{\tilde\phi(X_h^n)}_0^2.
  \end{displaymath}
  To this end, using Lemma \ref{lifted-Poincare} in the second step, we compute
  \begin{align*}
    2\mu\tau\sp{X_h^n}{\mathbf{1}}_{-1}
    &\leq
    2C\tau\norm{X_h^n}_{-1}\norm{\mathbf{1}}_{-1} \leq 2\tau C
      \norm{X_h^n}_0\norm{\mathbf{1}}_0\\
    &\leq 2\tau\left(\frac{3}{2} C + \frac1{6}\norm{X_h^n}_0^2\right) \leq
      3\tau C + \frac1{3}\tau\norm{\tilde\phi(X_h^n)}_0^2,
  \end{align*}
  further, by \eqref{2eq:normest}
  \begin{equation}\label{2eq:carefulest-1}
    \tau^2\norm{\Delta_h\tilde\phi(X_h^n)}_{-1}^2
    \leq 4\tau\frac{\tau}{h^2} \norm{\tilde\phi(X_h^n)}_0^2 \leq
    \frac1{3}\tau\norm{\tilde\phi(X_h^n)}_0^2,
  \end{equation}
  and
  \begin{equation}\label{2eq:carefulest-2}
    \abs{2\mu\tau^2\sp{\Delta_h\tilde\phi(X_h^n)}{\mathbf{1}}_{-1}}\leq
    2C\tau^2\sp{\tilde\phi(X_h^n)}{\mathbf{1}}_0 \leq 3\tau^2 C +
    \frac1{3}\tau^2\norm{\tilde\phi(X_h^n)}_0^2.
  \end{equation}
  Furthermore, note that by the definition of $\tilde\phi$, we have for all
  $x\in\R^{Z-1}$
  \begin{equation}\label{2eq:phitrick}
    \sp{x}{\tilde\phi(x)}_0  = \norm{\tilde\phi(x)}_0^2.
  \end{equation}
  Applying these estimates to \eqref{2eq:108} yields \eqref{2eq:107} with
  the first choice for $\mathcal{S}_{n,h}$ by induction.

  Now taking the expectation in \eqref{2eq:108}, we treat the remaining
  non-constant terms as follows. Recall the definition of the filtration
  $(\cF_h^n)_{n=0}^N$ in \eqref{2eq:149} and note that $X_h^n$ is
  $\cF_h^n$-measurable, while $\xi_h^{n}$ is independent of $\cF_h^n$.
  Hence, the two mixed terms on the right-hand side of \eqref{2eq:107}
  vanish, using the tower property of the conditional expectation. For the
  second-to-last term, we notice that
  \begin{equation}\label{2eq:16}
    \frac{\tau}{h}\, \E\norm{\xi_h^n}_{-1}^2 = \frac{\tau}{h}\, \E\sp{-\Delta_h^{-1}\xi_h^n}{h\xi_h^n}
    = \tau\, \E \sp{-\Delta_h^{-1}\xi_h^n}{\xi_h^n} = \tau \,\Tr(-\Delta_h^{-1}),
  \end{equation}
  since for any family $(\xi_i)_{i=1}^{Z-1}$ of random variables with
  $\E(\xi_i \xi_j) = \delta_{ij}$ and for any matrix
  $A\in\R^{(Z-1)\times(Z-1)}$, we have
  \begin{equation}
    \E\sp{A\xi}{\xi} = \E\sum_{i,j=1}^{Z-1} A_{ij}\xi_j\xi_i = \sum_{i,j=1}^{Z-1}
    A_{ij} \E(\xi_j\xi_i) = \sum_{i,j=1}^{Z-1} A_{ij} \delta_{ij}= \Tr(A). 
  \end{equation}
  In particular, \eqref{2eq:91} follows. Collecting all estimates, we conclude by
  induction that
  \begin{equation}\label{2eq:90}
    \E\norm{X_h^n}_{-1}^2 + \tau \sum_{k=0}^{n-1} \E\sp{X_h^n}{\tilde{\phi}(X_h^n)}_0 \leq
    \E\norm{x_h^0}_{-1}^2 + n\tau\, \Tr(-\Delta_h^{-1}) + n\tau C
  \end{equation}
  for $n\in\{0,\dots,N+1\}$, which proves \eqref{2eq:109} for the first
  choice of $\mathcal{S}_{n,h}$.  In view of \eqref{2eq:phitrick}, this
  immediately yields \eqref{2eq:107}
    and \eqref{2eq:109} for the second choice of
  $\mathcal{S}_{n,h}$. The relation $ \abs{\tilde\phi(x)}^2 \geq \abs{x}^2
  - 1$
  extends these statements to the last choice of
  $\mathcal{S}_{n,h}$. Using the estimates \eqref{2eq:carefulest-1} and
  \eqref{2eq:carefulest-2} without absorbing the respective terms, taking
  expectation and summing up, we obtain from \eqref{2eq:108}
  \begin{align*}
    &\E\norm{X_h^n}_{-1}^2 + 2 \tau\E
      \sum_{k=0}^{n-1}\sp{X_h^k}{\tilde{\phi}(X_h^k)}_0\\
    &\leq
    \E\norm{x_h^0}_{-1}^2 + n\tau\, \Tr(-\Delta_h^{-1}) +
      2\tau\mu\,\E\sum_{k=0}^{n-1}\sp{X_h^k}{\mathbf{1}}_{-1}\\
    &\quad +  \left(\frac{4\tau}{h^2} + \frac{\tau}{3}\right)
    \tau \E \sum_{k=0}^{n-1}\norm{\tilde\phi(X_h^k)}_0^2 + \tau n\tau C.
  \end{align*}
  Finally, using that $h\leq 1$ by assumption and \eqref{2eq:109},
  \eqref{2eq:faktor-korrigiert} follows.
\end{proof}

\begin{lem}\label{2discr-BDG-diagonal}
  Let $\tau, h> 0$ and $N,Z\in\N$ as in Assumption \ref{2ass-tau-h}, and let
  $(X_h^n)_{n=0}^N$ be constructed as in \eqref{2eq:88}. Then
  \begin{displaymath}
    \E\max_{n=0,\dots, N}\norm{X_h^n}_{-1}^2 \leq C,
  \end{displaymath}
  where $C$ is independent of $h$.
\end{lem}

\begin{proof}
  This follows as an application of the Burkholder-Davis-Gundy
  inequality in the form of \cite[Theorem 1]{Davis} to the discrete
  martingales in \eqref{2eq:107}.
\end{proof}

\begin{lem}\label{2discr-cnty-estimate}
  Let $\tau, h> 0$ as in Assumption \ref{2ass-tau-h}.
  Then, there exists a
  constant $C>0$ which only depends on $T$ and $x_h^0$, such that the
  discrete process in \eqref{2eq:88} satisfies
  \begin{displaymath}
    \E\norm{X_h^{n+1} - X_h^n}_{-1}^2  \leq C\frac{\tau}{h^2}\quad
    \text{for all }n\in\{0,\dots,N-1\}.
  \end{displaymath}
\end{lem}
\begin{proof}
  We compute
  \begin{align}\label{2eq:92}
  \begin{split}
    &\E\norm{X_h^{n+1} - X_h^n}_{-1}^2
    = \E\norm{\tau \Delta_h\tilde\phi(X_h^n) + \sfth \xi_h^n + \tau\mu\mathbf{1}}_{-1}^2\\
    &= \E\norm{\tau \Delta_h \tilde\phi(X_h^n)}_{-1}^2 + \frac{\tau}{h}
    \E\norm{\xi_h^n}_{-1}^2 + \tau^2\mu^2\norm{\mathbf{1}}_{-1}^2\\
    &\quad + 2\E\sp{\tau
      \Delta_h\tilde\phi(X_h^n) + \tau\mu\mathbf{1}}{\sfth \xi_h^n}_{-1} +
    2\tau^2\mu\,\E\sp{\Delta_h \tilde\phi(X_h^n)}{\mathbf{1}}_{-1}.
  \end{split}
  \end{align}
  As in the proof of
  Lemma \ref{2discr-energy-est}, we have that
  \begin{equation}\label{2eq:94}
    \E\sp{\tau \Delta_h\phi(X_h^n) + \tau\mu\mathbf{1}}{\sfth \xi_h^n}_{-1} = 0
  \end{equation}
  by the independence of $\xi_h^n$ of $\cF_h^n$, where $(\cF_h^n)_{n=0}^N$
  is given as in \eqref{2eq:149}. In view of \eqref{2eq:91} and Lemma
  \ref{2trace-inv-laplace}, one may choose $C$ independent of $h$
  satisfying
  \begin{equation}\label{2eq:95}
    \frac{\tau}{h}
    \E\norm{\xi_h^n}_{-1}^2 \leq \tau C.
  \end{equation}
  Finally, using \eqref{2eq:carefulest-1} and \eqref{2eq:carefulest-2}, we have
  \begin{displaymath}
    \tau^2\E\norm{\Delta_h\tilde\phi(X_h^n)}_{-1}^2 +
    2\tau^2\mu\E\sp{\Delta_h \tilde\phi(X_h^n)}{\mathbf{1}}_{-1} \leq
    C\frac{\tau^2}{h^2}\E\norm{\tilde\phi(X_h^n)}_0^2 + C\tau^2,
  \end{displaymath}
  such that we can use Lemma \ref{2discr-energy-est} and Lemma
  \ref{2trace-inv-laplace} to finish \eqref{2eq:92} by
  \begin{align*}
  \begin{split}
    \E\norm{X_h^{n+1} - X_h^n}_{-1}^2
    &\leq C\frac{\tau^2}{h^2} \sum_{k=0}^N \E\norm{\tilde\phi(X_h^k)}_0^2 + \tau
    C\\
    &\leq C\frac{\tau}{h^2}\left(\E\norm{x_h^0}_{-1}^2 + T\,
        \Tr(-\Delta_h^{-1}) + C\right) \leq C \frac{\tau}{h^2},
  \end{split}
  \end{align*}
  as required.
\end{proof}

The estimates proved above in the discrete setting can be transferred to
the extensions to functions, exploiting Lemma \ref{2est-H-1}.
\begin{Cor}\label{2cts-energy-est}\label{2Cty-est-for-interpol}\label{2conv-interpols}
  Let $\tau, h> 0, N,Z\in\N$ as in Assumption \ref{2ass-tau-h}, with $h$
  small enough for $\frac{\tau}{h^2} \leq \frac1{12}$ to be
  satisfied, and let $X_h$ be constructed as in \eqref{2eq:88}. Then, there
  exists $C>0$ only depending on $T$ (in particular, independent of $h$), such that
  \begin{align}
    \label{2eq:124}
    \max\left\{\E\int_0^T\norm{\lc{X_h}}_{L^2}^2\d t,\,
      \E\int_0^T\norm{\clc{X_h}}_{L^2}^2\d t,\,
      \E\int_0^T\norm{\crc{X_h}}_{L^2}^2\d t,\right\} &\leq C,\\
    \label{2eq:125}
    \E\int_0^T\norm{\tilde\phi(\clc{X_h})}_{L^2}^2\d t &\leq C,\\
    \label{2eq:146}
    \text{and}\quad\E\esssup_{t\in[0,T]}\norm{\lc{X_h}}_\Hd^2 &\leq C.
  \end{align}
  Furthermore,
  \begin{displaymath}
    \sup_{t\in[0,T]}\E\norm{\lc{X_h}(t) - \clc{X_h}(t)}_\Hd^2, \sup_{t\in[0,T]}\E\norm{\lc{X_h}(t) -
      \crc{X_h}(t)}_\Hd^2 \leq C\frac{\tau}{h^2}.
  \end{displaymath}
\end{Cor}

\subsection{Extraction of convergent subsequences}

\begin{Def}\label{2Def-F}
  Let $(X_h^n)_{n=0}^{N}$ and $(\xi_h^{n})_{n=0}^N$ be defined as in
  \eqref{2eq:88}. We then define random variables $Y_h, W_h: \Om\to \R^{(Z-1)(N+1)}$ by
  \begin{displaymath}
    Y_h = \left(\tilde\phi(X_h^n)\right)_{n=0}^N\quad \text{and}
    \quad W_h =
    \left(\sum_{k=0}^{n-1}\sqrt{\frac{\tau}{h}}\xi_h^k\right)_{n=0}^{N+1}.
  \end{displaymath}
  Furthermore, we define $F_h: \Om\to\cC([0,T]\times[0,1])$ to be the spatial
  antiderivative of $\lc{W_h}$, \ie
  \begin{equation}\label{2eq:45}
    F_h(t,x) = \int_0^x \lc{W_h}(t, x')\,\d x'.
  \end{equation}
\end{Def}

\begin{Rem}\label{2discr-noise-cts}
  Note that $F_h$ is continuous in time by the continuity of the piecewise
  linear prolongation, and absolutely continuous in space, since
  $\lc{W_h}(t,\cdot)$ is Lebesgue integrable at any time $t\in[0,T]$.
\end{Rem}

\begin{lem}
  The distributions of $(F_h)_{h>0}$ are tight with respect to the strong
  topology $\tau_C$ of $\cC([0,T]\times[0,1])$ and converge to the
  distribution of the Brownian sheet (for a Definition, see
  \cite[p.\,1]{BS-cap}) in $\cC([0,T]\times[0,1])$. 
\end{lem}

\begin{proof}
If we had defined the spatial
extension to be piecewise
constant \textit{between} the lattice points, this statement would have
followed immediately from \cite[Theorem 7.6]{Erickson}. Indeed, the process considered
there could be identified with $F_h$. In order to
transfer these results to the extension which is piecewise constant
\textit{around} the lattice points, as used in the present work, we exploit the explicit connection
between these two extensions. In particular, the uniform continuity used to
prove tightness carries over from one extension to the other, and the
convergence in law then follows from the convergence of the
finite-dimensional distributions of the process in \cite{Erickson} 
and the
fact that the finite-dimensional distributions of the difference of the two extensions
converge stochastically to zero.
\end{proof}

Corollary \ref{2cts-energy-est}
implies tightness of the distributions of $(\lc{X_h})_{h>0}$ with respect
to the weak* topology $\tau_w^*$ of $L^\infty([0,T]; \Hd)$ by the
Banach-Alaoglu theorem, and tightness of the distributions of all processes
in Corollary \ref{2cts-energy-est} with respect to the weak topology
$\tau_w$ of $L^2([0,T]; L^2)$ by separability, reflexivity and the
Eberlein-Smulian theorem. As a consequence, we have
tightness of the distributions of the family
\begin{equation}\label{2eq:tightness-tup}
\left((\lc{X_h}, \lc{X_h},
  \clc{X_h}, \crc{X_h}, \clc{Y_h}, F_h)\right)_{h>0}
\end{equation}
with respect to the
product topology of
$(\tau_w^*, \tau_w, \tau_w, \tau_w, \tau_w, \tau_C)$, which is a key
ingredient to prove the following.

\begin{lem}\label{2Prohorov-Skorohod}
  Let $(X_h)_{h>0}$, $(Y_h)_{h>0}$ and $(W_h)_{h>0}$ be defined as in \eqref{2eq:88} and
  Definition \ref{2Def-F}, respectively. Then, there is a probability space $(\tilde\Om, \tilde\cF, \tilde\P)$,
  stochastic processes
  \begin{align*}
    \tilde X
    &\in L^2(\tilde\Om; L^\infty([0,T]; \Hd)) \cap L^2(\tilde\Om; L^2([0,T]; L^2),\\
    \tilde Y
    &\in L^2(\tilde\Om; L^2([0,T]; L^2)),\\
    \tilde W
    &\in L^2(\tilde\Om; \cC([0,T]; \Hd)),
  \end{align*}
  where $\tilde W$ is a cylindrical $\mathrm{Id}$-Wiener process in $L^2$,
  a nonrelabeled
  subsequence $h\to 0$ such that for each $h$ in this subsequence, there
  are random variables
  $\tilde X_h, \tilde Y_h, \tilde W_h: \tilde\Om\to \R^{(N+1)(Z-1)}$, such
  that for
  each $h$ in this subsequence,
  \begin{align}\label{2eq:laws-discr}
    &\cL\left((\tilde X_h, \tilde Y_h, \tilde W_h)\right) = \cL\left((X_h,
    Y_h, W_h)\right),\\
  \label{2eq:laws-emb}
    \begin{split}
    &\cL\left((\lc{\tilde X_h}, \lc{\tilde X_h}, \clc{\tilde X_h}, \crc{\tilde X_h},
      \clc{\tilde Y_h}, \lc{\tilde W_h})\right)\\
    &= \cL\left((\lc{X_h},\lc{X_h}, \clc{X_h}, \crc{X_h}, \clc{Y_h}, \lc{W_h})\right)
  \end{split}
  \end{align}
  with respect to the product topology of $(\tau_w^*, \tau_w, \tau_w,
  \tau_w, \tau_w, \tilde\tau_C)$, and
  $\tilde\P$-almost surely, for $h\to 0$,
  \begin{align*}
    \lc{\tilde X_h}
    &\tows \tilde X \text{ in } L^\infty([0,T];\Hd),\\
    \lc{\tilde X_h}
    &\tow \tilde X, \clc{\tilde X_h} \tow \tilde X, \crc{\tilde X_h} \tow
      \tilde X, \clc{\tilde Y_h} \tow \tilde Y \text{ in } L^2([0,T]; L^2),\\
    \text{and} \quad\lc{\tilde W_h}
    &\to \tilde W \text{ in } \cC([0,T];\Hd).
  \end{align*}
\end{lem}

\begin{proof}[Proof of Lemma \ref{2Prohorov-Skorohod}]
  Since this type of result is classical in the framework of the so-called
  weak convergence approach (see \eg
  \cite{Flandoli-Gatarek,Fischer-Gruen,BFH-incompressible,Gess-Gnann-Dareiotis-Gruen}), we only mention
  the main steps. First, one applies the Skorohod-type result by Jakubowski
  (\cite[Theorem 2]{Jakubowski}) to the six-tuple \eqref{2eq:tightness-tup},
  which uses the abovementioned tightness. Having obtained corresponding
  almost surely converging processes $(\lc{\tilde X_h})_{h>0}$ etc.\ on a
  different probability space, we transfer the estimates in Corollary
  \ref{2Cty-est-for-interpol} to these processes and obtain the suitable
  integrability of the limits by the Fatou lemma and weak(*)
  lower-semicontinuity of the norms. Next, we identify the $\tau_w$-limits
  of $(\lc{\tilde X_h})_h$, $(\clc{\tilde X_h})_h$ and $(\crc{\tilde X_h})_h$
  using the last part of Corollary \ref{2Cty-est-for-interpol}, and we
  identify this limit with the $\tau_w^*$-limit of $(\lc{\tilde X_h})_h$
  by embedding the respective spaces into the common superspace
  $L^2([0,T]; \Hd)$. Passing to the distributional spatial derivative of
  $(\tilde F_h)_h$ and its limit $(\tilde F)$ provides $(\lc{\tilde W_h})$
  and $\tilde W$, where the latter is identified as a cylindrical
  $\mathrm{Id}$-Wiener process in $L^2$ by the fact that $\tilde F$ is a
  Brownian sheet. Finally, one identifies the newly-constructed processes
  as images of the respective extension operator, using that the projection
  to the respective finite-dimensional subspace is continuous, and
  constructs pre-images by applying this projection.
\end{proof}

\begin{Rem}
  Expected values with respect to $\tilde\P$ will be denoted by $\tilde\E$.
\end{Rem}

  \subsection{Identification of the limit as a solution and proof of Theorem \ref{2main-thm}}
We now turn to show that the limit processes belong to a weak
solution. As a stochastic basis, we choose $(\tilde\Om, \tilde\cF,
\tilde\P)$ endowed with the augmented filtration
  $(\tilde\cF_t)_{t\in[0,T]}$ of $(\tilde\cF_t')_{t\in[0,T]}$, where
  \begin{equation}\label{2eq:Def-stoch-basis}
    \tilde\cF_t' := \sigma\left(\tilde X|_{\tilde\Om\times[0,t]}, \tilde
      Y|_{\tilde\Om\times[0,t]}, \tilde W|_{\tilde\Om\times[0,t]}\right).
  \end{equation}
  As typical of the weak convergence approach, we
  obtain that $\tilde W$ is a Wiener process with respect to this basis in
  the sense of \cite[Definition 2.1.12]{Roeckner}, which is a consequence
  of the independence of external increments in the discrete dynamics and
  the continuity of the paths of $\tilde W$. We
  continue by showing that $(\tilde X, \tilde Y, \tilde W)$ satisfy
  equation \eqref{2eq:111}.

\begin{lem}\label{2ident-eqn}
  Let $(\tilde X_h, \tilde Y_h, \tilde W_h)$ be the processes from Lemma
  \ref{2Prohorov-Skorohod}. Then,
  \begin{equation}\label{2eq:110}
    \plt{\tilde X_h}(t) = x_h^0 + \int_0^t\Delta_h\pctl{\tilde Y_h}(r)\,\d r +
    \plt{\tilde W_h}(t) + \mu t \mathbf{1}
  \end{equation}
  in $L^2([0,T]; \R^{Z-1})$, $\tilde\P$-almost surely, and the limits in
  Lemma \ref{2Prohorov-Skorohod} satisfy
  \begin{equation}\label{2eq:28}
    \tilde X(t) = x_0 + \int_0^t \Delta \tilde Y(r) \d r + \tilde W(t) +
    \mu t
  \end{equation}
  in $L^2([0,T]; (L^2)')$, $\tilde\P$-almost surely.
\end{lem}
\begin{proof}
  \textit{Step 1:} We first prove \eqref{2eq:110}. We note that by construction of
  the prolongations in use here, \eqref{2eq:110} is equivalent to
  \begin{displaymath}
    \tilde X_h^n = x_h^0 + \tau\sum_{k=0}^{n-1}\Delta_h\tilde Y_h^k +
    \tilde W_h^n + n\tau \mu \mathbf{1}
  \end{displaymath}
  for all $n\in\{0,\dots,N\}$, $\tilde\P$-almost surely, which is verified by
  the construction of $(X_h, Y_h, W_h)$ in \eqref{2eq:88} and Definition
  \ref{2Def-F}, and by the equality of laws in \eqref{2eq:laws-discr}.
  
  \textit{Step 2:} We need to show that $\tilde\P$-almost surely, for every
  $\zeta\in L^2([0,T]; L^2)$,
  \begin{align}\label{2eq:29}
    \begin{split}
    \int_0^T \sp{\tilde X(t)}{\zeta(t)}_{(L^2)'\times L^2}\d t = \int_0^T \sp{x_0 + \int_0^t \Delta \tilde Y(r) \d r +
      \tilde W(t) + \mu t}{\zeta(t)}_{(L^2)'\times L^2}\d t,
  \end{split}
  \end{align}
  where $\sp{u}{v}_{(L^2)'\times L^2} = \sp{-\Delta^{-1}u}{v}_{L^2}$.
  In this step, we first show \eqref{2eq:29} for a test function $\zeta$ of
  the type
  \begin{equation}\label{2eq:38}
    \zeta = \theta(t)\eta,
  \end{equation}
  where $\eta\in L^2$ and $\theta\in L^\infty([0,T])$. By considering
  \eqref{2eq:110} and using an approximation $\eta_h\in\R^{Z-1}$ such that $\pc{\eta} \to \eta$ in $L^2$ for
  $h\to 0$,
  we obtain
  \begin{equation}\label{2eq:30}
    \int_0^T\sp{\plt{\tilde
            X_h}(t)}{\theta(t)\eta_h}_{-1}\d t = \int_0^T\sp{x_h^0 +
          \int_0^t \Delta_h \pctl{\tilde Y_h}(r)\d r +
          \plt{\tilde W_h}(t) + \mu t \mathbf{1}}{\theta(t)\eta_h}_{-1}\d t
  \end{equation}
  $\tilde\P$-almost surely. Lemma \ref{2Prohorov-Skorohod} and Proposition \ref{2conv-pc}
  then allow to pass to the limit obtaining \eqref{2eq:29}, using dominated convergence for the term
  including $\tilde Y_h$.

  \textit{Step 3:} By linearity, \eqref{2eq:29} is also true for linear
  combinations of test functions $\zeta$ of type \eqref{2eq:38} and thus for
  every polynomial. Thus, we obtain a $\tilde\P$-zero set outside of which
  \eqref{2eq:29} is satisfied for any polynomial. Using the density of
  polynomials in $L^2([0,T]\times[0,1])$ given by the Stone-Weierstrass
  theorem, outside this zero set the full statement \eqref{2eq:29} is
  satisfied by passing to the limit of approximating sequences.
\end{proof}

\begin{lem}\label{2cont-est}
  Let $\tilde Y$ and $\tilde W$ be constructed as in Lemma
  \ref{2Prohorov-Skorohod} and define the continuous $(L^2)'$-valued
  process
  \begin{displaymath}
    \tilde Z(t) := x_0 + \int_0^t \Delta \tilde Y(r) \d t + \tilde W(t) +
    \mu t
  \end{displaymath}
  for $t\in[0,T]$. Then, we have
  \begin{displaymath}
    \tilde\E \left(\sup_{t\in[0,T]} \norm{\tilde Z(t)}_\Hd^2\right) < \infty
  \end{displaymath}
  and
  \begin{align}\label{2eq:27}
    \begin{split}
    &\tilde\E\norm{\tilde Z(t)}_\Hd^2  + 2\,\tilde\E\int_0^t\sp{\tilde
      Z(r)}{\tilde Y(r)}_{L^2}\d r\\
    &= \norm{x_0}_\Hd^2 + t\norm{I'}_{L_2(L^2,\Hd)}^2 +
    2\,\tilde\E\int_0^t\sp{\tilde Z(r)}{\mu}_\Hd \d r,
  \end{split}
  \end{align}
  where $I': L^2\hookrightarrow \Hd$ is the canonical embedding (\cf Lemma
  \ref{2trace-inv-laplace}).
\end{lem}

\begin{proof}
  By Lemma \ref{2ident-eqn}, we have that $\tilde X$ and $\tilde Z$ are in
  the same $\tilde\P\otimes\d t$-equivalence class, and by the construction
  in Lemma \ref{2Prohorov-Skorohod} we know that
  $\tilde X\in L^2(\tilde\Om\times[0,T]; L^2)$. Moreover,
  $\tilde Y\in L^2(\tilde\Om; L^2([0,T]; L^2))$ and progressively
  measurable with respect to $(\tilde\cF_t)_{t\in[0,T]}$ by
  construction. Thus, Itô's formula from \cite[Theorem 4.2.5]{Roeckner}
  applies, which yields both claims.
\end{proof}

\begin{Rem}\label{2Rem:ident-Z-X}
  By \eqref{2eq:28} and the
  definition of $\tilde Z$ above, we have $\tilde Z = \tilde X$ in
  $L^2(\tilde\Om\times[0,T]; (L^2)')$. Furthermore, we have $\tilde X\in
  L^2(\tilde\Om\times[0,T]; L^2)$, such that the injectivity of the embedding
  $L^2\hookrightarrow \Hd \hookrightarrow (L^2)'$, which carries over to an embedding
  \begin{displaymath}
    L^2(\tilde\Om\times[0,T]; L^2) \hookrightarrow
    L^2(\tilde\Om\times[0,T]; (L^2)'),
  \end{displaymath}
  implies that
  $\tilde Z \in L^2(\tilde\Om\times[0,T]; L^2)$ and $\tilde Z = \tilde X$
  in $L^2(\tilde\Om\times[0,T]; L^2)$.
\end{Rem}

In view of \eqref{2eq:112}, it remains to inspect the relation of $\tilde
X$ and $\tilde Y$. To this end, we aim to use \eqref{2eq:27} for $\tilde Z$
replaced by $\tilde X$. Since this is only possible $\d t$-almost surely,
we need to use an integrated version of \eqref{2eq:27}. The resulting
double integral in the second term leads to the following definition, which
will be useful in the proof of Lemma \ref{2limsup-diagonal} below.

\begin{Def}\label{2weighted-T-space}
  We define the measure $\mu$ on $[0,T]$ as the
  measure with density
  \begin{displaymath}
    [0,T]\ni t\mapsto T-t
  \end{displaymath}
  with respect to $\d t$, and we write $[0,T]_\mu$ for the measure space $([0,T], \mu)$.
  Let $\cAOmT \subset L^2(\tilde\Om\times[0,T]_\mu; L^2) \times
  L^2(\tilde\Om\times[0,T]_\mu; L^2)$ be a multivalued operator (which we
  identify with its graph by a slight abuse of notation) defined by
  \begin{equation}\label{2eq:26}
    (X,Y)\in \cAOmT \quad \text{if and only if}\quad Y\in \phi(X) \ \text{for
      almost every }(\tilde\omega, t, x) \in \tilde\Om\times[0,T]\times[0,1].
  \end{equation}
\end{Def}

\begin{Rem}\label{2cA-max-mon}
  Similarly to the proof of Lemma \ref{2cAT-max-mon}, we obtain that the
  operator $\cA$ is maximal monotone.
\end{Rem}

\begin{lem}\label{2limsup-diagonal}
  Let $h>0$, $(\clc{\tilde X_h})_{h>0}, (\clc{\tilde Y_h})_{h>0}, \tilde X, \tilde Y$ be as in Lemma
  \ref{2Prohorov-Skorohod}. Then
  \begin{align*}
    \limsup_{h\to 0} \tilde E \int_0^T (T-t)
    \sp{\clc{\tilde X_h}(t)}{\clc{\tilde Y}_h(t)}_{L^2} \d t\leq \tilde\E
    \int_0^T (T-t)\sp{\tilde X(t)}{\tilde Y(t)}_{L^2}\d t.
  \end{align*}
\end{lem}

\begin{proof}
  We notice that for $f\in L^1([0,T]; \R)$ or measurable $f\geq 0$, we have by
  Fubini's (resp. Tonelli's) theorem
  \begin{align}\label{2eq:Fub}
    \begin{split}
      &\int_0^T\int_0^t f(r)\d r \d t = \int_0^T \int_0^T \Ind{[0,t]}(r)
      f(r)\d r \d t\\
      &= \int_0^T f(r) \int_0^T\Ind{[r, T]}(t)\d t\d r =
      \int_0^T(T-r)f(r)\d r.
  \end{split}
  \end{align}
  Since, due to Lemma
  \ref{2Prohorov-Skorohod}, 
  $(\crc{\tilde X_h})_{h>0}$
  is bounded in $L^2(\tilde\Om; L^2([0,T]; L^2))$ uniformly in $h$, and
  $\crc{\tilde X_h}\tow \tilde X$ $\tilde\P$-almost surely in
  $L^2([0,T]; L^2)$, we have
  \begin{displaymath}
    \crc{\tilde X_h} \tow \tilde X\quad
    \text{in }L^2(\tilde\Om; L^2([0,T]; L^2))
  \end{displaymath}
  for $h\to 0$.
  Hence,we have by weak lower-semicontinuity of
  the norm that
  \begin{equation}\label{2eq:114}
    - \tilde\E\int_0^T\norm{\tilde X(t)}_\Hd^2\d t \geq \limsup_{h\to 0}\left(
      - \tilde\E\int_0^T\norm{\crc{\tilde X_h}(t)}_\Hd^2\d t\right).
  \end{equation}
  Furthermore, by the same arguments as in the proof of Proposition \ref{2conv-pc}, we obtain
  \begin{displaymath}
    \pcx{\left(-\Delta_h^{-1} x_h^0)\right)} \tow -\Delta^{-1}x_0\quad \text{in }L^2\
    \text{for }h\to 0,
  \end{displaymath}
  which allows to compute
  \begin{align}\label{2eq:init-val-ident}
    \begin{split}
      \lim_{h\to 0}\norm{x_h^0}_{-1}^2 = \lim_{h\to 0}\sp{-\Delta_h^{-1}x_h^0}{x_h^0}_0
      &= \lim_{h\to 0}\sp{\pcx{(-\Delta_h^{-1}x_h^0)}}{\pcx{(x_h^0)}}_{L^2}\\
      &= \sp{-\Delta^{-1}x_0}{x_0}_{L^2}
      = \norm{x_0}_\Hd^2.
    \end{split}
  \end{align}
  For each $h>0$ in the subsequence of Lemma \ref{2Prohorov-Skorohod},
  consider $\tilde X_h$ and $\tilde Y_h$ as constructed in Lemma
  \ref{2Prohorov-Skorohod}. Then, by \eqref{2eq:Fub} and Remark
  \ref{2Rem:isometry},
  we obtain
  \begin{align}
    \nonumber
    &\limsup_{h\to 0} \tilde\E\int_0^T(T-t)\sp{\clc{\tilde X_h}(t)}{\clc{\tilde
      Y_h}(t)}_{L^2} \d t\\
    \nonumber
    &= \limsup_{h\to 0} \int_0^T \tilde\E\int_0^t\sp{\clc{\tilde
      X_h}(r)}{\clc{\tilde Y_h(r)}}_{L^2} \d r\,\d t\\  
    \label{2eq:Barbu1}
    &= \limsup_{h\to 0}\int_0^T\tilde\E\int_0^t\sp{\pctl{\tilde X_h}(s)}{\pctl{\tilde
      Y_h}(s)}_0\d s\,\d t.
  \end{align}
  Writing $t_\tau = \floor{t/\tau}\tau$ and using the definition of the
  left-sided piecewise constant embedding
  embedding, the positive sign of
  $\sp{\pctl{\tilde X_h}}{\pctl{\tilde Y_h}}_0$
  $\tilde\P\otimes\d t$-almost everywhere and Lemma \ref{2discr-energy-est}, we continue by
  \begin{align}
    \nonumber
    &\eqref{2eq:Barbu1}
    = \limsup_{h\to 0} \int_0^T
      \tilde\E\left[\sum_{n=0}^{\floor{t/\tau}}\tau\sp{\tilde X_h^n}{\tilde Y_h^n}_0
      - \int_t^{t_\tau + \tau}\sp{\pctl{\tilde X_h}(s)}{\pctl{\tilde
      Y_h}(s)}_0\d s\right]\d t\\
    \begin{split}\label{2eq:Barbu2}
    &\leq \frac1{2} \limsup_{h\to
      0}\left(-\int_0^T\tilde\E\norm{\tilde
      X_h^{\floor{t/\tau}+1}}_{-1}^2\d t\right) + \lim_{h\to 0}
  \int_0^T \tilde\E \sum_{n=0}^{\floor{t/\tau}} \tau\mu \sp{\tilde
    X_h^n}{\mathbf{1}}_{-1}\d t\\
    &\quad +
      \frac1{2}\lim_{h\to 0} \int_0^T\norm{x_h^0}_{-1}^2\d t + \frac1{2}\lim_{h\to 0}
      \int_0^T(t_{\tau}+\tau)\,\Tr(-\Delta_h^{-1})\,\d t\\
    &\quad + \frac1{2}\lim_{h\to 0}\frac{5\tau}{h^2}
    \left(\E\norm{x_h^0}_{-1}^2 + T\,\Tr(-\Delta_h^{-1}) + C\right).
  \end{split}
  \end{align}
  For the second term on the right
  hand side, we compute
  \begin{align*}
    \lim_{h\to 0} \int_0^T\tilde\E \sum_{n=0}^{\floor{t/\tau}} \tau\mu
    \sp{\tilde X_h^n}{\mathbf{1}}_{-1}\d t
    &= \lim_{h\to 0}\int_0^T \mu\,\tilde\E \sp{\sum_{n=0}^{\floor{t/\tau}}
      \tau\tilde X_h^n}{\mathbf{1}}_{-1}\d t\\
    &= \lim_{h\to 0}\int_0^T\mu\,\tilde\E\int_0^{T} \sp{\pctl{\tilde
      X_h}(s)}{\mathbf{1} \Ind{[0,t_\tau+\tau]}(s)}_{-1}\d s\,\d t.
  \end{align*}
  We first show that the expected value converges $\d t$-almost
  everywhere. To this end, note that for $h\to 0$
  \begin{displaymath}
    \pcx{(\mathbf{1} \Ind{[0,t_\tau + \tau]})} \to \Ind{[0,t]}\quad
    \text{in } L^2([0,T]; L^2)
  \end{displaymath}
  and
  \begin{displaymath}
    \clc{\tilde X_h} \tow \tilde X \quad\text{in }L^2([0,T]; L^2)\
    \tilde\P\text{-almost surely}
  \end{displaymath}
  by Lemma \ref{2Prohorov-Skorohod}. Hence, Proposition \ref{2conv-pc}
  yields
  \begin{displaymath}
    \int_0^{T} \sp{\pctl{\tilde
        X_h}(s)}{\mathbf{1} \Ind{[0,t_\tau+\tau]}(s)}_{-1}\d s \to \int_0^T
    \sp{\tilde X(s)}{\Ind{[0,t]}(s)}_\Hd\d s
  \end{displaymath}
  $\tilde\P$-almost surely. Furthermore,
  \begin{align}\label{2eq:L2Om-bound}
    \begin{split}
    \tilde\E \abs{\int_0^{T} \sp{\pctl{\tilde
    X_h}(s)}{\mathbf{1} \Ind{[0,t_\tau+\tau]}(s)}_{-1}\d s}^2
    &\leq \tilde\E\int_0^T \norm{\pctl{\tilde X_h}(s)}_{-1}^2 \norm{\mathbf{1}}_{-1}^2\d
      s\\
    &\leq C \tilde\E\int_0^T \norm{\clc{\tilde X_h}(s)}_\Hd^2\d s\\
    &\leq C \tilde\E\int_0^T \norm{\clc{\tilde X_h}(s)}_{L^2}^2\d s \leq C,
  \end{split}
  \end{align}
  where the last step is due to Corollary \ref{2cts-energy-est}. Hence, for $h\to 0$,
  \begin{displaymath}
    \tilde\E\int_0^{T} \sp{\pctl{\tilde
        X_h}(s)}{\mathbf{1} \Ind{[0,t_\tau+\tau]}(s)}_{-1}\d s \to
    \tilde\E\int_0^T \sp{\tilde X(s)}{\Ind{[0,t]}(s)}_\Hd\d s
  \end{displaymath}
  $\d t$-almost everywhere. The calculation \eqref{2eq:L2Om-bound} also
  justifies using the dominated convergence theorem for the outer
  integral. Using these considerations, the definition of the right-sided
  piecewise constant embedding, \eqref{2eq:init-val-ident}, Lemma
  \ref{2trace-inv-laplace} and Lemma \ref{2est-H-1}, we obtain
  \begin{align}
    \nonumber
    &\eqref{2eq:Barbu2} =
    \frac1{2}\limsup_{\h\to
      0}\left(-\int_0^T\tilde\E\norm{\pctr{\tilde X_h}(t)}_{-1}^2\d
      t\right) + \int_0^T\E\int_0^t\sp{\tilde X(s)}{\mu}_\Hd\d s
      \,\d t\\
    \nonumber
    &\quad + \frac1{2}\int_0^T\norm{x_0}_\Hd^2\d t +
      \frac1{2}\int_0^T t \,\norm{I'}_{L_2(L^2,\Hd)}^2\d t\\
    \begin{split}\label{2eq:Barbu3}
    &\leq \limsup_{\h\to
      0}\left(-\int_0^T\tilde\E\norm{\crc{\tilde X_h}(t)}_\Hd^2\d t\right)
    + \int_0^T\E\int_0^t\sp{\tilde X(s)}{\mu}_\Hd\d s\,\d t\\
    &\quad + \frac1{2}\int_0^T\norm{x_0}_\Hd^2\d t +
    \frac1{2}\int_0^T t \,\norm{I'}_{L_2(L^2,\Hd)}^2\d t.
    \end{split}
  \end{align}
  Using \eqref{2eq:114}, Lemma \ref{2cont-est}, Remark
  \ref{2Rem:ident-Z-X} and the integrability from Lemma \ref{2Prohorov-Skorohod}, we obtain
  \begin{align*}
    \eqref{2eq:Barbu3}
    &\leq -\frac1{2}\int_0^T\tilde\E\norm{\tilde
      X(t)}_\Hd^2\d t + \int_0^T\E\int_0^t\sp{\tilde
      X(s)}{\mu}_\Hd\d s\,\d t\\
    &\quad + \frac1{2}\int_0^T\norm{x_0}_\Hd^2\d t +
      \frac1{2}\int_0^T t \norm{I'}_{L_2(L^2,\Hd)}^2 \d t \\
    &= \int_0^T\tilde\E \int_0^t\sp{\tilde X(s)}{\tilde
      Y(r)}_{L^2}\d s\,\d t = \tilde\E \int_0^T(T-t)\sp{\tilde X(t)}{\tilde
      Y(t)}_{L^2}\d t,
  \end{align*}
  which finishes the proof.
\end{proof}

\begin{proof}[Proof of Theorem \ref{2main-thm}.]
  By Lemma \ref{2Prohorov-Skorohod},
  we have that a (nonrelabeled) subsequence of
  $\left(\lc{\tilde X_h}\right)_{h>0}$ converges to $\tilde X$ weakly
  in $L^2([0,T]; L^2)$ and weakly* in $L^\infty([0,T]; \Hd)$,
  $\tilde\P$-almost surely, which implies by the Slutsky theorem (\cf
  \cite[Theorem 13.18]{Klenke}) that
  \begin{displaymath}
    \cL\left(\lc{\tilde X_h}\right) \to \cL(\tilde X)
  \end{displaymath}
  with respect to the weak topology in $L^2([0,T]; L^2)$ and the weak*
  topology in $L^\infty([0,T]; \Hd)$. Since we also have by Lemma
  \ref{2Prohorov-Skorohod} that
  $\cL\left(\lc{\tilde X_h}\right) = \cL\left(\lc{X_h}\right)$ in both
  spaces, these convergence results transfer to $\cL\left(\lc{X_h}\right)$.

  We next show that
  $\left( (\tilde \Om, \tilde\cF, (\tilde\cF_t)_{t\in[0,T]}, \tilde\P),
    \tilde X, \tilde W\right)$ as constructed in Lemma
  \ref{2Prohorov-Skorohod} and \eqref{2eq:Def-stoch-basis}, is a weak
  solution to \eqref{2eq:43} in the sense of Definition
  \ref{2Def-weak-soln} belonging to the process $\tilde Y$ given in Lemma
  \ref{2Prohorov-Skorohod}. Considering the definition of the filtration
  $(\tilde\cF_t)_{t\in[0,T]}$, progressive measurability of $\tilde X$ and
  $\tilde Y$ is clear by construction, and $\tilde W$ is
  a cylindrical $\mathrm{Id}$-Wiener process in $L^2$ with respect to
  $(\tilde\cF_t)_{t\in[0,T]}$. Equality \eqref{2eq:111} is proved in Lemma
  \ref{2ident-eqn}. Hence, it only remains to show \eqref{2eq:112}, or,
  equivalently, $(\tilde X, \tilde Y) \in \cA$, which, according to \cite[Corollary 2.4]{Barbu}, can be done by
  proving
  \begin{align}
    \label{2eq:56}
    &\left(\clc{\tilde X_h}, \clc{\tilde Y_h}\right) \in \cA\quad \text{for all
      }h\in (0,1],\\
    &\begin{cases}\label{2eq:57}
      \clc{\tilde X_h} \tow \tilde X \quad\text{in }
      L^2(\tilde\Om\times[0,T]_\mu;L^2),\\
      \clc{\tilde Y_h} \tow \tilde Y \quad\text{in }
      L^2(\tilde\Om\times[0,T]_\mu;L^2),
    \end{cases}
    \\
    \label{2eq:58}\text{and}\quad
    &\limsup_{h\to 0} \tilde\E\int_0^T(T-t)\sp{\clc{\tilde X_h}}{\clc{\tilde
      Y_h}}_{L^2} \d t \leq \tilde\E\int_0^T(T-t)\sp{\tilde X}{\tilde Y}_{L^2}\d t.
  \end{align}

  \textit{Ad \eqref{2eq:56}:} We notice that by Lemma \ref{2Prohorov-Skorohod} and
  Definition \ref{2Def-F}, we have $\tilde\P$-almost surely
  \begin{displaymath}
    \tilde Y_h = \tilde\phi(\tilde X_h),
  \end{displaymath}
  and hence
  \begin{equation}\label{2eq:113}
    \clc{\tilde Y_h} = \tilde\phi\left(\clc{\tilde X_h}\right) \in \phi\left(\clc{\tilde X_h}\right)
  \end{equation}
  $\tilde\P$-almost surely in $L^2([0,T]; L^2)$. By \cite[Korollar
  V.1.6]{Elstrodt}, this implies that \eqref{2eq:113} is satisfied for
  almost every $(\omega, t,x)\in\tilde\Om\times[0,T]\times[0,1]$, which is
  equivalent to \ref{2eq:56}.
  
  \textit{Ad \eqref{2eq:57}:} By Lemma \ref{2Prohorov-Skorohod}, we have
  \begin{equation}\label{2eq:87}
    \clc{\tilde X_h} \tow \tilde X\quad \text{and} \quad
    \clc{\tilde Y_h}\tow \tilde Y\quad \text{in }L^2(\tilde\Om; L^2([0,T]; L^2))
  \end{equation}
  for $h\to 0$. Furthermore,
  for $\zeta \in L^2(\tilde\Om\times[0,T]_\mu; L^2)$, we note that
  \begin{displaymath}
    \tilde\E\int_0^T\norm{(T-t)\zeta}_{L^2}^2\d t \leq T\,
    \tilde\E\int_0^T(T-t)\norm{\zeta}_{L^2}^2\d t =
    T\norm{\zeta}_{L^2(\tilde\Om\times[0,T]_\mu; L^2)}^2,
  \end{displaymath}
  which yields that $(T-t)\zeta\in L^2(\tilde\Om\times[0,T]; L^2)$. Thus,
  for $h\to 0$, we have
  \begin{align*}
    &\tilde\E\int_0^T\sp{\clc{\tilde X_h}(t)}{\zeta(t)}_{L^2}\mu(\d t) = \tilde\E \int_0^T\sp{\clc{\tilde X_h}(t)}{(T-t)\zeta(t)}_{L^2}\d t\\
    &\to \tilde\E\int_0^T\sp{\tilde X(t)}{(T-t)\zeta(t)}_{L^2}\d t = \tilde\E\int_0^T\sp{\tilde
      X(t)}{\zeta(t)}_{L^2} \mu(\d t),
  \end{align*}
  as required. For $\tilde Y$, an analogous calculation applies.
  
  \textit{Ad \eqref{2eq:58}:} This is proved in Lemma
  \ref{2limsup-diagonal}.
  
  The same course of arguments also applies to any subsequence of
  $(h_m)_{m\in\N}$, which means that each subsequence of
  $(\lc{X_h})_{h>0}$ contains a subsubsequence converging in law to a
  weak solution of \eqref{2eq:43}. Since every weak solution to
  \eqref{2eq:43} is distributed according to the same law by Theorem
  \ref{2uniqueness-thm}, each of these subsubsequences converges in law to
  the same limit, which implies convergence in law of the whole
  sequence. This completes the proof.
\end{proof}

\section{Continuum limit for the deterministic BTW model}\label{2sec:det-analysis}

Towards the second main result, we still use
Assumption \ref{2ass-tau-h}. For each $m\in\N$, we then define
$(u_{h_m}^n)_{n\in\{0,\dots, N_m +1\}}  \subset \R^{Z_m-1}$ iteratively by
\begin{align}\label{eq:BTWdiscr}
  \begin{split}
    u_{h_m}^{n+1} &= u_{h_m}^n + \tau_m \Delta_{h_m}\tilde\phi_1(u_{h_m}^n)\quad\text{for }n=0,\dots, N_m,\\
    u_{h_m}^0 &= u_{h_m}^*,
\end{split}
\end{align}
where $(u_{h_m}^*)_{m\in\N}\subset \R^{Z_m-1}$ such that
$\pcx{(u_{h_m}^*)} \to u_0$ in $L^2$ for $m\to\infty$ for some
$u_0\in L^2$. The process in \eqref{eq:BTWdiscr} then is the deterministic
part of the one-dimensional version of the BTW model as introduced
above. Its
scaling limit candidate is the singular-degenerate partial differential
equation
\begin{align}\label{2eq:150}
  \begin{split}
    \partial_t u(t) &\in \Delta(\phi_1(u(t)),\\
    u(0) &= u_0,
  \end{split}
\end{align}
on a bounded interval $(0,1)\subset\R$ with zero Dirichlet boundary
conditions, where $\phi_1: \R\to 2^\R$ is the maximal monotone extension
of $\tilde\phi_1$ (see \eqref{1eq:148-new}). Furthermore, let
\begin{equation}\label{2eq:Def-psi}
  \psi: \R\to [0,\infty),\quad \psi(x) = \int_0^x \tilde\phi_1(y)\d y =
  \Ind{\R\setminus[-1,1]}(x) (\abs{x} - 1),
\end{equation}
and $\vphi: \Hd \to [0,\infty)$,
\begin{equation}\label{2eq:Def-vphi}
  \vphi(u)=
  \begin{cases}
    \norm{\psi(u)}_{TV},\quad & \text{if }u\in \cM\cap\Hd,\\
    +\infty, & \text{else,}
  \end{cases}
\end{equation}
where the precise definition
of the convex functional of a measure is given in
\cite{Neuss-SVI-SD}. We then define the following notion of solution, which is
a special case of a stochastic variational inequality (SVI) solution (\cf
\cite{Neuss-SVI-SD} for a more detailed analysis). In the spirit of
Clément \cite{Clement}, we will refer to this as an EVI (evolution
variational inequality) solution.

\begin{Def}[EVI solution] \label{2Def-VI-soln} Let $u_0\in\Hd$, $T>0$. We
  say that $u\in \cC([0,T];\Hd)$ is an EVI solution to \eqref{2eq:150}
  if the following conditions are satisfied:
  \begin{enumerate}[label=(\roman*)]
  \item (Regularity)
    \begin{displaymath}
      \vphi(X)\in L^1([0,T]).
    \end{displaymath}
  \item (Variational inequality) For each
    $G\in L^2([0,T];\Hd)$, and $Z\in L^2([0,T];L^2)\cap \cC([0,T];\Hd)$ solving the
    equation
    \begin{displaymath}
      Z(t)-Z(0) = \int_0^tG(s)\,\d s \quad\text{for
        all }t\in[0,T],
    \end{displaymath}
    we have
    \begin{align}
      \begin{split}
        &\norm{u(t)-Z(t)}_\Hd^2 + 2\int_0^t\vphi(u(r))\d r\\
        &\leq \norm{u_0-Z(0)}_\Hd^2 + 2\int_0^t\vphi(Z(r))\d r\\
        &\quad-2\,\int_0^t\sp{G(r)}{u(r)-Z(r)}_\Hd\d r,
      \end{split}
    \end{align}
    for almost all $t\in[0,T]$.
  \end{enumerate}
\end{Def}

\begin{Rem}
  The existence and uniqueness of solutions to \eqref{2eq:150} in this sense is
  shown in \cite{Neuss-SVI-SD} with noise coefficient chosen to
  be zero. Furthermore, note that the previous definition contains a slight
  abuse of notation. In closer analogy to \cite{Clement}, a solution in
  this sense would be
  called an integral solution to the EVI \eqref{2eq:150}.
\end{Rem}

 Then, we have the
following result, which will be proved at the end of Section \ref{2sec:det-analysis}.
\begin{thm}\label{2BTW-thm}
  Recall the notation from Section \ref{2sec:notation} and let Assumption
  \ref{2ass-tau-h} be satisfied. Then, the process $\lc{u_{h_m}}$ obtained
  from \eqref{eq:BTWdiscr} converges weakly* to the EVI solution of
  \eqref{2eq:150} in $L^\infty([0,T]; \Hd)$ for $m\to\infty$.
\end{thm}

\subsection{A priori estimates and transfer to extensions}

As in the previous section, we keep the convention of dropping the index
$m$ of the discretization sequences
\begin{displaymath}
(h_m)_{m\in\N}, (Z_m)_{m\in\N}, (\tau_m)_{m\in\N}, (N_m)_{m\in\N},
\end{displaymath}
writing instead $(h)_{h>0}$ etc. Moreover, convergence of
sequences and usually nonrelabeled subsequences indexed by $h_m$ for
$m\to\infty$ will be denoted by $h\to 0$. Finally, we will drop the index in
$\phi_1$, hence $\tilde\phi$ denotes the BTW nonlinearity given in \eqref{1eq:148-new} and $\phi$ its
maximal monotone extension.

In oder to obtain convergent subsequences by compactness arguments, we
use a very similar strategy as in Section
\ref{2sec:constr-weak-solution}. Hence, we will often refer to the
proofs of the corresponding lemmas.

\begin{lem}\label{2BTWdiscr-energy-est}
  Let $\tau, h> 0$ and $Z, N\in\N$ as in Assumption \ref{2ass-tau-h}, where we
  choose $h$ small enough for $\frac{\tau}{h^2} \leq \frac1{4}$
  to be satisfied. Let $(u_h)_{h\geq 0}$ be the discrete process defined in
  \eqref{eq:BTWdiscr}. Then,
  \begin{displaymath}
    \max_{n\in\{0,\dots,N+1\}} \norm{u_h^n}_{-1}^2 \leq \norm{u_h^*}_{-1}^2.
  \end{displaymath}
\end{lem}

The proof of Lemma \ref{2BTWdiscr-energy-est} is conducted by the same
arguments as the proof of Lemma \ref{2discr-energy-est}, using
\begin{displaymath}
  \sp{x}{\tilde\phi(x)}_0  \geq \norm{\tilde\phi(x)}_0^2\quad \text{for }x\in\R^{Z-1}
\end{displaymath}
instead of \eqref{2eq:phitrick}.

We have the following stronger version of Lemma
\ref{2discr-cnty-estimate} due to the boundedness of the BTW nonlinearity.
\begin{lem}\label{2BTWdiscr-cnty-estimate}
  Let $\tau, h> 0$ as in Assumption \ref{2ass-tau-h}.
  Then, the discrete process in (\ref{eq:BTWdiscr}) satisfies
  \begin{displaymath}
    \norm{u_h^{n+1} - u_h^n}_{-1}^2  \leq 4\frac{\tau^2}{h^2}\quad
    \text{for all }n\in\{0,\dots,N-1\}.
  \end{displaymath}
\end{lem}
\begin{proof}
  Using Lemma \ref{2lem-spectralnorm}, we compute for $n\in\{0,\dots,N-1\}$
  \begin{displaymath}
    \norm{u_h^{n+1} - u_h^n}_{-1}^2
    = \norm{\tau \Delta_h\tilde\phi(X_h^n) }_{-1}^2 \leq
    \tau^2 \norm{-\Delta_h}\E\norm{\tilde\phi(X_h^n)}_0^2 \leq 4\frac{\tau^2}{h^2},
  \end{displaymath}
  using the boundedness of $\tilde\phi$ in the last step.
\end{proof}

Again, the previous estimates can be transfered to the continuous setting
by Lemma \ref{2est-H-1}.

\begin{Cor}\label{2BTWcont-energy-Cty}
  Let $\tau, h> 0$ and $Z, N\in\N$ as in Assumption \ref{2ass-tau-h}, where we
  choose $h$ small enough for $\frac{\tau}{h^2} \leq \frac1{4}$
  to be satisfied. Let $(u_h)_{h\geq 0}$ be the discrete process defined in
  \eqref{eq:BTWdiscr}. Then, there exists a positive constant $C$ independent of $h$, such that
  \begin{equation}\label{2eq:control-by-init}
    \max\left\{\esssup_{t\in[0,T]}\norm{\lc{u_h}(t)}_\Hd^2,
      \esssup_{t\in[0,T]}\norm{\clc{u_h}(t)}_\Hd^2\right\} \leq
    \norm{u_h^*}_{-1} \leq C
  \end{equation}
  for $h>0$. Moreover,
  \begin{equation}\label{2eq:BTW-stability-est}
    \esssup_{t\in[0,T]}\norm{\lc{u_h}(t) - \clc{u_h}(t)}_\Hd^2 \leq C\frac{\tau^2}{h^2}
  \end{equation}
  for $h>0$.
\end{Cor}

\subsection{Extraction of convergent subsequences}

\begin{lem}\label{2BTW-extract-subseq}
  Let $\tau, h> 0$ and $Z, N\in\N$ as in Assumption \ref{2ass-tau-h}, and
  let $(u_h)_{h\geq 0}$ be the discrete process defined in
  \eqref{eq:BTWdiscr}. Then, there exists $u\in L^\infty([0,T]; \Hd)$ and a nonrelabeled subsequence such
  that
  \begin{displaymath}
    \lc{u_h} \tows u\quad \text{and}\quad \clc{u_h}\tows u
  \end{displaymath}
  for $h\to 0$.
\end{lem}

\begin{proof}
  The existence of $u\in\Hd$ and a nonrelabeled subsequence such that
  $\lc{u_h}\tows u$ for $h\to 0$ follows by the Banach-Alaoglu theorem and
  the fact that convergence with respect to the weak* topology on the dual
  of a normed space is equivalent to weak* convergence (\cf
  \cite[Proposition A.51]{Fonseca}). From this subsequence, the same
  argument allows to extract another subsequence such that
  $\clc{u_h}\tows \tilde u$ for some $\tilde u\in L^\infty([0,T]; \Hd)$.
  Using \eqref{2eq:BTW-stability-est}, one shows that 
  $u=\tilde u$, which
  finishes the proof.
\end{proof}

\subsection{Identification of the limit as a solution and proof of Theorem
  \ref{2BTW-thm}}

\begin{Def}
Let $h>0$ and $Z\in\N$ as in Assumption \ref{2ass-tau-h}. We then define the
functional $\vphi_h: \R^{\gridlim-1}\to [0,\infty)$ by
\begin{displaymath}
  \vphi_h(w_h) = \sum_{i=1}^{\gridlim-1} h\,\psi(w_{h,i}),
\end{displaymath}
where $\partial\psi = \phi$ as defined in \eqref{2eq:Def-psi}.
\end{Def}

\begin{Rem}\label{2:compatibility-vphih-vphi}
  We note that $\vphi_h(w_h) = \vphi(\pcx{w_h})$, where $\vphi$ is defined
  as in \eqref{2eq:Def-vphi}. Furthermore, using that $\tilde\phi \in
  \partial\psi$, one can verify that
  \begin{equation}\label{2eq:discr-subdiff}
    -\Delta_h\tilde\phi(w_h)\in \partial_{-1}\vphi_h(w_h),
  \end{equation}
  where $\partial_{-1}$ denotes the subdifferential with respect to the
  inner product $\sp{\cdot}{\cdot}_{-1}$.
\end{Rem}

\begin{lem}\label{2lem:SVI-discr}
  Let $h>0$ and $Z\in\N$ as in Assumption \ref{2ass-tau-h}. Let
  $v_h\in \cC([0,T]; \R^{Z-1})$ be almost everywhere differentiable,
  $\partial_t v_h \in L^2([0,T]; \R^{Z-1})$ and $u_h$ be defined as in
  \eqref{eq:BTWdiscr}. For all $t\in[0,T]$, we then have
\begin{align}\label{2eq:8}
  \begin{split}
    \norm{v_h(t) - \plt{u_h}(t)}_{-1}^2 \leq& \norm{v_h(0) -
      u_h^*}_{-1}^2 + 2\int_0^t \vphi_h(v_h(r))\d r -
    2\int_0^t\vphi_h(\pctl{u_h}(r))\d r\\
    &+ 2\int_0^t\sp{v_h(r) - \plt{u_h}(r)}{\partial_t v_h(r)}_{-1}\d r\\
    &+ 2\int_0^t\sp{\pctl{u_h}(r) -
      \plt{u_h}(r)}{-\Delta_h\tilde\phi(\pctl{u_h}(r))}_{-1}\d r.
  \end{split}
\end{align}
\end{lem}
\begin{proof}
  This follows by the construction of $u_h$, the chain rule and \eqref{2eq:discr-subdiff}.
\end{proof}

\begin{prop}\label{2prop-Sobolev-VI}
  Let
  \begin{displaymath}
    v\in W^{1,2}(0,T; L^2, \Hd) := \left\{v\in L^2([0,T]; L^2)| \partial_t
      v \in L^2([0,T]; \Hd)\right\},
  \end{displaymath}
  and $u\in L^\infty([0,T]; \Hd)$ be the limit
  process of $(u_h)_{h>0}$ as in Lemma \ref{2BTW-extract-subseq}. Then
  \begin{align}\label{2eq:VI-for-limit}
  \begin{split}
    \norm{v(t) - u(t)}_\Hd^2 + 2\int_0^t\vphi(u(r))\d r
    \leq& \norm{v(0) -
      u(0)}_\Hd^2 + 2\int_0^t \vphi(v(r))\d r \\
    &+ 2\int_0^t\sp{v(r) - u(r)}{\partial_t v(r)}_\Hd\d r
  \end{split}
  \end{align}
  for almost all $t\in[0,T]$.
\end{prop}

\begin{proof}
  \textit{Step 1:} We first show the statement for $v\in
  \cC^1([0,T];L^2)$. Let Assumption \ref{2ass-tau-h} be satisfied. To show
  \eqref{2eq:VI-for-limit}, we aim to pass to the limit in (\ref{2eq:8})
  for a subsequence $h\to 0$ realizing the convergence in Lemma
  \ref{2BTW-extract-subseq}, using a sequence
  $(v_h)_{h>0}\subset \cC([0,T]; \R^{Z-1})$ such that for all $h>0$, $v_h$
  is differentialble in time almost everywhere and
  \begin{equation}\label{2eq:smooth-approx}
    \pcx{v_h}\to v\quad\text{and}\quad \pcx{(\partial_t v_h)}\to \partial_t
    v\quad \text{in }L^2([0,T]; L^2).
  \end{equation}
  Such a sequence can be constructed by using the density of
  $\cC_c^0([0,T]\times[0,1])$ in $L^2([0,T]; L^2)$ and the fact that each
  compact set in $[0,1]$ is included in $\supp(\pcx{\cS_h})$ for $h$ small
  enough. Note that Lemma \ref{2lem:SVI-discr} applies to $v_h$, since
  \eqref{2eq:smooth-approx} implies that $\pcx{(\partial_t v_h)}$ is
  bounded in $L^2([0,T]; L^2)$ and hence
  \begin{displaymath}
    \int_0^T \norm{\partial_t v_h}_0^2\d t = \int_0^T\norm{\pcx{(\partial_t
        v_h)}}_{L^2}^2\d t < \infty
  \end{displaymath}
  by the isometry in Remark \ref{2Rem:isometry}. Then, integrating
  \eqref{2eq:8} against $\gamma\in L^\infty([0,T])$ yields
  \begin{align}\label{2eq:integrated-VI-discr}
    \begin{split}
      &\int_0^T\gamma(t)\norm{v_h(t) - \plt{u_h}(t)}_{-1}^2\d t +
      2\int_0^T\gamma(t)\int_0^t \vphi_h(\pctl{u_h}(r))\,\d r\,\d t\\
      &\leq
      \int_0^T\gamma(t)\norm{v_h(0) - u_h^*}_{-1}^2\d t +
      2\int_0^T\gamma(t)\int_0^t\vphi_h(v_h(r))\,\d r\,\d t\\
      &\quad + 2\int_0^T\gamma(t)\int_0^t\sp{v_h(r) -
        \plt{u_h}(r)}{\partial_t v_h(r)}_{-1}\d r\,\d t\\
      &\quad + 2\int_0^T\gamma(t) \int_0^t\sp{\pctl{u_h}(r) -
        \plt{u_h}(r)}{-\Delta_h\phi(\pctl{u_h}(r))}_{-1}\d r\,\d t.
    \end{split}
  \end{align}
  We treat each term in \eqref{2eq:integrated-VI-discr} separately. For the
  first term, we use the lower-semicontinuity of the norm, the convergence
  from Lemma \ref{2BTW-extract-subseq} and the construction of $v_h$. For
  the second term, we use the lower-semicontinuity of $\vphi$, as proven in
  \cite[Proposition 3.1]{Neuss-SVI-SD}, in a weighted space arising from
  Fubini's theorem similar to \eqref{2eq:Fub}. The third term can be
  treated as in \eqref{2eq:init-val-ident}. For the fourth term, we exploit
  that $v$ and $v_h$ are $L^1$ functions, which allows to use the formula
  \begin{displaymath}
    \vphi(v) = \int_0^1 \psi(v(x))\,\d x.
  \end{displaymath}
  The Lipschitz continuity of $\psi$ then allows to pass to the limit. The
  fifth term can be treated by the dominated convergence theorem in the
  outer integral and Proposition \ref{2conv-pc} for the inner integral.
  
  In order to treat the last term, we use Estimate
  \eqref{2eq:normest}, Lemma \ref{2BTWdiscr-cnty-estimate} and the
  boundedness of $\tilde\phi$ to obtain
  \begin{align}\label{2eq:17}
    \begin{split}
      &\abs{\int_0^T\gamma(t)\int_0^t\sp{\pctl{u_h}(r) -
          \plt{u_h}(r)}{-\Delta_h\tilde\phi(\pctl{u_h}(r))}_{-1}\d r\,\d t}\\
      &\leq \int_0^T\gamma(t)\,\d t \sup_{t\in[0,T]}\norm{\plt{u_h}(t) -
        \pctl{u_h}(t)}_{-1} \int_0^T
      \norm{-\Delta_h\tilde\phi(\pctl{u_h}(r))}_{-1}
      \d r\\
      &\leq \int_0^T \gamma(t)\,\d t\, 2\frac{\tau}{h}
      \int_0^T\frac{2}{h}\norm{\tilde\phi(\pctl{u_h}(r))}_0 \d r \leq
      \int_0^T\gamma(t)\,\d t\,4T \frac{\tau}{h^2} \to 0
    \end{split}
  \end{align}
  for $h\to 0$. Hence, taking $\liminf_{h\to 0}$ in
  \eqref{2eq:integrated-VI-discr}, we obtain
  \begin{align}\label{2eq:integrated-VI-cts}
    \begin{split}
      &\int_0^T\gamma(t)\norm{v(t) - u(t)}_\Hd\d t +
      2\int_0^T\gamma(t)\int_0^t \vphi(u(r))\,\d r\,\d t\\
      &\leq
      \int_0^T\gamma(t)\norm{v(0) - u(0)}_\Hd\d t +
      2\int_0^T\gamma(t)\int_0^t\vphi(v(r))\,\d r\,\d t\\
      &\quad + 2\int_0^T\gamma(t)\int_0^t\sp{v(r) -
        u(r)}{\partial_t v(r)}_\Hd\d r\,\d t,
    \end{split}
  \end{align}
  and since $\gamma\in L^\infty([0,T])$, $\gamma\geq 0$, was chosen
  arbitrarily, \eqref{2eq:VI-for-limit} follows for $v\in\cC^1([0,T]; L^2)$.
  
  \textit{Step 2:} In order to extend the statement to Sobolev functions
  $v$, recall that $\cC^1([0,T]; L^2)$ is dense in $W^{1,2}(0,T;L^2,\Hd)$
  with respect to the norm
  \begin{displaymath}
    \norm{u}_{W^{1,2}(0,T;L^2,\Hd)}^2 = \norm{u}_{L^2([0,T];L^2)}^2 +
    \norm{\partial_t u}_{L^2([0,T];\Hd)}^2
  \end{displaymath}
  according to \cite[Theorem 2.1]{Lions-Magenes}, and that the embedding
  \begin{displaymath}
    W^{1,2}(0,T; L^2,\Hd)\hookrightarrow \cC([0,T];\Hd)
  \end{displaymath}
  is continuous by \cite[Theorem 3.1]{Lions-Magenes}. Thus, we obtain the
  full statement by an approximaion agument.
\end{proof}

\begin{proof}[Proof of Theorem \ref{2BTW-thm}]
  For each sequence $(h_m)_{m\in\N}$ satisfying Assumption
  \ref{2ass-tau-h}, Lemma \ref{2BTW-extract-subseq} provides a subsequence
  denoted by $h\to 0$ and $u\in L^\infty([0,T];\Hd)$, such that
  $u_h\tows u$ in $L^\infty([0,T];\Hd)$ for $h\to 0$. Proposition
  \ref{2prop-Sobolev-VI} then implies that $u$ satisfies the variational
  inequality in Definition \ref{2Def-VI-soln}. Revisiting the uniqueness
  argument in \cite{Neuss-SVI-SD}, we see that the continuity of the
  solution is not needed by using an almost-everywhere version of
  Gronwall's inequality (see \eg \cite[Theorem 1.1]{Webb}),
  such that $u$ can be identified as a $\d t$ version of the EVI solution to
  \eqref{2eq:150}.
  By a standard contradiction argument, we obtain that the whole sequence
  $(u_{h_m})_{m\in\N}$ converges to this solution, which finishes the proof.
\end{proof}

\section{Numerical experiments}\label{sec_num} 

In this section we perform numerical simulations to illustrate the features of the two discrete (stochastic) models \eqref{1eq:2-new} in one and two spatial dimensions, as well as the convergence to the limiting SPDEs. We strive here for simulations going beyond the setting of the results proven in the previous sections, by (a) estimating a rate of convergence, (b) investigating the convergence under more general assumptions, for example by relaxing the strong CFL condition, and (c) by considering higher spatial dimension. In addition, we verify the validity of the fundamental power law scaling of avalanche sizes (see, e.\,g.\,\cite{BTW88}) for the weakly driven BTW/Zhang model \eqref{1eq:2-new} considered in this work.

More precisely, recall that in the proof of the convergence of the discrete
Zhang dynamics to the solutions of the corresponding SPDE, the strong CFL condition \eqref{1eq:CFL} was assumed in the sense that $\tau = o(h^2)$.
In this section, we relax this assumption by choosing $\tau = h^2$ in all simulations below, and still
empirically observe convergence (with rates). 

For a given mesh
size $h>0$ (note $\tau = h^2$) throughout this section, we let $X_{h}^n\equiv X_{h}^{n, \cdot}$ be the
discrete solution at time $n$ (cf. \eqref{1eq:150-new}, \eqref{2eq:88}), and in order to simplify the presentation,
with a slight abuse of notation, we interpret the numerical solution as a
grid function writing $X_{h}^n(z_j) = X_{h}^{n,j}$ for $z_j =
jh$, whenever the meaning becomes clear from the context.

The simulations below are performed for a slightly more general discrete
model than (\ref{2eq:88}), by introducing the free parameters $D$ and
$\sigma$ in front of the diffusion $\Delta_h$ and the noise in
\eqref{2eq:88} respectively. More precisely, for $h=\frac1{M}, \tau = h^2$, we set $N=T/\tau$, $X_{h}^0 = x_{h}^0$ and compute 
\begin{equation}\label{scheme1}
  X_{h}^{n} = X_{h}^{n-1} + \tau D \Delta_{h}\tilde\phi\left(X_{h}^{n-1}\right) + \tau \mu  + \sigma \sqrt{\frac{\tau}{h}}\xi_{h}^{n},\quad\text{for }n=1,\dots, N,
\end{equation}
with zero boundary conditions $X_{h}^{n,0}= X_{h}^{n,M}=0$, for all  $n > 0$.
In what follows, we consider both the case $\tilde\phi = \tilde\phi_{1}$
(BTW model) and the case $\tilde\phi =
\tilde\phi_{2}$ (Zhang model).

The random variables $\xi_{h}^{n} = (\xi^{n,j}_h)_{j=1}^M$ in \eqref{scheme1} are chosen to 
be either i.\,i.\,d.\, $\cN(0,1)$-distributed or i.\,i.\,d.\,Bernoulli distributed with values $\{-1,1\}$
attained with equal probability $\frac{1}{2}$.
Since, in most simulations below, the choice of the noise had only minor effect on the results,
we only mention the concrete choice of the noise where relevant.

The discrete model \eqref{scheme1} can be regarded as an approximation of the SPDE 
\begin{align}\label{eq_zhang2}
\d X(t) & = D \Delta \tilde\phi\big(X(t)\big) + \mu\d t + \sigma \d W(t),
\\\nonumber
X(0) & = x_0.
\end{align}

\subsection{Rate of convergence} 

In this section, we present numerical simulations for the rate of convergence
of the discrete approximation (\ref{scheme1}) to the stochastic Zhang and the stochastic BTW PDE with space-time white noise, that is, (\ref{eq_zhang2}) in spatial dimension $d=1$. Since we use $\tau = h^2$, 
the experimental results go beyond the strong CFL condition \eqref{1eq:CFL}. 

For the following numerical simulations, we consider the initial data
\begin{equation*}
x_0(z) = 
\left\{
\begin{split}
& 1 \qquad && \text{ for } z\in[0.2,0.4]\cup[0.6,0.8]\,,
\\
& 0 \qquad && \text{ otherwise},
\end{split}
\right.
\end{equation*}
and choose the remaining parameters in \eqref{scheme1} as $K=1$, $D=0.01$,
$\mu=0.1$; the choice of $\sigma$ will be specified below.

We measure the pathwise convergence of the error in the (discrete)
$H^{-1}$ and $L^{2}$ norms
as defined below, evaluated at the final time
$T=0.1$. Since no explicit solution is known in the
stochastic setting, we examine the error with respect to a reference solution which is
computed on a fine mesh with mesh size $\hmin=1/{\tilde M}$ for
$\tilde{M}= 12800$ (and $\taumin = \hmin^2$). To construct realizations of the noise consistently for all discretization levels we consider
realizations of the random variable $\xi_{\hmin} = \{\xi^{n,j}_{\hmin}\}_{j,n=1}^{\tilde M-1, \tilde M^2}$ on the fine space-time grid.
The noise on the coarser levels is constructed as follows.
The random variables $\xi_{\hmin}$ on the fine grid  can be interpreted as (unscaled) Wiener increments
of a discrete (piecewise constant in time and space over the fine space-time partition with steps-sizes $\hmin$, $\taumin$) 
space-time Brownian sheet   
$ \int_0^{\tilde{t}_n} \int_0^{\tilde{z}_j} \d \tilde{W}_{\hmin} (s, x) = \sum_{k=1}^n \sum_{\ell=1}^j \sqrt{\taumin \hmin}\xi_{\hmin}^{k,\ell}$, cf. \eqref{2eq:45}.
Hence, on the coarse grid with $h  = 1/M = L/\tilde M$, $L\geq1$  we construct the increments of the space-time white noise as
\begin{equation}\label{w_interpol}
\sqrt{\frac{\tau}{h}}\xi^{n,j}_h := \frac{1}{h} \sum_{k=L^2(n-1)+1}^{L^2n} \sum_{\ell=L(j-1)+1}^{Lj} \sqrt{\taumin \hmin}\xi_{\hmin}^{k,\ell} := \frac{1}{h}\int_{t_{n-1}}^{t_n} \int_{z_{j-1}}^{z_j} \d \tilde{W}_{\hmin} (s, x).
\end{equation}
We observe that $\mathbb{E} \sqrt{\frac{\tau}{h}}\xi^{n,j}_h = 0$ and $\mathbb{E}\left(\sqrt{\frac{\tau}{h}}\xi^{n,j}_h\right)^2 = \frac{\tau}{h}$.

For a mesh with mesh size $h$ we let $X_h$ be the piecewise linear
interpolant of the corresponding numerical solution.  For two numerical
solutions $X_{\hmin}$, $X_h$, computed over partitions with mesh sizes
$\hmin$, $h > \hmin$, the error in the $L^2$ norm is evaluated as
$\|X_{\hmin} - X_{h}\|_{L^2}^2:= \hmin \sum_{j=0}^{\tilde M}
\big(X_{\hmin}(z_j) - X_{h}(z_j)\big)^2$, $z_j = j \hmin$.
The experimental order of convergence is computed as $\|X_{\hmin} - X_h\|_{-\hmin}^2$, where
$\|X_h \|_{-\hmin} \equiv \|\nabla (-\Delta_{\hmin})^{-1}X_h \|_{L^2}$ is a
discrete approximation of the $H^{-1}$ norm
$\|X_h \|_{-1}=\|\nabla (-\Delta)^{-1} X_h \|_{L^2}$.  Here,
$\Delta_{\hmin}$ denotes the finite difference Laplacian \eqref{2eq:Deltah}
with $h=\hmin$, and, with a slight abuse of notation, the discrete inverse
Laplacian $Z_{\hmin} = (-\Delta_{\hmin})^{-1} X_h$ is defined as the
solution of
$$
-\Delta_{\hmin} Z_{\hmin}(z_j) = X_h(z_j)\qquad z_j = j \tilde h,\ j = 1, \dots,  \tilde M-1,
$$
with homogeneous Dirichlet boundary conditions $Z_{\hmin}(z_0)=Z_{\hmin}(z_{\tilde M}) =0$.

We simulate the discrete system (\ref{scheme1}) for a sequence of nested meshes with $h=1/M$, $M=100,400,1600,3200,6400$ over the time interval $[0,T]$ with $T=0.1$.
In Figure~\ref{fig_path_err_i003} we display the order of convergence in the discrete $H^{-1}$- and $L^2$ norms at the final time for the stochastic Zhang- and BTW model
with Bernoulli noise with intensity $\sigma=0.03$. 
The error is computed pathwise and averaged over $100$ different realizations of the noise $\xi_{\hmin}(\omega_k)$, $k=1,\dots, 100$, that is, 
$\mathbb{E}\|X_{\hmin}^{\tilde{M}^2}- X_h^{M^2}\|_{-\hmin} \approx \frac{1}{100}\sum_{k=1}^{100} \|X_{\hmin}^{\tilde{M}^2}(\omega_k)- X_h^{M^2}(\omega_k)\|_{-\hmin}$ and analogically for the $L^2$ norm.
The convergence plots for the Zhang and the BTW models are graphically indistinguishable.
We observe a linear convergence rate w.r.t. the mesh size $h$ in the $H^{-1}$ norm. In the $L^2$ norm there is no observable convergence rate for coarse mesh sizes, while for small mesh sizes the convergence rate approaches the order $1/4$. 
\begin{figure}[!htp]
\center
\includegraphics[width=0.45\textwidth]{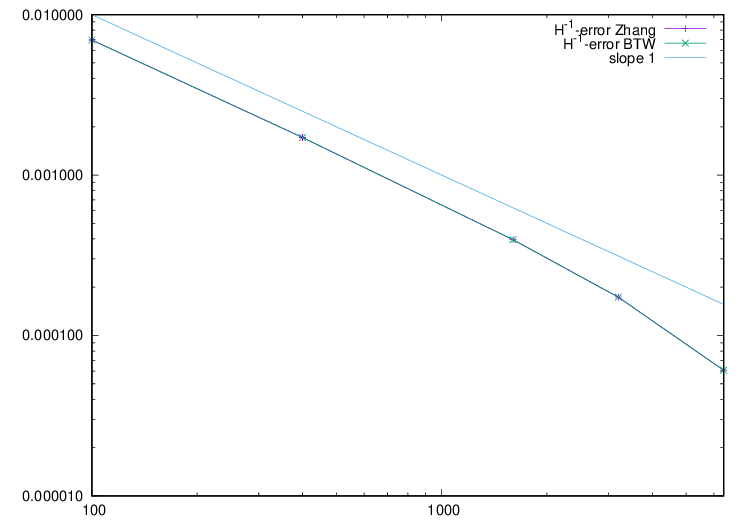}
\includegraphics[width=0.45\textwidth]{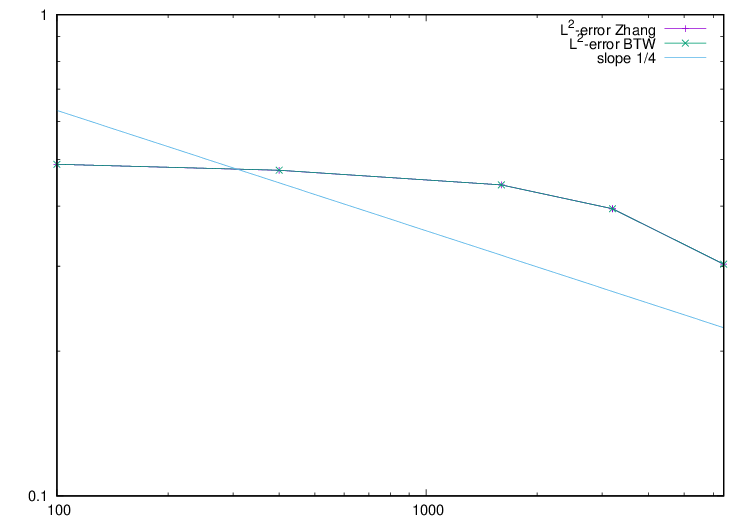}
\caption{Experimental convergence of the error in the $H^{-1}$ norm for $M=100,400,1600,3200,6400$ at time $T=0.1$ for the Zhang and BTW model with $\sigma=0.03$ (left)
and the corresponding error measured in the $L^2$ norm (right).}
\label{fig_path_err_i003}
\end{figure}

This improved convergence behavior for smaller mesh size can be explained by an interplay of the noise intensity with the dissipative effect of the diffusion: Since the variance of the noise scales with $\frac{1}{h}$, the smaller the mesh size, the larger the fluctuations of the noise. Since the diffusion is active only on supercritical sites ($|X_h| > K$) and absent on subcritical sites ($|X_h| < K$), the regularizing effect of the diffusion becomes more pronounced when the values of the solutions are driven towards larger values by an exploding variance of the noise. 

This intuitive explanation suggests that an increase of the noise intensity $\sigma$ should lead to an improved convergence behavior in the $L^2$ norm. This motivates the following simulations: In Figure~\ref{fig_path_err_i017} we display the order of convergence for stronger noise
$\sigma=0.17$, the results are again averaged over $100$ realizations of the
noise.  We observe a linear convergence rate in the $H^{-1}$ norm and a
convergence rate of order $1/4$ in the $L^2$ norm. 
\begin{figure}[!htp]
\center
\includegraphics[width=0.45\textwidth]{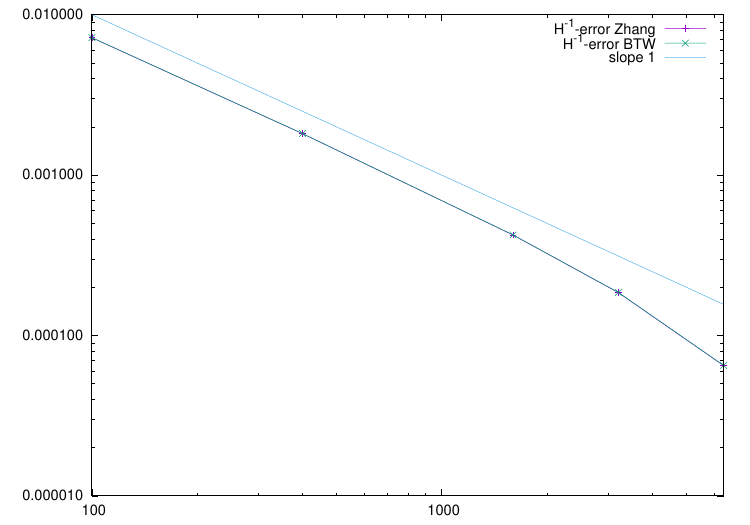}
\includegraphics[width=0.45\textwidth]{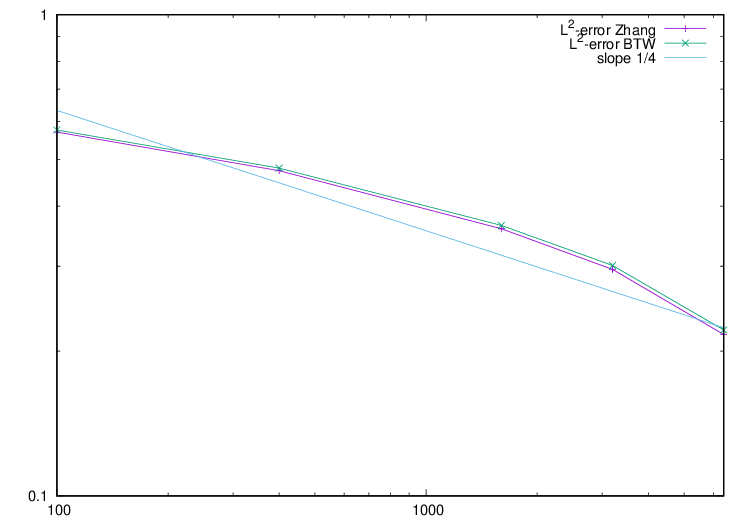}
\caption{Experimental convergence of the error in the $H^{-1}$ norm for $M=100,400,1600,3200,6400$ at time $T=0.1$ for the Zhang and BTW model with $\sigma=0.17$ (left)
and the corresponding error measured in the $L^2$ norm (right).}
\label{fig_path_err_i017}
\end{figure}

In particular, the convergence of the $L^2$ norm improves over the one observed for $\sigma=0.03$, in the sense that the rate of convergence becomes visible also at coarser mesh sizes. 
This improved behavior is consistent with the above explanation of the improved convergence for smaller mesh sizes.

The numerical solution averaged over $100$ realizations of the noise at the final time $T=0.1$, with $M=100,400,1600,3200,6400,12800$ and $\sigma=0.03$, is displayed in Figure~\ref{fig_uh_i003}. 
The solutions of the Zhang and BTW models are graphically indistinguishable.
\begin{figure}[!htp]
\center
\includegraphics[width=0.45\textwidth]{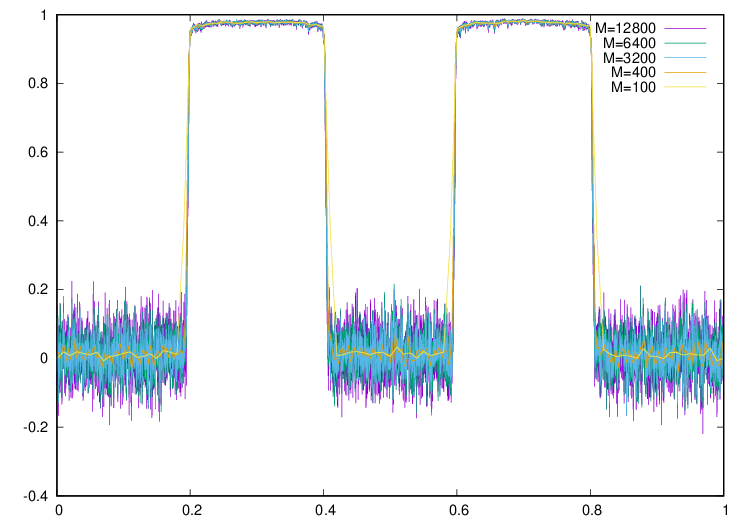}
\includegraphics[width=0.45\textwidth]{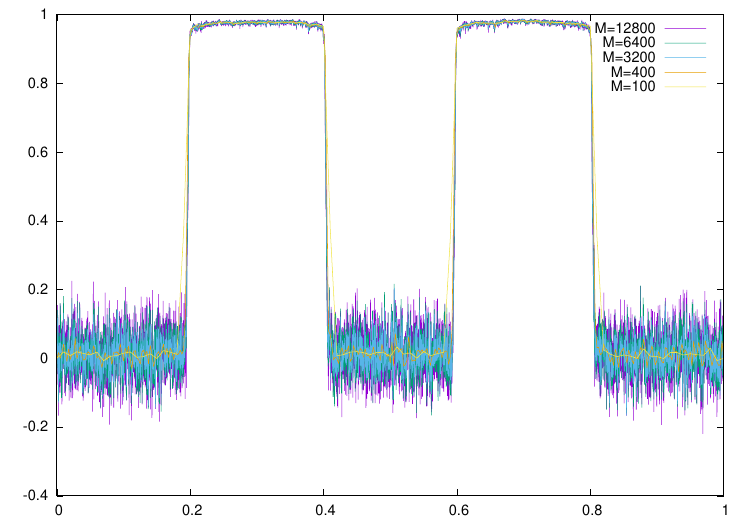}
\caption{Numerical solution $X_h$ of the Zhang model (left) and the BTW model (right) at time $T=0.1$ for $\sigma=0.03$  averaged over $100$ realizations of the noise.}
\label{fig_uh_i003}
\end{figure}
The numerical solutions averaged over $100$ realizations of the noise at the final time $T=0.1$, with $M=100,400,1600,3200,6400,12800$ 
and the stronger noise $\sigma=0.17$, is displayed in Figure~\ref{fig_uh_i017}. The simulations of the Zhang and BTW models are similar but there are noticeable differences
in the regions where $X_h \approx K$ due to the different effects of the respective nonlinearities.
\begin{figure}[!htp]
\center
\includegraphics[width=0.45\textwidth]{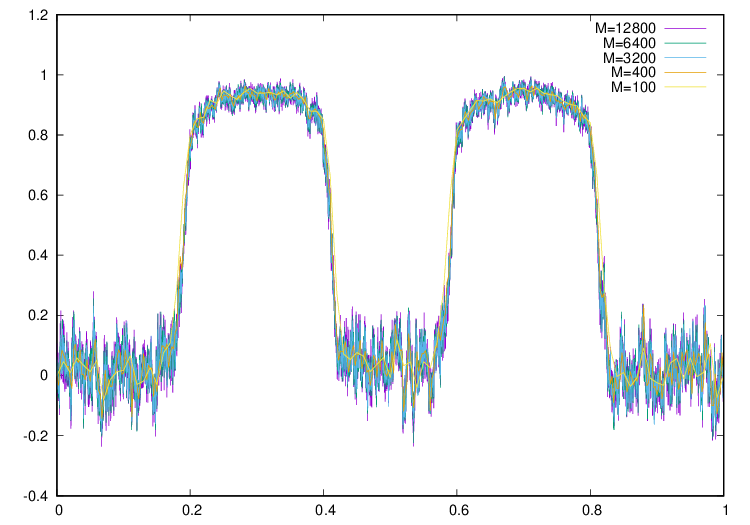}
\includegraphics[width=0.45\textwidth]{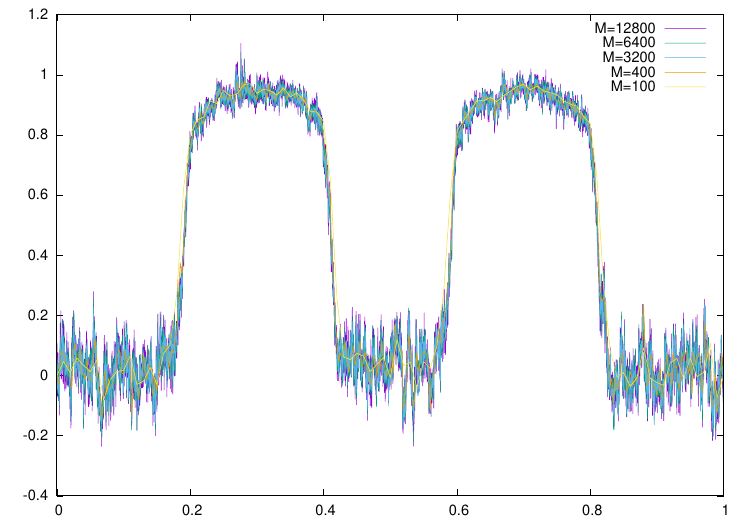}
\caption{
Numerical solution $X_h$ of the Zhang model (left) and the BTW model (right) at time $T=0.1$ for $\sigma=0.17$  averaged over $100$ realizations of the noise.}
\label{fig_uh_i017}
\end{figure}

Numerical solutions of the Zhang model for a single realization of the noise for $\sigma=0.03, 0.17$ are displayed in Figure~\ref{fig_uh_zhang_path}. Again, 
we observe the improved ''smoothing'' effect for stronger noise intensity $\sigma=0.17$.
\begin{figure}[!htp]
\center
\includegraphics[width=0.45\textwidth]{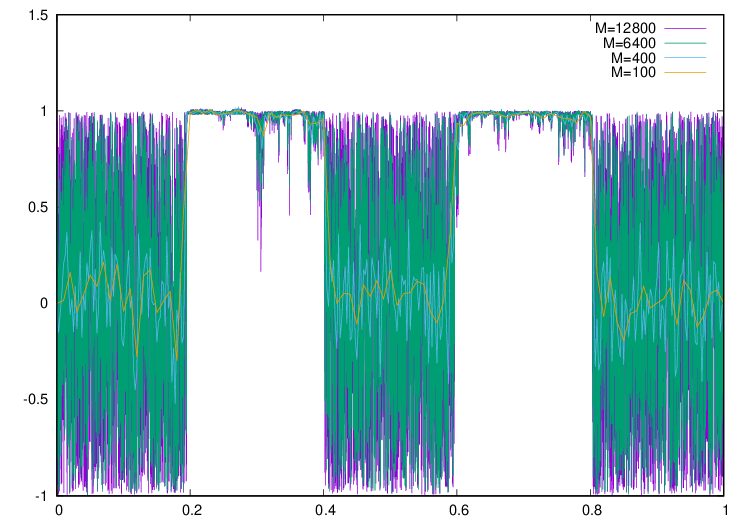}
\includegraphics[width=0.45\textwidth]{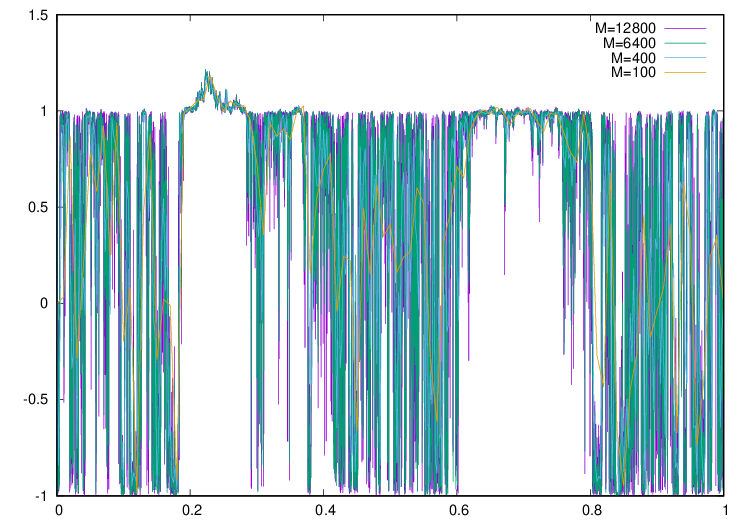}
\caption{Single realization of the numerical solution $X_h$ of the Zhang model for $\sigma=0.03$ (left) and $\sigma=0.17$ (right) at time $T=0.1$ for different mesh size.}
\label{fig_uh_zhang_path}
\end{figure}

\subsection{Power law scaling}\label{sec_num_power}

A fundamental property of the original BTW model is the
observation of power law scalings without explicit tuning of parameters to
critical values, see \cite{BTW88}. In particular, it is observed that the
size/duration of avalanches shows a power law scaling. In this section, we
investigate the validity of such power law scalings for the weakly driven
BTW/Zhang model \eqref{1eq:2-new} and their dependency on the grid size. 

More precisely, we consider the discrete model
\eqref{scheme1} in spatial dimension $d=2$ and, if not mentioned otherwise, we set $K=10$, $D=0.25$ and $\mu=10^{-4}$, $\sigma=0.01$. 
The spatial domain is taken to be a unit square, that is covered by a uniform grid $z_{i,j} = (i h, j h)$ 
with mesh size $h = 1/M$. The discrete  two-dimensional Laplace operator is defined as 
$$\Delta_h X(z_{i,j}) =\frac{1}{h^2} \Big(-4 X(z_{i,j}) + X(z_{i-1,j}) + X(z_{i+1,j}) + X(z_{i,j-1}) + X(z_{i,j+1})\Big).$$

In all simulations below, the initial condition is chosen randomly and subcritical, that is, $\max_j |X_h^{0}(z_j)| \leq K$. This is realized by sampling a random initial condition and subsequently simulating only the deterministic dynamics without forcing until a subcritical state is attained. This state is then taken as the initial condition for the subsequent simulations.

The duration and size of avalanches is defined as follows: Starting from a ``subcritical'' point in time  $n_0 \geq 0$,
that is, $n_0$ such that all sites are subcritical in the sense that
$\max_j |X_h^{n_0}(z_j)| \leq K$, we say that an avalanche occurs at time
point $n_*> n_0$ if $\max_{j}|X_h^{n}(z_j)| \leq K$ for $n_0 \leq n< n_*$
and $\max_{j}|X_h^{n_*}(z_j)| > K$. The duration of this avalanche is then
defined as the number of time-steps until the numerical solution reaches a
subcritical level again, that is, until $\max_j |X^{n^*}_h(z_j)|\leq K$ is
reached for some $n^* > n_*$.  The avalanche size is defined as
$\#\{z_j:\ |X_h^{n}(z_j)| > K\ \text{for some}\ n\in [n_*,n^*]\}$.

Note that in dimension $d=2$ the existence and regularity of solutions of the SPDE \eqref{eq_zhang2}, as well as the convergence of the numerical approximation \eqref{scheme1}
are open problems. Nevertheless, the numerical results reported in this section reveal power-law characteristics which appear to be  preserved for decreasing mesh size.

Figure~\ref{fig_aval_2d_l30}  displays log-scale plots of avalanche sizes and durations
depending on different values of the parameters. We observe that, up to a certain threshold, the values of $\mu$ have a negligible effect on the statistics of the avalanches.
Furthermore, we observe that the noise intensity $\sigma$ mainly influences
the lower range of the size and duration of the avalanches,
i.\,e.\,larger noise variance increases the frequency of smaller
avalanches.

\begin{figure}[!htp]
\center
\includegraphics[width=0.45\textwidth]{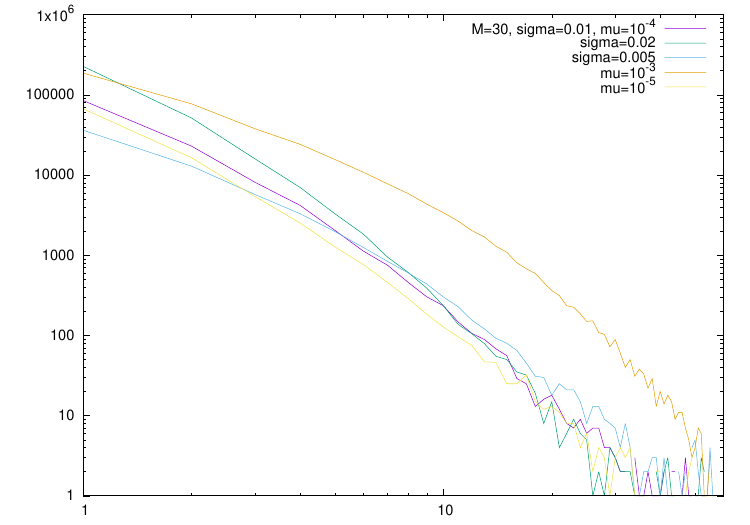}
\includegraphics[width=0.45\textwidth]{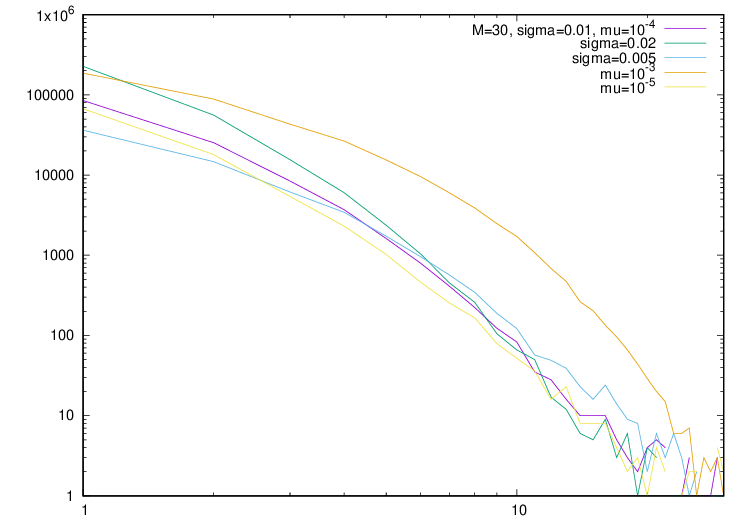}
\caption{Log-scale plot of the frequency of avalanche sizes (left) and the duration of the avalanches (right).}
\label{fig_aval_2d_l30}
\end{figure}

Next, we examine the dependence of the avalanche statistics on the mesh size.
In Figure~\ref{fig_aval_2d_conv} we display the log-scale plot of the avalanche size and duration for $M=15,30,60$.
We observe that the power-law distributions for decreasing mesh size exhibit a similar shape.
\begin{figure}[!htp]
\center
\includegraphics[width=0.45\textwidth]{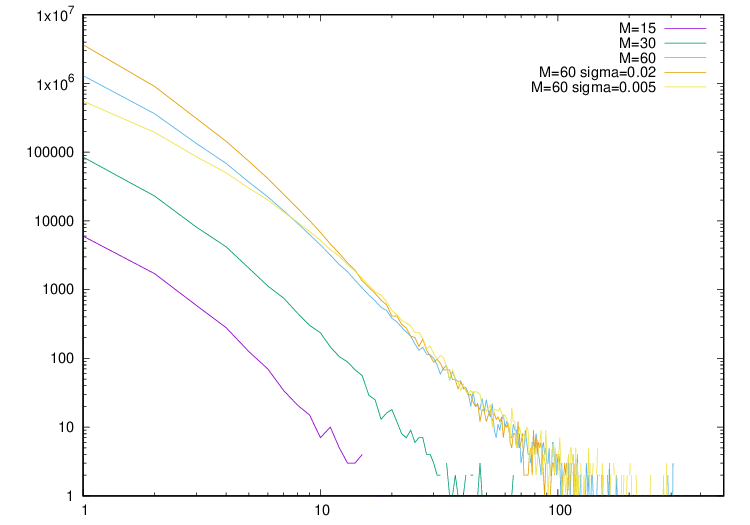}
\includegraphics[width=0.45\textwidth]{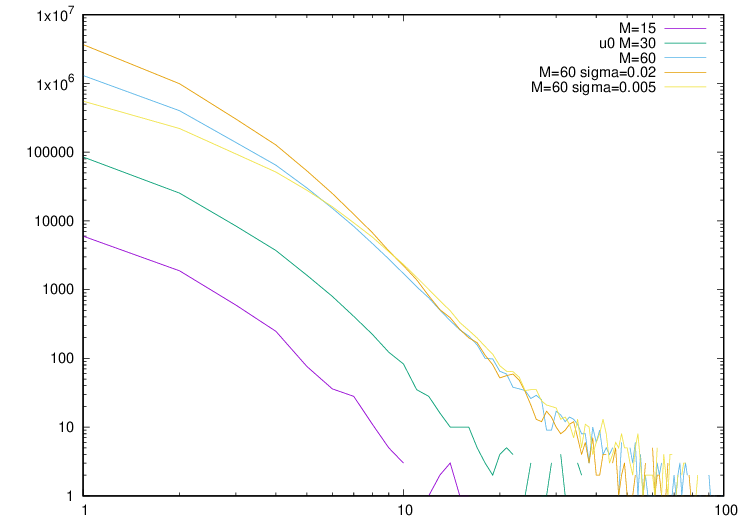}
\caption{Log-scale plot of the the frequency of avalanche sizes (left) and the duration of the avalanches (right) for $M=15,30,60$.}
\label{fig_aval_2d_conv}
\end{figure}

We compare the scaling for the BTW model and the Zhang model in Figure~\ref{fig_aval_2d_btw}. The scaling appears to be qualitatively similar with the difference that the frequency of smaller avalanches is higher in the BTW model.
\begin{figure}[!htp]
\center
\includegraphics[width=0.45\textwidth]{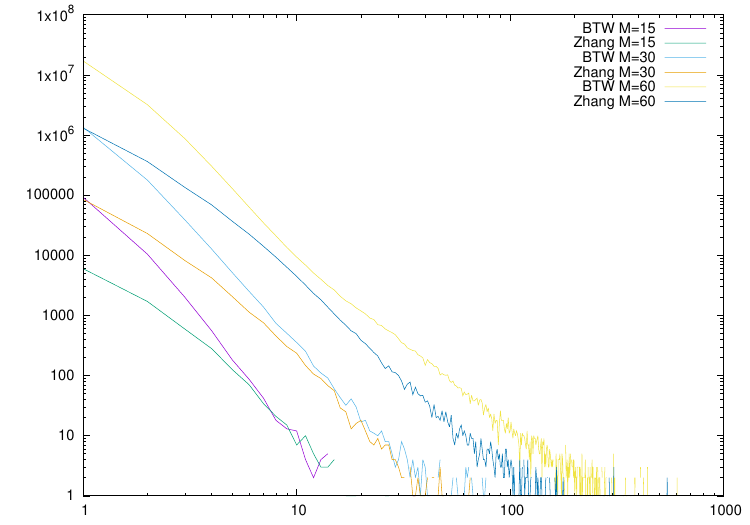}
\includegraphics[width=0.45\textwidth]{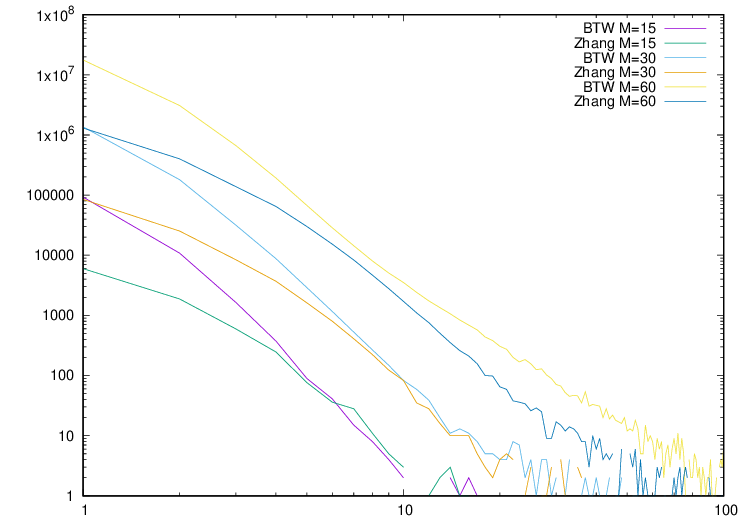}
\caption{Log-scale plot of the the frequency of avalanche sizes (left) and the duration of the avalanches (right): comparison of BTW and Zhang models for $M=15,30,60$.}
\label{fig_aval_2d_btw}
\end{figure}

\subsection{Scaling with totally asymmetric noise}

The original BTW model introduced in \cite{BTW88} enforces a strict
separation of time scales between the random forcing of the system and its
relaxation into subcritical states by diffusion. This is realized by stopping
the forcing during an avalanche until a subcritical state is reached. In
addition, in \cite{BTW88} the noise is totally asymmetric, in the sense
that energy is only added.  In contrast, in the discrete model
(\ref{scheme1}) both dynamics are active simultaneously, and energy is
randomly added or subtracted, with positive average. The relative speed of
driving by random forcing and diffusion can be steered by varying their
relative intensities $D,\mu,\sigma$.

In this section, we compare the power-law scaling of the discrete model in two spatial dimensions with weakly
asymmetric noise (i.e., noise also taking negative values) to the same
model with a totally asymmetric noise instead (i.e., noise which only takes
positive values).  We consider the same parameters as in the previous
section except for $\mu=0$, $\sigma=10^{-3}$ and the random variables
$(\xi^{n,j}_h)$ are chosen to have an i.i.d.~Bernoulli distribution with
values $\{0,1\}$ achieved with equal probability
$\frac{1}{2}$.

In view of \cite{BTW88},
we compare the avalanche statistics of simulations with strict scale
separation to those with simultaneous forcing. In the first regime, the
random forcing is switched off during an avalanche until the system reaches
a subcritical state, while in the latter regime,
the forcing remains active during avalanches.

In Figure~\ref{fig_aval_2d_pos} we display the the log-scale plots of
avalanche sizes and durations. We observe that the avalanche distribution for
the model with simultaneous forcing and diffusion and with larger mesh size
$M=30$ and $\sigma=10^{-3}$ obeys a power law. For larger intensity of the
asymmetric noise $\sigma=4\cdot 10^{-3}$ the power law is no longer
preserved and the distribution is biased towards avalanches with larger
size and duration. A
similar situation occurs in the simulation with smaller mesh size $M=60$,
$\sigma=10^{-3}$. In contrast, enforcing the strict scale separation of
driving force and diffusion as in \cite{BTW88}, the avalanche distribution obeys a power law
even in the case $M=60$.

An explanation for these observations is  that the
  effect of ``overlapping avalanches'' becomes dominant for large noise
  intensities: In systems near criticality, multiple simultaneous
  avalanches may occur, which, due to the global character of the avalanche
  definition, are counted as one large event. As a result, large
  avalanches would be overrepresented in the simulation. This effect is avoided when the strict separation of forcing and relaxation scale is enforced.

\begin{figure}[!htp]
\center
\includegraphics[width=0.45\textwidth]{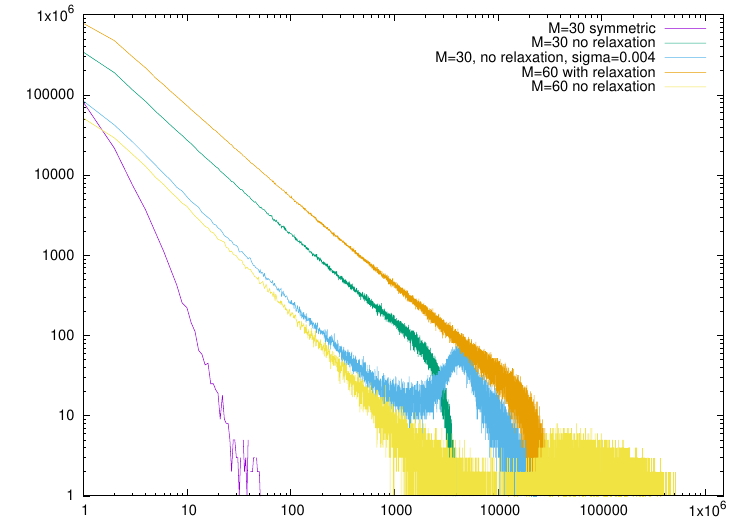}
\includegraphics[width=0.45\textwidth]{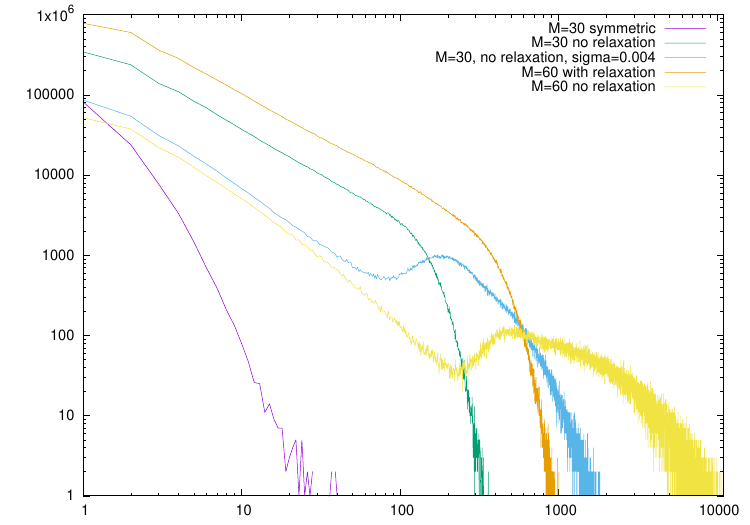}
\caption{Log-scale plot of the frequency of avalanche sizes (left) and the duration of the avalanches (right) for asymmetric noise.}
\label{fig_aval_2d_pos}
\end{figure}

\subsection{Simulations and scaling limits in 2D}\label{sec_num2d}

In this section we investigate the existence of scaling limits of the discrete  dynamics (\ref{scheme1}) in spatial dimension $d=2$, with a focus on the effect of the lower regularity of the limiting space-time white noise compared to one spatial dimension.

The simulation parameters are as follows: $T = 0.015625$, $\sigma=1$, $\mu=0$, $D=0.25$, $K=1$, $\tau=h^2$, $h=1/128$,  the spatial domain is the unit square $(0,1)^2$ and
the initial condition is  taken as $x_h^0 =
\frac{1}{2}\Ind{[0.25,0.75]\times[0.25,0.75]}$.

In Figure~\ref{fig_white_2d} we display the expected value of the discrete model \eqref{scheme1} with Zhang nonlinearity at the final time $T$ averaged over $10^6$ realizations and with mesh size
$h=1/128$. The analogous expected value of \eqref{scheme1} with BTW nonlinearity is displayed in Figure~\ref{fig_white_2d_btw}
(left). For better comparison the color range is restricted to
$[0,0.5]$. There are only minor overshoots of these values due to the error
of the Monte-Carlo approximation of the expected value in the stochastic
case. In Figure~\ref{fig_white_2d} (right), for comparison, we display the expected value of \eqref{scheme1} with $\phi(x)=x$ corresponding to the stochastic heat equation

We observe that the simulation of \eqref{scheme1} with Zhang nonlinearity is very similar to the expectation of the numerical solution of the stochastic heat
equation (i.e., the linear counterpart of \eqref{scheme1}). In contrast, the simulation of \eqref{scheme1} with BTW nonlinearity is closer to the initial condition, that is, the effect of the diffusion is weaker in this case. 

To offer an explanation of this observation, we note that the computed probability of the solution to be supercritical, i.e.,
$ \frac{\mathbb{E} \left[\#\{z_{i,j}:\,\, |X_h^n(z_{i,j})| \ge K\}\right]}{\#\{z_{i,j}\}}$ (not counting the grid points at the boundary)
is above $0.50$, and the probability increases with smaller mesh size due to the scaling of the noise with $h^{-1}$, see Figure~\ref{fig_kplus} where we display the computed evolution 
of the probability for the Zhang model. Since in the supercritical ``regime'', the effect of the diffusion of the Zhang nonlinearity is identical to that of the stochastic heat equation, this could offer an explanation of the observation above. In contrast, the BTW nonlinearity does not equal that of the heat equation even for supercritical values of the solution. Therefore, one expects that,
even for small grid size $h$, the BTW model behaves differently from the stochastic heat equation, which is indeed observed in Figure~\ref{fig_white_2d_btw}. 
\begin{figure}[!htp]
\center
\includegraphics[width=0.48\textwidth]{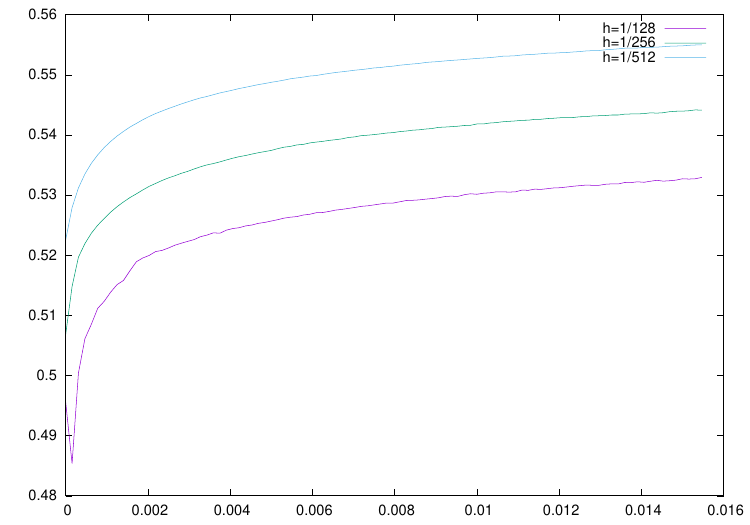}
\caption{Evolution of $\frac{\mathbb{E} \left[\#\{z_{i,j}:\,\, |X_h^n(z_{i,j})| \ge K\}\right]}{\#\{z_{i,j}\}}$
for $h=1/128, 1/256, 1/512$ for the Zhang model.}
\label{fig_kplus}
\end{figure}

To illustrate the effect of the fluctuations and the resulting irregularity of the solution
we display one realization of the solution of the BTW model at the final time in Figure~\ref{fig_white_2d_btw} (right). In contrast to the 1d case (see Figures~\ref{fig_uh_i003} and \ref{fig_uh_i017}) the solution oscillates and exceeds the critical value $K$.

We note that in the deterministic case, for the considered parameters, the solutions of the Zhang and BTW model are equal to the initial condition since $\max_{i,j} |X^0_h(z_{i,j})| \leq \frac{1}{2} < K$.

Repeating the simulations with the initial condition taken to be the (unscaled) indicator function of a square with side $\frac{1}{2}$ placed at the center of the domain
yielded analogous results (not displayed).

\begin{figure}[!htp]
\center
\includegraphics[width=0.48\textwidth]{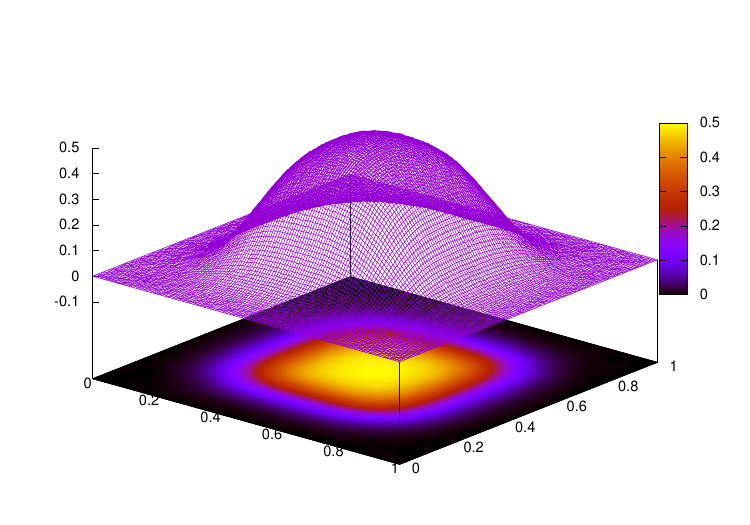}
\includegraphics[width=0.48\textwidth]{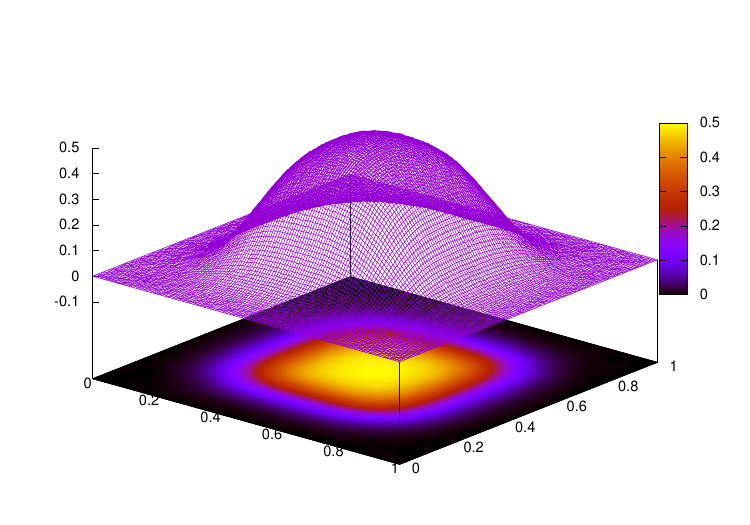}
\caption{Expected value of the numerical solution computed with the Zhang model (left) and the stochastic heat equation (right) at time $t=T$.}
\label{fig_white_2d}
\end{figure}

\begin{figure}[!htp]
\center
\includegraphics[width=0.48\textwidth]{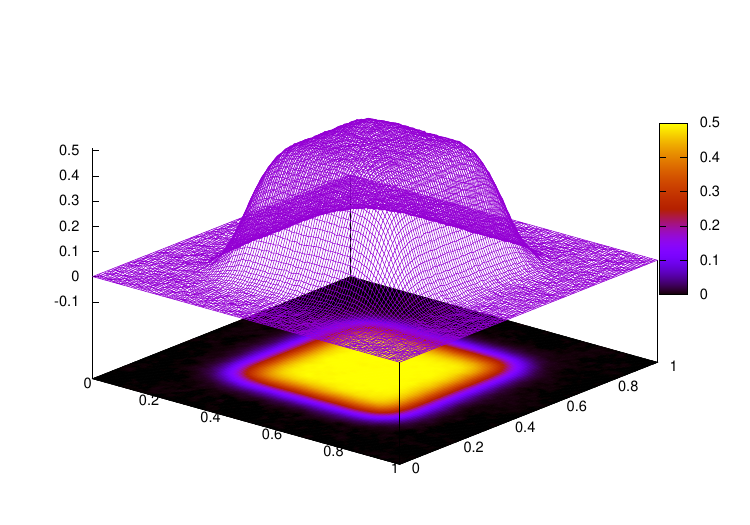}
\includegraphics[width=0.48\textwidth]{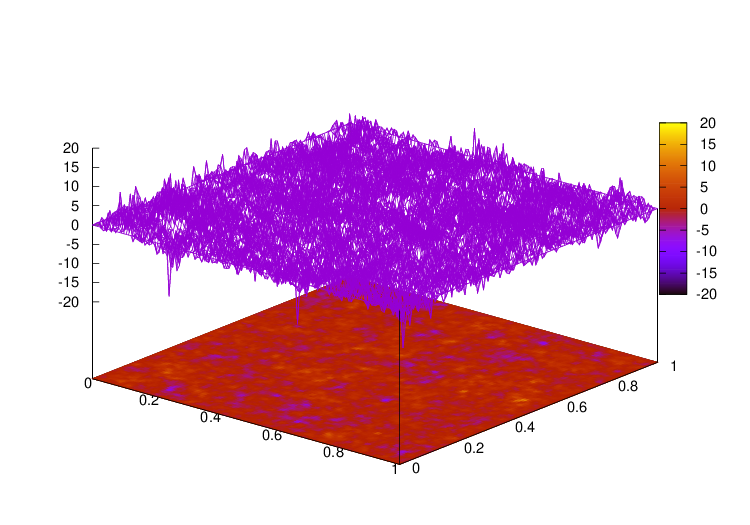}
\caption{Expected value of the numerical solution computed with the BTW model (left) and one realization of the solution of the BTW model (right) at time $t=T$.}
\label{fig_white_2d_btw}
\end{figure}

\appendix

\section{Uniqueness of laws of weak solutions}\label{2sec:uniqueness}

In this section, we prove the following,
using the main result from \cite{KurtzII}.
\begin{thm}\label{2uniqueness-thm}
  The processes $(\tilde X, \tilde W)$ of every weak solution to
  \eqref{2eq:43} have the same law with respect to the Borel
  $\sigma$-algebra of $L^2([0,T]; L^2)\times\cC([0,T]; \Hd)$.
\end{thm}
We first give some preparatory
results and helpful notions.

\begin{Def}
  We define a multivalued operator by its graph $\cA_T\subset L^2([0,T];
  L^2)\times L^2([0,T]; L^2)$, given by
  \begin{equation}\label{2eq:59}
    (f, g)\in \cA_T \quad \text{if and only if}\quad g\in \phi_2(f) \
    \text{for almost every }(t,x)\in[0,T]\times[0,1].
  \end{equation}
\end{Def}

\begin{lem}\label{2cAT-max-mon}
  The operator $\cA_T$ is maximal monotone.
\end{lem}

\begin{proof}
  By \cite[Theorem 2.8]{Barbu}, it is enough to show that $\cA_T$ is the
  subdifferential of a convex, proper and
  lower-semicontinuous functional $\vphi: H\to [0,\infty]$ on a real Banach
  space $H$. To this end, define $\tilde\psi: \R\to[0,\infty)$ by
  \begin{displaymath}
    \tilde\psi(x) = \Ind{\{\abs{x}\geq 1\}} (x^2 - 1),
  \end{displaymath}
  which is proper, convex and continuous, and for which we have
  $\partial \tilde\psi = \phi_2$.  We note that $H:= L^2([0,T]; L^2)$
  is a Hilbert space. Defining
  \begin{equation}\label{2eq:47}
    \vphi_T: H\to [0,\infty], \quad \vphi_T(u) =
    \int_0^T\int_0^1\tilde\psi(u(t,x))\d x \d t, 
  \end{equation}
  we obtain by \cite[Theorem 16.50]{Bauschke} that $\vphi_T$ is convex,
  proper and lower-semicontinuous and $\cA_T = \partial \vphi_T$, as required.
\end{proof}

\begin{lem}\label{2closed-graph-vorne}
  The graph $\cA_T$ is a closed subset of $L^2([0,T];
    L^2)\times L^2([0,T];L^2)$ and thus
  measurable with respect to the Borel $\sigma$-algebra on $L^2([0,T]; L^2)$.
\end{lem}
\begin{proof}
  The first statement is true for any maximal monotone operator by
  \cite[Proposition 2.1]{Barbu}. The measurability then follows by definition
  of the Borel $\sigma$-algebra.
\end{proof}

We define two kinds of Sobolev spaces that we are going to
use.
\begin{Def}\label{2Def:general-Sobolev}
  Let $V\subset H \subset V'$ a Gelfand triple and $T>0$. We define
  \begin{align*}
    W^{1,2}([0,T]; V') &:= \{u\in L^2([0,T]; V'): u'\in L^2([0,T]; V')\}\\
    \text{and}\quad W^{1,2}([0,T]; V, H) &:= \{u\in L^2([0,T], V): u'\in
                                           L^2([0,T]; V')\},
  \end{align*}
  where $u'$ is the weak derivative of $u$ as defined \eg in \cite[Definition
  2.5.1]{Neerven}. These spaces are Banach spaces with the norms
  \begin{align*}
    \norm{u}_{W^{1,2}([0,T]; V')}
    &= \left(\norm{u}_{L^2([0,T]; V')}^2 +
      \norm{u'}_{L^2([0,T]; V')}^2\right)^{\frac1{2}}\\
    \text{and}\quad\norm{u}_{W^{1,2}([0,T]; V,H)}&= \left(\norm{u}_{L^2([0,T]; V)}^2 + \norm{u'}_{L^2([0,T]; V')}^2\right)^{\frac1{2}},
  \end{align*}
  respectively. These norms are norm-equivalent to the ones given in
  \cite[Section 2.5.b]{Neerven} and \cite[Proposition 23.23]{ZeidlerIIA},
  respectively, where also the Banach space property is proved.
\end{Def}

We have the following measurability properties.
\begin{lem}\label{2measurability-lem}
  The subset
  \begin{equation}\label{2eq:135}
    M_1 := \left\{
      \begin{aligned}
        &\qquad \qquad \qquad \quad (u,z) \in L^2([0,T]; L^2) \times L^2([0,T];(L^2)'):\\
        &\exists v\in L^2([0,T]; L^2) \text{ such that } z= \Delta v\ \d
        t\text{-almost everywhere and } (u,v)\in\cA_T
      \end{aligned}
    \right\}
  \end{equation}
  is Borel-measurable. The map $\partial_t: W^{1,2}([0,T]; (L^2)') \to L^2([0,T]; (L^2)')$ is
  continuous and
  \begin{equation}\label{2eq:134}
    M_2:= (\Pi_1, \partial_t(\Pi_2))^{-1} (M_1) \subseteq L^2([0,T]; L^2)\times
    W^{1,2}([0,T]; (L^2)')
  \end{equation}
  is Borel-measurable. The set $M_2$ is also Borel-measurable as a subset
  of $L^2([0,T]; L^2) \times L^2([0,T]; (L^2)')$. Finally, let
  $I_{xw}: L^2([0,T]; L^2) \times \cC([0,T]; \Hd) \hookrightarrow \left(
    L^2([0,T]; (L^2)')\right)^2$ be the canonical continuous
  embedding. Then, the subset
  \begin{displaymath}
    M_3:= \left\{\left(\Pi_1, \Pi_1(I_{xw}) - \Pi_2(I_{xw})\right) - \mu t \in M_2\right\} \subseteq L^2([0,T]; L^2)\times \cC([0,T]; \Hd)
  \end{displaymath}
  is Borel-measurable, where we write $\mu t \in L^2([0,T]; (L^2)')$ for
  the canonically embedded $L^2([0,T]; L^2)$ function given by
  \begin{displaymath}
    (t,x) \mapsto \mu t.
  \end{displaymath}
\end{lem}

\begin{proof}
  We notice that $M_1$ is the image of the set $\cA_T$, which is
  Borel-measurable by Lemma \ref{2closed-graph-vorne}, under the isometry
  $(\Pi_1, \Delta \circ\Pi_2)$, and hence Borel-measurable by the
  Kuratowski theorem (\cf \cite[Theorem 3.9]{Parthasarathy}). The operator
  $\partial_t: W^{1,2}([0,T]; (L^2)') \to L^2([0,T]; (L^2)')$ is linear and
  bounded by the definition of the Sobolev space. Hence it is continuous,
  which implies Borel-measurability. Thus, also $(\Pi_1, \partial_t)$ is
  continuous and Borel-measurable, which yields measurability of $M_2$
  using the measurability of $M_1$. The set $M_2$, viewed as a subset of
  $L^2([0,T]; L^2) \times L^2([0,T]; (L^2)')$, is the image of the
  canonical embedding and thus Borel-measurable by the Kuratowski
  theorem. The Borel-measurability of $M_3$ follows by the continuity of $I_{xw}$.
\end{proof}

The previous lemma alludes that \eqref{2eq:111} and \eqref{2eq:112} are
actually distributional properties, which motivates the following
definition.

\begin{Def}
  We call a probability measure $Q$ on the probability space $L^2([0,T]; L^2)\times
  \cC([0,T]; \Hd)$ endowed with its Borel $\sigma$-algebra a
  pre-solution to \eqref{2eq:43}, if
  \begin{align}
    Q(M_3) = 1,
  \end{align}
  where $M_3$ is defined as in Lemma \ref{2measurability-lem}.
\end{Def}

\begin{lem}\label{2weak-implies-pre}
  The joint law of the processes $(X, W)$ of each weak solution to
  \eqref{2eq:43} in the sense of Definition \ref{2Def-weak-soln} is a
  pre-solution.
\end{lem}

\begin{proof}
  Let $\left( (\Om, \cF, (\cF_t)_{t\in[0,T]}, \P), X,  W\right)$ be a weak solution to \eqref{2eq:43} and
  $Y\in L^2(\tilde\Om\times[0,T]; L^2)$ the corresponding drift process as
  in Definition \ref{2Def-weak-soln}. Then, \cite[Proposition 2.5.9]{Neerven},
  \eqref{2eq:111} and \eqref{2eq:112} yield
  \begin{displaymath}
    \partial_t \left(\Pi_1(I_{xw}(X, W)) - \Pi_2(I_{xw}(X,W)) - \mu t\right) =
    \Delta Y\quad\P\text{-almost surely}
  \end{displaymath}
  with $(X,Y) \in \cA_T$.
  Hence, using the notation from Lemma
  \ref{2measurability-lem}, we have
  \begin{displaymath}
    \left(X, \partial_t(\Pi_1(I_{xw}(X, W)) -
        \Pi_2(I_{xw}(X,W)) - \mu t)\right) \in M_1 \quad \P\text{-almost surely},
  \end{displaymath}
  which by construction is equivalent to $(X,W)\in M_3$ $\P$-almost
  surely. This finishes the proof.
\end{proof}

We cite the concept of pointwise uniqueness from \cite[Definition 1.4]{KurtzII}.
\begin{Def}
  Pointwise uniqueness holds for pre-solutions, if and only if for any processes
  $(X_1, X_2, W)$ defined on the same probability space with
  $\cL((X_1, W))$ and $\cL((X_2, W))$ being pre-solutions, $X_1 = X_2$
  almost surely.
\end{Def}

\begin{lem}
  Pointwise uniqueness holds for pre-solutions to \eqref{2eq:43}.
\end{lem}

\begin{proof}
  Let $(X^1, X^2, W)$ be defined on a probability space $(\Om, \cF, \P)$,
  such that $\cL((X^1, W))$ and $\cL((X^2, W))$ are two
  pre-solutions to \eqref{2eq:43}. Let $M_3$ be defined as in Lemma
  \ref{2measurability-lem}, and let
  \begin{displaymath}
    \tilde M_3 := (X^1, W)^{-1}(M_3) \cap (X^2, W)^{-1}(M_3),
  \end{displaymath}
  which implies that $\P(\tilde M_3) = 1$ by construction. From now on, we
  conduct all arguments pointwise for $\omega \in \tilde M_3$. We define
  $Y^i \in L^2([0,T]; L^2)$ for $i= 1,2$ by
  \begin{displaymath}
    Y^i = \Delta^{-1}(\partial_t(X^i - W - \mu t)),
  \end{displaymath}
  which is well-defined due to the construction of
  $M_3$. Moreover, it
  follows that \eqref{2eq:111} and \eqref{2eq:112} are satisfied for
  $(X^i, Y^i, W)$ for $i=1,2$, which implies
  \begin{equation}\label{2eq:136}
    X^1(t) - X^2(t) = \int_0^t\Delta(Y^1(r) - Y^2(r))\d r\quad \text{in }L^2([0,T]; (L^2)').
  \end{equation}
  By \cite[Proposition 2.5.9]{Neerven}, \eqref{2eq:136} implies that
  $X^1 - X^2$ is weakly differentiable with
  $(X^1 - X^2)' = \Delta (Y^1 - Y^2)$. Since $X^i, Y^i \in L^2([0,T]; L^2)$
  for $i= 1,2$ by construction, $X^1- X^2 \in W^{1,2}([0,T]; L^2,
  \Hd)$. \cite[Proposition 23.23]{ZeidlerIIA} then yields that there exists
  a continuous $\Hd$-valued $\d t$-version $Z$ of $X^1 - X^2$, for which we
  have
  \begin{align*}
    \norm{Z(t)}_\Hd^2
    &= \int_0^t\sp{\Delta(Y^1(r) - Y^2(r))}{X^1(r) -
      X^2(r)}_{(L^2)'\times L^2}\d r\\
    &= -\int_0^t\sp{Y^1(r) - Y^2(r)}{X^1(r) - X^2(r)}_{L^2}\d r \leq 0,
  \end{align*}
  where the last step follows from \eqref{2eq:112}. This implies that $Z$
  and consequently $X^1- X^2$ is zero $\d t$-almost everywhere. Since this
  is true for every $\omega\in \tilde M_3$, $X^1= X^2$ $\P$-almost surely,
  as required.
\end{proof}

\begin{Cor}\label{2uniqueness-pre}
  There exists at most one pre-solution to \eqref{2eq:43}.
\end{Cor}

\begin{proof}
  This is part of the statement of \cite[Theorem 1.5]{KurtzII}.
\end{proof}

\begin{proof}[Proof of Theorem \ref{2uniqueness-thm}]
  The claim is a direct consequence of Lemma \ref{2weak-implies-pre} and
  Corollary \ref{2uniqueness-pre}.
\end{proof}

\section{Properties of discrete spaces and prolongations}
We state some properties on the discrete spaces used above and their
interplay to the corresponding function spaces by via prolongations. For
the sake of brevity, we omit details which can be considered classical.

\begin{lem}\label{2lem-spectralnorm}
  Let $\Delta_h\in \R^{(Z-1)\times(Z-1)}$ be defined as in \eqref{2eq:Deltah}. Then, $-\Delta_h$
  is positive definite and
  \begin{displaymath}
    \norm{-\Delta_h} \leq \frac{4}{h^2}.
  \end{displaymath}
\end{lem}
\begin{proof}
  From \cite[Equation (2.23)]{Leveque}, we obtain that the eigenvalues of
  $-\Delta_h$ are
  \begin{displaymath}
    \lambda_j = \frac{2}{h^2}(1 - \cos(j\pi h)) \in \left(0, \frac{4}{h^2}\right), \quad j=1,\dots, Z-1,
  \end{displaymath}
  which implies that $-\Delta_h$ is positive definite.
  Equation (2.77) in \cite{Schwarz} then yields the second claim. 
\end{proof}

\begin{Cor}
  For $u\in\R^{Z-1}$, Lemma \ref{2lem-spectralnorm} yields
  \begin{align}\label{2eq:normest}
    \begin{split}
      \norm{\Delta_h u}_{-1}^2 &= \abs{\sp{-\Delta_h u}{u}_0} = h
      \abs{\sp{-\Delta_h u}{u}}\\
      &\leq h \norm{-\Delta_h u}\norm{u} \leq
      \norm{- \Delta_h} h\norm{u}^2 \leq \frac{4}{h^2} \norm{u}_0^2.
    \end{split}
  \end{align}
\end{Cor}

\begin{lem}\label{2trace-inv-laplace}
  Let $h>0$ as in Assumption \ref{2ass-tau-h}, and let $I': L^2 \to \Hd$ be the
  canonical embedding. Then, $I'\in L_2(L^2, \Hd)$ and
  \begin{equation}\label{eq:Tr-inv-laplace}
    \lim_{h\to 0} \Tr(-\Delta_h^{-1})  = \sum_{k=1}^\infty \frac1{\pi^2k^2}
    = \frac1{6} =
    \norm{I'}_{L_2(L^2,\Hd)}^2.
  \end{equation}
\end{lem}

Recall the partitions $(K_i)_{i=0}^Z$ and $(J_i)_{i=0}^{Z-1}$ and the grids
$(x_i)_{i=0}^Z$ and $(y_i)_{i=0}^{Z-1}$ as given in
\eqref{2eq:partition}, and the definition of prolongations of functions on
these grids as given in Definition \ref{2Def:prolong} and Definition
\ref{2Def-prolong-time}. Then, the following statements can be verified by
direct computations.

\begin{lem}\label{2prop-gridfns}
  Let $u = (u_i)_{i=1}^{Z-1} \in \R^{Z-1}$ and
  $v = (v_i)_{i=0}^{Z-1} \in \R^Z$ and recall the convention
  $u_0 = u_Z = 0$. Define the piecewise linear prolongation with zero-Neumann
  boundary conditions with respect to the grid $(y_i)_{i=0,\dots,Z-1}$ by
  \begin{displaymath}
    \Iply: \R^\gridlim \to H^1, v\mapsto v_0\Ind{K_0} + \sum_{i=1}^{\gridlim-1} \left[v_{i-1} + \frac{v_i - v_{i-1}}{h} (\cdot -
      x_i)\right]\Ind{K_i} + v_{\gridlim-1}\Ind{K_\gridlim}.
  \end{displaymath}
  and the piecewise constant prolongation by $\Ipcy: \R^\gridlim \to L^2,
  v\mapsto  \sum_{i=0}^{\gridlim-1} v_i\Ind{J_i}$. Then,
  \allowdisplaybreaks
  \begin{align*}
  &\norm{\Iply v}_{L^2} \leq \norm{\Ipcy
      v}_{L^2} \leq 3 \norm{\Iply v}_{L^2},\\
    &\int_0^1 \Iply v\ \d x = \int_0^1 \Ipcy v\ \d x,\\
      &\Iply v - a = \Iply (v - a) \quad \text{and} \quad \Ipcy
      v -a = \Ipcy (v-a) \quad\text{for all }a\in\R,\\
      &\partial_x \Ipcy v = \sum_{i=1}^{Z-1}\delta_{x_i}(v_i - v_{i-1}),\\
      &\partial_x \left(\Iply \left(\sum_{j=0}^i h u_j\right)_{i=0}^{Z-1}\right) = \Ipcx u,\\
    &\norm{u}_1 = \norm{\Iplx u}_\Hsob,\\
      &\norm{\partial_{xx}(\Iplx u)}_\Hd =
      \norm{\sum_{i=1}^{Z-1}\frac1{h}(-u_{i-1} + 2u_i -
        u_{i+1})\delta_{x_i}}_\Hd.
    \end{align*}
\end{lem}

\begin{lem}\label{2est-H-1}
  Let $u = (u_i)_{i=1}^{Z-1} \in \R^{Z-1}$, where $Z$ is defined as in
  \eqref{2eq:partition}. Then
  \begin{displaymath}
    \norm{\Ipcx u}_\Hd \leq \norm{u}_{-1} \leq 3 \norm{\Ipcx u}_\Hd.
  \end{displaymath}
\end{lem}
\begin{proof}
  Let $v = (v_i)_{i=0}^{Z-1} \in \R^Z$ be defined by $v_i = \sum_{j=0}^i h u_j$.
  Then, using the convention
  $\left(\Delta_h^{-1}u\right)_0 = \left(\Delta_h^{-1}u\right)_Z = 0$ and
  Lemma \ref{2prop-gridfns}, we
  have
  \allowdisplaybreaks
  \begin{align*}
    &\norm{\Ipcx u}_\Hd
    = \norm{\Iply v - \int_0^1\Iply v\ \d x}_{L^2}
    = \norm{\Iply \left(v - \int_0^1\Iply v\ \d x\right)}_{L^2}\\
    &\leq \norm{\Ipcy \left(v - \int_0^1\Ipcy v\ \d x\right)}_{L^2}
    = \norm{\partial_x (\Ipcy v)}_\Hd
    = \norm{\sum_{i=1}^{Z-1}\delta_{x_i} h u_i}_\Hd\\
    &= \norm{\sum_{i=1}^{Z-1}\delta_{x_i} h (\Delta_h\Delta_h^{-1}u)_i}_\Hd
      = \norm{\sum_{i=1}^{Z-1}\delta_{x_i}\frac1{h} \left(-\left(\Delta_h^{-1}u\right)_{i-1} +2
      \left(\Delta_h^{-1}u\right)_i - \left(\Delta_h^ {-1}u\right)_{i+1}\right)}_\Hd\\
    &= \norm{\partial_{xx}(\Iplx \Delta_h^{-1}u)}_\Hd
    = \norm{\Iplx \Delta_h^{-1}u}_\Hsob 
    = \norm{\Delta_h^{-1}u}_1
    = \norm{u}_{-1},
  \end{align*}
  which yields the first inequality. The same calculation yields the second
  inequality if we start with $3\norm{\Ipcx u}_\Hd$
  and replace the third step by
  \begin{displaymath}
    3 \norm{\Iply \left(v - \int_0^1\Iply v\ \d x\right)}_{L^2}
    \geq \norm{\Ipcy \left(v - \int_0^1\Ipcy v\ \d x\right)}_{L^2}.
  \end{displaymath}
\end{proof}

\begin{prop}\label{2conv-pc}
  Let $h>0$ denote a sequence converging to $0$, $u\in L^2([0,T]; \Hd)$,
  $\eta\in L^2([0,T]; L^2)$, and for all $h$ in this sequence, $t\in[0,T]$,
  let $u_h(t), \eta_h(t)\in \R^{\gridlim -1}$
  such that $\pcx{u_h} \in L^2([0,T]; \Hd)$ with $\pcx{u_h} \tow u$ in
  $L^2([0,T]; \Hd)$ and $\pcx{\eta_h}\in L^2([0,T]; L^2)$ with
  $\pcx{\eta_h}\to \eta$ in $L^2([0,T]; L^2)$. Then, for $h\to 0$,
  \begin{displaymath}
    \int_0^T  \sp{\eta_h(t)}{u_h(t)}_{-1} \d t \to \int_0^T\sp{\eta(t)}{u(t)}_\Hd \d t.
  \end{displaymath}
\end{prop}

\begin{proof}
  After deriving a Poincaré inequality for the discrete norms, \ie
  $\norm{v}_0^2 \leq C\norm{v}_1^2$ for $v\in\R^{Z-1}$, the uniform bound of
  $\pcx{u_h}$ in $L^2([0,T]; \Hd)$ implies a uniform bound of
  $\pcx{(-\Delta_h^{-1}u_h)}$ in $L^2([0,T]; L^2)$. This allows to extract
  a weakly converging nonrelabeled subsequence with limit $f\in L^2([0,T];
  L^2)$. In order to show that $f = -\Delta u$, it is key to realize that 
  \begin{equation}
    \xi(t,x)\pcx{\left(-\Delta_hu_h(t)\right)}(x) = -\xi(t,x)
    \left(D_h^-D_h^+ \pcx{u_h}(t)\right)(x)\quad \text{for almost
      all }t\in[0,T], x\in[0,1]
  \end{equation}
  for a test function $\xi\in\cC_c^\infty([0,T]\times[0,1])$, where
  $D_h^\pm$ are the $h$-difference quotients to the left resp.\ right.
  Hence, conducting a discrete
  integration by parts and considering that $\xi$ has compact support,
  we compute
  \begin{align}\label{2eq:105}
    \begin{split}
      &\int_0^T \sprod{u}{\xi}_{\Hd\times\Hsob}\d t
      = \lim_{h\to 0} \int_0^T\sprod{\pcx{\bar u_h}}{\xi}_{L^2}\d t
      = \lim_{h\to 0} \int_0^T\sprod{\pcx{\left(\Delta_h \Delta_h^{-1}u_h\right)}}{\ \xi}_{L^2}\d t\\
      &= \lim_{h\to 0} \int_0^T \sprod{-D_h^-D_h^+ \pcx{\left(-\Delta_h^{-1}u_h\right)}}{\ \xi}_{L^2}\d t\\
      &= \lim_{h\to 0} \int_0^T\sprod{\pcx{\left(-\Delta_h^{-1}u_h\right)}}{-D_h^-D_h^+\xi}_{L^2}\d t =
      \int_0^T\sprod{f}{-\Delta\xi}_{L^2}\d t,
    \end{split}
  \end{align}
  using the strong convergence of the second-order difference quotient of
  $\xi$ to
  its second derivative. By a density argument, one concludes that
  $f= -\Delta u$ $\d t$-almost everywhere. The proof can then be finished
  by computing
  \begin{align*}
    \int_0^T\sprod{u_h}{\eta_h}_{-1}\d t
    &= \int_0^T\sprod{-\Delta_h^{-1}u_h}{\eta_h}_0\d t
      = \int_0^T
      \sprod{\pcx{\left(-\Delta_h^{-1}u_h\right)}}{\pc{\eta}}_{L^2}\d t\\
      &\quad \to \int_0^T\sprod{-\Delta^{-1}u}{\eta}_{L^2}\d t =
    \int_0^T\sprod{u}{\eta}_{H^{-1}}\d t.
  \end{align*}
\end{proof}

\begin{lem}\label{lifted-Poincare}
  Recall Definitions \ref{2Def:prolong} and \ref{2def-discr-norms} and let $u\in \R^{Z-1}$. Then
  there exists a constant $C$ independent of $h$ such that
  \begin{displaymath}
    \norm{u}_{-1}^2 \leq C \norm{u}_0^2.
  \end{displaymath}
\end{lem}

\begin{proof}
  Note that since $-\Delta_h$ is symmetric and positive definite, one may
  define its symmetric and positive definite sqare root operator $A_h$,
  which satisfies $A_hA_h = -\Delta_h$, $\norm{u}_{1} = \norm{A_h u}_0$ and $\norm{u}_{-1} = \norm{A_h^{-1}u}_0$. 
  Furthermore, we have a Poincaré inequality for the discrete norms by
  \begin{equation}\label{2eq:5}
    \norm{u_h}_0^2 = \norm{\pcx{u_h}}_{L^2}^2 \leq
    C\norm{\plx{u_h}}_{L^2}^2 \leq
    C\norm{\nabla\plx{u_h}}_{L^2}^2 = C \norm{u_h}_1^2 
  \end{equation}
  for $C$ independent of $h$, where the first inequality can be obtained by connecting
  \cite[Propositions 3.1 and 3.2]{Zuazua}, and the last equality is the sixth
  statement in Lemma \ref{2prop-gridfns}. We then compute
  \begin{displaymath}
    \norm{u}_{-1}^2 = \norm{A_h A_h^{-1} u}_{-1}^2 = \norm{A_h^{-1}u}_0^2
    \leq C \norm{u}_0^2.
  \end{displaymath}
\end{proof}

\endappendix
\section*{Acknowledgements}

We acknowledge support by the Max Planck Society through the Max Planck
Research Group ``Stochastic partial differential equations''.
and the funding by the Deutsche Forschungsgemeinschaft (DFG, German Research Foundation) - SFB 1283/2 2021-317210226.

\pdfbookmark{References}{references}

\bibliographystyle{abbrv}
\bibliography{mybooks}

\begin{thebibliography}{10}

\bibitem{Bak}
P.~Bak.
\newblock {\em How nature works}.
\newblock Springer New York, 1996.

\bibitem{BTW}
P.~Bak, C.~Tang, and K.~Wiesenfeld.
\newblock Self-organized criticality: An explanation of the 1/f noise.
\newblock {\em Phys. Rev. Lett.}, 59:381--384, 1987.

\bibitem{BTW88}
P.~Bak, C.~Tang, and K.~Wiesenfeld.
\newblock Self-organized criticality.
\newblock {\em Phys. Rev. A}, 38:364--374, 1988.

\bibitem{Barbu}
V.~Barbu.
\newblock {\em Nonlinear Differential Equations of Monotone Types in Banach
  Spaces}.
\newblock Springer New York, 2010.

\bibitem{Barbu-SOC-convergence}
V.~Barbu.
\newblock Self-organized criticality and convergence to equilibrium of
  solutions to nonlinear diffusion equations.
\newblock {\em Annual Reviews in Control}, 34(1):52--61, 2010.

\bibitem{Barbu-Cellular}
V.~Barbu.
\newblock Self-organized criticality of cellular automata model; absorbtion in
  finite-time of supercritical region into the critical one.
\newblock {\em Mathematical Methods in the Applied Sciences},
  36(13):1726--1733, 2013.

\bibitem{Barbu17}
V.~Barbu.
\newblock The steepest descent algorithm in wasserstein metric for the sandpile
  model of self-organized criticality.
\newblock {\em SIAM Journal on Control and Optimization}, 55(1):413--428, 2017.

\bibitem{BBDPrR-SOC-via-SPDE}
V.~Barbu, P.~Blanchard, G.~Da~Prato, and M.~R{\"o}ckner.
\newblock Self-organized criticality via stochastic partial differential
  equations.
\newblock {\em Theta Series in Advanced Mathematics, "Potential Theory and
  Stochastic Analysis" in Albac. Aurel Cornea Memorial Volume}, pages 11--19,
  2009.

\bibitem{Barbu-Bogachev}
V.~Barbu, V.~I. Bogachev, G.~Da~Prato, and M.~Röckner.
\newblock Weak solutions to the stochastic porous media equation via
  {Kolmogorov} equations: The degenerate case.
\newblock {\em Journal of Functional Analysis}, 237(1):54--75, 2006.

\bibitem{BDPrR-SPMEandSOC}
V.~Barbu, G.~Da~Prato, and M.~R{\"o}ckner.
\newblock Stochastic porous media equations and self-organized criticality.
\newblock {\em Communications in Mathematical Physics}, 285(3):901--923, Feb
  2009.

\bibitem{BDPrR-SPMEbook}
V.~Barbu, G.~Da~Prato, and M.~R{\"o}ckner.
\newblock {\em Stochastic Porous Media Equations}.
\newblock Lecture Notes in Mathematics. Springer International Publishing,
  2016.

\bibitem{BDPrR-existence-strong}
V.~Barbu, G.~Da~Prato, and M.~Röckner.
\newblock Existence of strong solutions for stochastic porous media equation
  under general monotonicity conditions.
\newblock {\em The Annals of Probability}, 37(2):428--452, 2009.

\bibitem{BDPrR-FTE}
V.~Barbu, G.~Da~Prato, and M.~Röckner.
\newblock Finite time extinction of solutions to fast diffusion equations
  driven by linear multiplicative noise.
\newblock {\em Journal of Mathematical Analysis and Applications}, 389(1):147
  -- 164, 2012.

\bibitem{BM}
V.~Barbu and C.~Marinelli.
\newblock Strong solutions for stochastic porous media equations with jumps.
\newblock {\em Infinite Dimensional Analysis, Quantum Probability and Related
  Topics}, 12(03):413--426, 2009.

\bibitem{BDPrR-existence-nonneg}
V.~Barbu, G.~D. Prato, and M.~Röckner.
\newblock Existence and uniqueness of nonnegative solutions to the stochastic
  porous media equation.
\newblock {\em Indiana University Mathematics Journal}, 57(1):187--211, 2008.

\bibitem{BR-SPMEandSOCdims}
V.~Barbu and M.~R{\"o}ckner.
\newblock Stochastic porous media equations and self-organized criticality:
  Convergence to the critical state in all dimensions.
\newblock {\em Communications in Mathematical Physics}, 311(2):539--555, Apr
  2012.

\bibitem{BRR-Rd}
V.~Barbu, M.~Röckner, and F.~Russo.
\newblock Stochastic porous media equations in ${\R}^d$.
\newblock {\em Journal de Mathématiques Pures et Appliquées},
  103(4):1024--1052, 2015.

\bibitem{Bauschke}
H.~Bauschke and P.~Combettes.
\newblock {\em Convex Analysis and Monotone Operator Theory in Hilbert Spaces}.
\newblock CMS Books in Mathematics. Springer International Publishing, 2017.

\bibitem{spme_bgv}
v.~Ba\v{n}as, B.~Gess, and C.~Vieth.
\newblock Numerical approximation of singular-degenerate parabolic stochastic
  pdes, 2022.
\newblock \url{https://arxiv.org/abs/2012.12150}.

\bibitem{BCRE}
J.~Bouchaud, M.~Cates, J.~R. Prakash, and S.~Edwards.
\newblock A model for the dynamics of sandpile surfaces.
\newblock {\em Journal De Physique I}, 4:1383--1410, 1994.

\bibitem{BFH-incompressible}
D.~Breit, E.~Feireisl, and M.~Hofmanov\'{a}.
\newblock Incompressible limit for compressible fluids with stochastic forcing.
\newblock {\em Arch. Ration. Mech. Anal.}, 222(2):895--926, 2016.

\bibitem{Bantay-Janosi}
P.~Bántay and I.~M. Jánosi.
\newblock Self-organization and anomalous diffusion.
\newblock {\em Physica A}, 185:11--18, 1992.

\bibitem{Cafiero1995}
R.~Cafiero, V.~Loreto, L.~Pietronero, A.~Vespignani, and S.~Zapperi.
\newblock Local rigidity and self-organized criticality for avalanches.
\newblock {\em Europhysics Letters ({EPL})}, 29(2):111--116, jan 1995.

\bibitem{Carlson-modelling}
J.~M. Carlson, J.~T. Chayes, E.~R. Grannan, and G.~H. Swindle.
\newblock Self-organized criticality in sandpiles: nature of the critical
  phenomenon.
\newblock {\em Phys. Rev. A (3)}, 42(4):2467--2470, 1990.

\bibitem{Carlson-Annals-Prob}
J.~M. Carlson, E.~R. Grannan, G.~H. Swindle, and J.~Tour.
\newblock Singular diffusion limits of a class of reversible self-organizing
  particle systems.
\newblock {\em Ann. Probab.}, 21(3):1372--1393, 1993.

\bibitem{Clement}
P.~Clément.
\newblock An introduction to gradient flows in metric spaces, 2009.
\newblock \url{http://www.math.leidenuniv.nl/reports/files/2009-09.pdf}, 16
  April 2021.

\bibitem{DPrR-weaksolns}
G.~Da~Prato and M.~R{\"o}ckner.
\newblock Weak solutions to stochastic porous media equations.
\newblock {\em Journal of Evolution Equations}, 4(2):249--271, May 2004.

\bibitem{DPrRRW}
G.~Da~Prato, M.~Röckner, B.~L. Rozovskii, and F.-Y. Wang.
\newblock Strong solutions of stochastic generalized porous media equations:
  Existence, uniqueness, and ergodicity.
\newblock {\em Communications in Partial Differential Equations},
  31(2):277--291, 2006.

\bibitem{Gess-Gnann-Dareiotis-Gruen}
K.~Dareiotis, B.~Gess, M.~V. Gnann, and G.~Grün.
\newblock Non-negative martingale solutions to the stochastic thin-film
  equation with nonlinear gradient noise, 2020.

\bibitem{Davis}
B.~Davis.
\newblock On the integrability of the martingale square function.
\newblock {\em Israel J. Math.}, 8:187--190, 1970.

\bibitem{DelTeso-fractional}
F.~del Teso.
\newblock Finite difference method for a fractional porous medium equation.
\newblock {\em Calcolo}, 51(4):615--638, 2014.

\bibitem{delTesoI}
F.~del Teso, J.~Endal, and E.~R. Jakobsen.
\newblock Robust numerical methods for nonlocal (and local) equations of porous
  medium type. {P}art {I}: {T}heory.
\newblock {\em SIAM J. Numer. Anal.}, 57(5):2266--2299, 2019.

\bibitem{Dhar-Ramaswamy}
D.~Dhar and R.~Ramaswamy.
\newblock Exactly solved model of self-organized critical phenomena.
\newblock {\em Phys. Rev. Lett.}, 63:1659--1662, Oct 1989.

\bibitem{Diaz-G-PhysRevA}
A.~D\'{\i}az-Guilera.
\newblock Noise and dynamics of self-organized critical phenomena.
\newblock {\em Phys. Rev. A}, 45:8551--8558, Jun 1992.

\bibitem{Hoff-DiBenedetto}
E.~DiBenedetto and D.~Hoff.
\newblock An interface tracking algorithm for the porous medium equation.
\newblock {\em Trans. Amer. Math. Soc.}, 284(2):463--500, 1984.

\bibitem{Dickman2000}
R.~Dickman, M.~A. Mu{\~n}oz, A.~Vespignani, and S.~Zapperi.
\newblock {Paths to self-organized criticality}.
\newblock {\em {Brazilian Journal of Physics}}, 30:27 -- 41, 03 2000.

\bibitem{Dickman-Vespignani-Zapperi}
R.~Dickman, A.~Vespignani, and S.~Zapperi.
\newblock Self-organized criticality as an absorbing-state phase transition.
\newblock {\em Phys. Rev. E}, 57:5095--5105, May 1998.

\bibitem{Diaz-G}
A.~Díaz-Guilera.
\newblock Dynamic renormalization group approach to self-organized critical
  phenomena.
\newblock {\em EPL (Europhysics Letters)}, 26:177 -- 182, 1994.

\bibitem{Elstrodt}
J.~Elstrodt.
\newblock {\em Ma{\ss}- und Integrationstheorie}.
\newblock Springer Berlin Heidelberg, 2004.

\bibitem{Erickson}
R.~V. Erickson.
\newblock Lipschitz smoothness and convergence with applications to the central
  limit theorem for summation processes.
\newblock {\em Ann. Probab.}, 9(5):831--851, 1981.

\bibitem{Zuazua}
S.~Ervedoza and E.~Zuazua.
\newblock {\em Numerical Approximation of Exact Controls for Waves}.
\newblock SpringerBriefs in Mathematics. Springer New York, 2013.

\bibitem{Evje-Karlsen-implicit}
S.~Evje and K.~H. Karlsen.
\newblock Degenerate convection-diffusion equations and implicit monotone
  difference schemes.
\newblock In {\em Hyperbolic problems: theory, numerics, applications, {V}ol.
  {I} ({Z}\"{u}rich, 1998)}, volume 129 of {\em Internat. Ser. Numer. Math.},
  pages 285--294. Birkh\"{a}user, Basel, 1999.

\bibitem{Evje-Karlsen-explicit}
S.~Evje and K.~H. Karlsen.
\newblock Monotone difference approximations of {BV} solutions to degenerate
  convection-diffusion equations.
\newblock {\em SIAM J. Numer. Anal.}, 37(6):1838--1860, 2000.

\bibitem{Feder-Feder-Stickslip}
H.~J.~S. Feder and J.~Feder.
\newblock Self-organized criticality in a stick-slip process.
\newblock {\em Phys. Rev. Lett.}, 66:2669--2672, May 1991.

\bibitem{Fischer-Gruen}
J.~Fischer and G.~Gr\"{u}n.
\newblock Existence of positive solutions to stochastic thin-film equations.
\newblock {\em SIAM J. Math. Anal.}, 50(1):411--455, 2018.

\bibitem{Flandoli-Gatarek}
F.~Flandoli and D.~Gatarek.
\newblock Martingale and stationary solutions for stochastic {N}avier-{S}tokes
  equations.
\newblock {\em Probab. Theory and Relat. Fields}, 102:367--391, 1995.

\bibitem{Fonseca}
I.~Fonseca and G.~Leoni.
\newblock {\em Modern Methods in the Calculus of Variations: L\^{}p Spaces}.
\newblock Springer Monographs in Mathematics. Springer New York, 2007.

\bibitem{ricepile}
V.~Frette, K.~Christensen, A.~Malthe-Sørenssen, J.~Feder, T.~Jøssang, and
  P.~Meakin.
\newblock Avalanche dynamics in a pile of rice.
\newblock {\em Nature}, 379:49--52, 1996.

\bibitem{Garrido-Lebowitz-Maes-Spohn}
P.~L. Garrido, J.~L. Lebowitz, C.~Maes, and H.~Spohn.
\newblock Long-range correlations for conservative dynamics.
\newblock {\em Phys. Rev. A}, 42:1954--1968, Aug 1990.

\bibitem{Gerencser-Gyongy-FD}
M.~Gerencs\'{e}r and I.~Gy\"{o}ngy.
\newblock Finite difference schemes for stochastic partial differential
  equations in {S}obolev spaces.
\newblock {\em Appl. Math. Optim.}, 72(1):77--100, 2015.

\bibitem{Gess-FTE}
B.~Gess.
\newblock Finite time extinction for stochastic sign fast diffusion and
  self-organized criticality.
\newblock {\em Communications in Mathematical Physics}, 335(1):309--344, Apr
  2015.

\bibitem{Gess-Gnann}
B.~{Gess} and M.~V. {Gnann}.
\newblock {The stochastic thin-film equation: existence of nonnegative
  martingale solutions}.
\newblock {\em arXiv e-prints}, page arXiv:1904.08951, Apr. 2019.

\bibitem{G-R}
B.~Gess and M.~Röckner.
\newblock Singular-degenerate multivalued stochastic fast diffusion equations.
\newblock {\em SIAM J. Math. Anal.}, 47:4059 -- 4090, 2015.

\bibitem{Grillmeier-Gruen}
H.~Grillmeier and G.~Gr\"{u}n.
\newblock Nonnegativity preserving convergent schemes for stochastic
  porous-medium equations.
\newblock {\em Math. Comp.}, 88(317):1021--1059, 2019.

\bibitem{Grinstein-Lee-Sachdev}
G.~Grinstein, D.-H. Lee, and S.~Sachdev.
\newblock Conservation laws, anisotropy, and ``self-organized criticality'' in
  noisy nonequilibrium systems.
\newblock {\em Phys. Rev. Lett.}, 64:1927--1930, Apr 1990.

\bibitem{Gyongy-FD}
I.~Gy{\"{o}}ngy.
\newblock On stochastic finite difference schemes.
\newblock {\em Stoch. Partial Differ. Equ. Anal. Comput.}, 2(4):539--583, 2014.

\bibitem{Hergarten-Neugebauer}
S.~Hergarten and H.~J. Neugebauer.
\newblock Self-organized criticality in a landslide model.
\newblock {\em Geophysical Research Letters}, 25(6):801--804, 1998.

\bibitem{Hwa-Kardar}
T.~Hwa and M.~Kardar.
\newblock Dissipative transport in open systems: An investigation of
  self-organized criticality.
\newblock {\em Phys. Rev. Lett.}, 62:1813--1816, Apr 1989.

\bibitem{Neerven}
T.~Hyt{\"o}nen, J.~van Neerven, M.~Veraar, and L.~Weis.
\newblock {\em Analysis in Banach Spaces : Volume I: Martingales and
  Littlewood-Paley Theory}.
\newblock Springer International Publishing, 2016.

\bibitem{Jakubowski}
A.~Jakubowski.
\newblock The almost sure {S}korokhod representation for subsequences in
  nonmetric spaces.
\newblock {\em Teor. Veroyatnost. i Primenen.}, 42(1):209--216, 1997.

\bibitem{BS-cap}
D.~Khoshnevisan and Z.~Shi.
\newblock Brownian sheet and capacity.
\newblock {\em Ann. Probab.}, 27(3):1135--1159, 1999.

\bibitem{Klenke}
A.~Klenke.
\newblock {\em Probability Theory: A Comprehensive Course}.
\newblock Universitext. Springer London, 2013.

\bibitem{KurtzII}
T.~G. Kurtz.
\newblock Weak and strong solutions of general stochastic models.
\newblock {\em Electron. Commun. Probab.}, 19:no. 58, 16, 2014.

\bibitem{Leveque}
R.~LeVeque.
\newblock {\em Finite Difference Methods for Ordinary and Partial Differential
  Equations: Steady-State and Time-Dependent Problems}.
\newblock Other Titles in Applied Mathematics. Society for Industrial and
  Applied Mathematics, 2007.

\bibitem{Lions-Magenes}
J.-L. Lions and E.~Magenes.
\newblock {\em Non-homogeneous boundary value problems and applications. {V}ol.
  {I}}.
\newblock Springer-Verlag, New York-Heidelberg, 1972.
\newblock Translated from the French by P. Kenneth, Die Grundlehren der
  mathematischen Wissenschaften, Band 181.

\bibitem{Liu-Stephan}
W.~Liu and M.~Stephan.
\newblock Yosida approximations for multivalued stochastic partial differential
  equations driven by {L}\'{e}vy noise on a {G}elfand triple.
\newblock {\em J. Math. Anal. Appl.}, 410(1):158--178, 2014.

\bibitem{Manna_1991}
S.~S. Manna.
\newblock Two-state model of self-organized criticality.
\newblock {\em Journal of Physics A: Mathematical and General},
  24(7):L363--L369, apr 1991.

\bibitem{McDonald}
S.~McDonald.
\newblock Finite difference approximation for linear stochastic partial
  differential equations with method of lines.
\newblock {\em Munich Personal RePEc Archive}, 2006.

\bibitem{Carlson-Montakhab}
A.~Montakhab and J.~M. Carlson.
\newblock Avalanches, transport, and local equilibrium in self-organized
  criticality.
\newblock {\em Phys. Rev. E}, 58:5608--5619, Nov 1998.

\bibitem{Neuss-SVI-SD}
M.~Neuß.
\newblock Well-posedness of {SVI} solutions to singular-degenerate stochastic
  porous media equations arising in self-organised criticality.
\newblock {\em Stochastics and Dynamics}, 2020.

\bibitem{OFC}
Z.~Olami, H.~J.~S. Feder, and K.~Christensen.
\newblock Self-organized criticality in a continuous, nonconservative cellular
  automaton modeling earthquakes.
\newblock {\em Phys. Rev. Lett.}, 68:1244--1247, Feb 1992.

\bibitem{opw}
M.~Ondrej\'{a}t, A.~Prohl, and N.~Walkington.
\newblock Numerical approximation of nonlinear {SPDE}'s.
\newblock {\em Stoch. Partial Differ. Equ. Anal. Comput.}, 2022.
\newblock \url{https://doi.org/10.1007/s40072-022-00271-9}.

\bibitem{Parthasarathy}
K.~Parthasarathy.
\newblock {\em Probability Measures on Metric Spaces}.
\newblock Probability and Mathematical Statistics - Academic Press. Academic
  Press, 1967.

\bibitem{Pegden-Smart}
W.~Pegden and C.~K. Smart.
\newblock Convergence of the {A}belian sandpile.
\newblock {\em Duke Math. J.}, 162(4):627--642, 2013.

\bibitem{Arenas}
C.~J. {P{\'e}rez}, {\'A}.~{Corral}, A.~{D{\'\i}az-Guilera}, K.~{Christensen},
  and A.~{Arenas}.
\newblock {On Self-Organized Criticality and Synchronization in Lattice Models
  of Coupled Dynamical Systems}.
\newblock {\em International Journal of Modern Physics B}, 10:1111--1151, 1996.

\bibitem{Andrade-Pires}
R.~S. Pires, A.~A. Moreira, H.~A. Carmona, and J.~S. Andrade.
\newblock Singular diffusion in a confined sandpile.
\newblock {\em {EPL} (Europhysics Letters)}, 109(1):14007, jan 2015.

\bibitem{Prigozhin94}
L.~Prigozhin.
\newblock Sandpiles and river networks: Extended systems with nonlocal
  interactions.
\newblock {\em Phys. Rev. E}, 49:1161--1167, Feb 1994.

\bibitem{Prigozhin}
L.~Prigozhin.
\newblock Variational model of sandpile growth.
\newblock {\em European J. Appl. Math.}, 7(3):225--235, 1996.

\bibitem{Roeckner}
C.~Prévot and M.~Röckner.
\newblock {\em A Concise Course on Stochastic Partial Differential Equations}.
\newblock Springer Berlin/Heidelberg, 2007.

\bibitem{RRW}
J.~Ren, M.~Röckner, and F.-Y. Wang.
\newblock Stochastic generalized porous media and fast diffusion equations.
\newblock {\em Journal of Differential Equations}, 238(1):118--152, 2007.

\bibitem{RWZ-reflection}
M.~Röckner, F.-Y. Wang, and T.~Zhang.
\newblock Stochastic generalized porous media equations with reflection.
\newblock {\em Stochastic Processes and their Applications},
  123(11):3943--3962, 2013.

\bibitem{Schwarz}
H.~Schwarz and N.~K{\"o}ckler.
\newblock {\em Numerische Mathematik}.
\newblock Vieweg+Teubner Verlag, 2011.

\bibitem{Watkins-Pruessner}
N.~W. Watkins, G.~Pruessner, S.~C. Chapman, N.~B. Crosby, and H.~J. Jensen.
\newblock 25 years of self-organized criticality: Concepts and controversies.
\newblock {\em Space Science Reviews}, 198(1):3--44, Jan 2016.

\bibitem{Webb}
J.~Webb.
\newblock An extension of {G}ronwall's inequality.
\newblock
  \url{http://dspace.nbuv.gov.ua/bitstream/handle/123456789/169293/31-Webb.pdf},
  10 March 2021.

\bibitem{Yamada-Watanabe}
T.~Yamada and S.~Watanabe.
\newblock On the uniqueness of solutions of stochastic differential equations.
\newblock {\em J. Math. Kyoto Univ.}, 11:155--167, 1971.

\bibitem{ZeidlerIIA}
E.~Zeidler and L.~Boron.
\newblock {\em Nonlinear Functional Analysis and Its Applications: II/ A:
  Linear Monotone Operators}.
\newblock Monotone operators / transl. by the author and by Leo F. Boron.
  Springer New York, 1989.

\bibitem{Zhang}
Y.-C. Zhang.
\newblock Scaling theory of self-organized criticality.
\newblock {\em Phys. Rev. Lett.}, 63:470--473, Jul 1989.

\end{thebibliography}

\end{document}